\documentclass[hidelinks, reqno]{amsart}

\usepackage{amsmath}
\usepackage{amsthm}
\usepackage{amssymb}
\usepackage{amsfonts} %Already loaded by amssymb.
\usepackage{amscd, dsfont}
\usepackage{thmtools}
\usepackage{mathtools}
\usepackage{graphicx}
\usepackage[all,arc]{xy}
\usepackage{enumerate}
\usepackage[inline]{enumitem}
\usepackage{array}
\usepackage{multirow}
\usepackage[scr]{rsfso} %Better than mathrsfs. Removed euscript. https://sites.math.washington.edu/~lee/Writing/typesetting-script.pdf 
\usepackage{caption}
\usepackage{subcaption}
\usepackage[pagebackref = true, colorlinks, linkcolor = Red, citecolor = Green, bookmarksdepth=3, unicode, linktocpage=true, hypertexnames=false]{hyperref}
\usepackage[capitalize]{cleveref}
\usepackage{comment}
\usepackage{color}
\usepackage[usenames,dvipsnames]{xcolor}
\usepackage{tikz-cd, tikz}% for drawing graph
\usepackage{pgfplots}
\usetikzlibrary{patterns}
\usetikzlibrary{decorations.markings}
\usepackage{bm} %For bold math (keeping italics). For bold upright, use mathbf.
\usepackage{bbm} %Alternative style to mathbb. https://ctan.math.illinois.edu/macros/latex/contrib/bbm/bbm.pdf
\usepackage[utf8]{inputenc}

\newtheorem{theorem}{Theorem}[section]
\newtheorem{proposition}[theorem]{Proposition}
\newtheorem{lemma}[theorem]{Lemma}
\newtheorem{corollary}[theorem]{Corollary}

\newtheorem{assumption}[theorem]{Assumption}

\theoremstyle{definition}
\newtheorem{definition}[theorem]{Definition}

\theoremstyle{remark}
\newtheorem{remark}[theorem]{Remark}
\makeatletter
\renewcommand\subsubsection{\@startsection{subsubsection}{3}%
\z@{.5\linespacing\@plus.7\linespacing}{-.5em}%
{\normalfont\bfseries\itshape}} %Changed \normalfont\itshape to \normalfont\bfseries\itshape
\makeatother

\DeclareMathOperator{\SO}{SO}
\DeclareMathOperator{\Id}{Id}

\DeclareMathOperator{\GL}{GL}
\DeclareMathOperator{\SL}{SL}

\DeclareMathOperator{\Ad}{Ad}
\DeclareMathOperator{\BP}{BP}

\DeclareMathOperator{\Stab}{Stab}
\DeclareMathOperator{\diam}{diam}
\let\U\relax %Because arXiv's compiler complained
\DeclareMathOperator{\U}{U}

\DeclareMathOperator{\supp}{supp}
\DeclareMathOperator{\Lip}{Lip}

\DeclareMathOperator{\T}{T}
\DeclareMathOperator{\F}{F}
\DeclareMathOperator{\tr}{tr}

\DeclareMathOperator{\Isom}{Isom}
\DeclareMathOperator{\Core}{Core}
\DeclareMathOperator{\Hull}{Hull}
\DeclareMathOperator{\interior}{int}

\DeclareMathOperator{\Cay}{Cay}
\DeclareMathOperator{\vis}{vis}

\let\C\relax %Because arXiv's compiler complained
\newcommand{\C}{\mathbb{C}}
\newcommand{\N}{\mathbb{N}}
\newcommand{\Q}{\mathbb{Q}}
\newcommand{\R}{\mathbb{R}}
\newcommand{\Z}{\mathbb{Z}}
\renewcommand{\H}{\mathbb{H}}
\newcommand{\K}{\mathbb{K}}

\newcommand{\cal}[1]{\mathcal{#1}}

\newcommand{\calA}{\cal A}
\newcommand{\gen}{\Gamma_\calA}

\newcommand{\calO}{\cal O}

\newcommand{\calM}{\cal M}

\newcommand{\calL}{\cal L}

\newcommand{\sg}{\mathrm{sg}}

\newcommand{\LieA}{\mathfrak{a}}

\newcommand{\LieM}{\mathfrak{m}}

\newcommand{\HS}{\mathbb{H}^n}
\newcommand{\Ret}{\mathscr{R}}
\newcommand{\FRet}{\mathscr{F}}
\newcommand{\Hol}{\mathscr{H}}
\newcommand{\GHol}{\mathscr{G}}
\newcommand{\m}{m^{\mathrm{BMS}}}

\newcommand{\Leb}{m^{\mathrm{Leb}}}
\newcommand{\jj}{j_0}

\newcommand{\bigcdot}{\boldsymbol{\cdot}}

\DeclareFontFamily{U}{mathb}{\hyphenchar\font45}
\DeclareFontShape{U}{mathb}{m}{n}{
	<5> <6> <7> <8> <9> <10> gen * mathb
	<10.95> mathb10 <12> <14.4> <17.28> <20.74> <24.88> mathb12
}{}
\DeclareSymbolFont{mathb}{U}{mathb}{m}{n}
\DeclareMathSymbol{\bigast}{2}{mathb}{"06}

\Crefname{subsection}{Subsection}{Subsections}

\title[Generalization of Selberg's $3/16$ theorem]{Generalization of Selberg's $3/16$ theorem for geometrically finite thin subgroups of $\SO(n, 1)$}

\author{Pratyush Sarkar}
\address{Simons Laufer Mathematical Sciences Institute (SLMath), 17 Gauss Way, Berkeley, CA 94720}
\curraddr{Department of Mathematics, University of Utah, 155 South 1400 East, Salt Lake City, UT 84112}
\email{p.sarkar@utah.edu}

\date{\today}

\begin{document}

\begin{abstract}
Let $\Gamma$ be a geometrically finite thin subgroup of an arithmetic lattice $\Gamma_0 < G \coloneq \SO(n, 1)$ and consider the congruence covers of $\Gamma \backslash G$. In the breakthrough work of Bourgain--Gamburd--Sarnak, the expansion machinery was used to establish a uniform spectral gap in the setting $(G, \Gamma_0) = (\SL_2(\R), \SL_2(\Z))$ when the critical exponent satisfies $\delta_\Gamma > \frac{1}{2}$. The main applications are affine sieve for $\Gamma$-orbits and uniform resonance-free half-planes for the resolvent of the Laplacian. These results were generalized in subsequent works by Mohammadi--Oh, Oh--Winter, the author, and Edwards--Oh.
Yet, the region $\delta_\Gamma \in \bigl(\frac{1}{2}, n - 2\bigr]$ for $n \geq 3$ remains to be treated when there are cusps. The purpose of this paper is to fill in this gap in the literature. The difficulty lies in working with a countably infinite coding due to the presence of cusps. In particular, we incorporate new tools to prove the Zariski density and full trace field properties of the return trajectory subgroups.
\end{abstract}

\maketitle

\setcounter{tocdepth}{1}
\tableofcontents

\section{Introduction}
\label{sec:Introduction}
Mixing properties of the geodesic and frame flows for hyperbolic manifolds of finite volume, corresponding to lattices in $G = \Isom^+(\H^n) \cong \SO(n, 1)^\circ$, have been studied extensively. For instance, exponential mixing is known by the independent works of Ratner and Moore \cite{Rat87,Moo87} using spectral gap and representation theory. Moreover, exponential mixing with a \emph{uniform rate} over a family of congruence covers can be derived from property $(\tau)$ with respect to the corresponding family of congruence subgroups---such properties were initially obtained from Selberg's $\frac{3}{16}$ theorem \cite{Sel65}, and studied further by Burger--Sarnak \cite{BS91}, Sarnak--Xue \cite{SX91}, Clozel \cite{Clo03}, and Kelmer--Silberman \cite{KS13}, to name a few.

In contrast this paper is concerned with such uniform exponential mixing of the frame flow for thin subgroups. We focus on the class of \emph{geometrically finite} subgroups which contains both the sub-classes of \emph{convex cocompact} subgroups and lattices. Recall that a torsion-free discrete subgroup $\Gamma < G$ and its corresponding hyperbolic manifold are said to be convex cocompact (resp. geometrically finite) if the core---the smallest closed convex subset containing all closed geodesics---(resp. the $1$-neighborhood of the core) is compact (resp. has finite volume). Also recall that the corresponding frame bundle is given by $\Gamma \backslash G$ and the frame flow is given by the right translation action on $\Gamma \backslash G$ by a one-parameter subgroup of semisimple elements $A = \{a_t\}_{t \in \R}$. We endow $\Gamma \backslash G$ with the Bowen--Margulis--Sullivan (BMS) probability measure $m^{\mathrm{BMS}}$; it is frame flow invariant, supported on the non-wandering set, and the unique probability measure of maximal entropy.
The entropy coincides with the critical exponent $\delta_\Gamma \in [0, n - 1]$ of $\Gamma$. When $\Gamma$ is a lattice, the BMS measure coincides with the Haar probability measure.

To this end, we assume that the torsion-free discrete subgroup $\Gamma < G$ is Zariski dense, geometrically finite, and contained in an arithmetic lattice. Note that the first condition is necessary because otherwise the frame flow is not even ergodic. For a brief survey of mixing and related problems in infinite volume homogeneous dynamics, we refer the reader to the article of Oh \cite[\S 7--8]{Oh26} which will appear in the Proceedings of the ICM 2026. In \cite{BGS11} (see also \cite{BGS10}), Bourgain--Gamburd--Sarnak treated the surface setting, i.e., $n = 2$, where they obtained an $L^2$ spectral gap for the Laplace--Beltrami operator for the congruence family---using superstrong approximation---provided $\delta_\Gamma > \frac{1}{2}$. Such spectral gaps can be converted to uniform exponential mixing results by the work of Mohammadi--Oh \cite{MO15}. For $\delta_\Gamma \in (0, \frac{1}{2}]$, it is well-known that $\Gamma$ is convex cocompact and Oh--Winter \cite{OW16} proved uniform exponential mixing for the congruence family. See \cite{MOW19} for related results for Schottky and continued fractions semigroups. Thus, the picture is complete in the $n = 2$ case.

When $n \geq 3$, there are two distinct flows to study: the geodesic flow and the frame flow. In \cite{MO15}, Mohammadi--Oh also proved uniform exponential mixing of frame flows for the congruence family provided
\begin{align*}
\delta_\Gamma > \max\left\{\frac{n - 1}{2}, n - 2\right\};
\end{align*}
note that $\frac{n - 1}{2} = n - 2$ for $n = 3$. Similarly, uniform exponential mixing of \emph{geodesic} flows for the congruence family holds provided $\delta_\Gamma > \frac{n - 1}{2}$ due to Edwards--Oh \cite{EO21}. In \cite{Sar22a} (see also \cite{Sar22b}), the author proved uniform exponential mixing of frame flows for congruence families (over arbitrary totally real number fields) when $\Gamma$ is convex cocompact. See \cite{Sar25} for the generalization of \cite{MOW19}. Unlike the $n = 2$ case, the picture is not yet complete---in higher dimensions, we have a gap for the critical exponent, namely
\begin{align*}
\delta_\Gamma \in \left(\frac{1}{2}, n - 2\right],
\end{align*}
for which the desired result is yet to be established in full generality for $\Gamma$ geometrically finite with parabolic elements. In this paper, we fill in this gap in the literature using various techniques that has been developed previously. We also show how to treat congruence families over arbitrary totally real number fields strictly larger than $\Q$ by incorporating new tools. In this case, due to certain limitations in the symbolic coding that we use, we impose a certain assumption regarding trace fields and explain what is required to remove it. For this reason and simplicity in the setup, we first state our main theorem over $\Q$, separately.

\subsection{Main theorem over $\Q$}
Suppose $G = \mathbf{G}(\R)^\circ$ for some linear algebraic group $\mathbf{G}$ defined over $\Q$. Let $\tilde{\pi}: \tilde{\mathbf{G}} \to \mathbf{G}$ be a simply connected cover defined over $\Q$. For all $q \in \Z$, let
\begin{align*}
\pi_q: \tilde{\mathbf{G}}(\Z) \to \tilde{\mathbf{G}}(\Z/q\Z)
\end{align*}
be the reduction map. Let $\Gamma < \mathbf{G}(\Q)$ be a Zariski dense torsion-free geometrically finite subgroup such that $\tilde{\pi}^{-1}(\Gamma)$ is contained in $\tilde{\mathbf{G}}(\Z)$. Then, $\Gamma$ satisfies the strong approximation property. For all $q \in \N$, let $\Gamma_q < \Gamma$ be the congruence subgroup of level $q$, meaning that
\begin{align*}
\Gamma_q = \Gamma \cap \tilde{\pi}(\ker(\pi_q));
\end{align*}
and let $m^{\mathrm{BMS}}_q$ be the Bowen--Margulis--Sullivan measure on $\Gamma_q \backslash G$ induced from the one on $\Gamma \backslash G$. We prove the following theorem when $\Gamma$ has parabolic elements; otherwise, it was obtained in \cite{Sar22a,Sar22b}).

\begin{theorem}
\label{thm:UniformExponentialMixingOverQ}
There exist $\eta > 0$, $C > 0$, and $q_0 \in \Z_{>1}$ such that for all (square-free if $n = 3$) $q \in \N$ coprime to $q_0$, and $\phi, \psi \in C^1(\Gamma_q \backslash G)$, and $t > 0$, we have
\begin{multline*}
\left|\int_{\Gamma_q \backslash G} \phi(xa_t) \psi(x) \, dm_q^{\mathrm{BMS}}(x) - \frac{1}{m^{\mathrm{BMS}}_q(\Gamma_q \backslash G)} m_q^{\mathrm{BMS}}(\phi) \cdot m_q^{\mathrm{BMS}}(\psi)\right| \\
\leq Cq^C e^{-\eta t} \|\phi\|_{C^1} \|\psi\|_{C^1}.
\end{multline*}
\end{theorem}

\begin{remark}
\label{rem:HdS_AbsolutelySimpleHypothesis}
The theorem of He--de Saxc\'{e} \cite[Theorem 6.1]{HdS22} requires that the algebraic group be absolutely simple but $\SO_{3, 1}$ is not. In this case, we are forced to use the theorem of Golsefidy--Varj\'{u} \cite{GV12} and hence require the square-free assumption for $n = 3$. As soon as one proves \cite[Theorem 6.1]{HdS22} for \emph{semisimple} algebraic groups, the assumption for $n = 3$ can also be removed.
\end{remark}

\subsection{Main theorem over other number fields}
Suppose $G = \mathbf{G}(\R)^\circ$ for some linear algebraic group $\mathbf{G}$ defined over a totally real number field $\K$ such that $\mathbf{G}(\R) \cong \SO(n, 1)$ and $\mathbf{G}^\sigma(\R)$ is compact for all nontrivial embeddings $\sigma: \K \hookrightarrow \R$. Let $\tilde{\pi}: \tilde{\mathbf{G}} \to \mathbf{G}$ be a simply connected cover defined over $\K$. For all ideals $\mathfrak{q} \subset \calO_\K$, let
\begin{align*}
\pi_{\mathfrak{q}}: \tilde{\mathbf{G}}(\calO_\K) \to \tilde{\mathbf{G}}(\calO_\K/\mathfrak{q})
\end{align*}
be the reduction map. Let $\Gamma < \mathbf{G}(\K)$ be a Zariski dense torsion-free geometrically finite subgroup such that $\tilde{\pi}^{-1}(\Gamma)$ is contained in $\tilde{\mathbf{G}}(\calO_\K)$ with trace field $\mathbb Q(\tr(\Ad(\tilde{\pi}^{-1}(\Gamma)))) = \K$. We impose these conditions so that $\Gamma$ satisfies the strong approximation property. For all nontrivial ideals $\mathfrak{q} \subset \calO_\K$, let $\Gamma_\mathfrak{q} < \Gamma$ be the congruence subgroup of level $\mathfrak{q}$, meaning that $\Gamma_\mathfrak{q} = \Gamma \cap \tilde{\pi}(\ker(\pi_\mathfrak{q}))$. For all nontrivial ideals $\mathfrak{q} \subset \calO_\K$, let $N_\K(\mathfrak{q})$ be the norm of the ideal $\mathfrak{q}$ and $m^{\mathrm{BMS}}_\mathfrak{q}$ be the Bowen--Margulis--Sullivan measure on $\Gamma_\mathfrak{q} \backslash G$ induced from the one on $\Gamma \backslash G$. We prove the more general \cref{thm:UniformExponentialMixing} under the following assumption.

\begin{assumption}
\label{assumption}
The semigroup $\langle\gen\rangle_\sg$ generated by the set of generators $\gen$, associated to the alphabet $\calA$ of the symbolic coding, has full trace field:
\begin{align*}
\mathbb Q(\tr(\Ad(\langle\gen\rangle_\sg))) = \K.
\end{align*}
\end{assumption}

The assumption is mild when the degree of the number field is small. For instance, for $\K = \Q(\sqrt{2})$, only one of the hyperbolic elements which form $\langle\gen\rangle_\sg$ needs to have an irrational trace, under the adjoint representation.

We also prove the following theorem when $\Gamma$ has parabolic elements; otherwise, it was obtained without assumptions in \cite{Sar22a,Sar22b}).

\begin{theorem}
\label{thm:UniformExponentialMixing}
If $n = 3$, assume $\K = \Q$; otherwise, impose \cref{assumption}. There exist $\eta > 0$, $C > 0$, and a nontrivial proper ideal $\mathfrak{q}_0 \subset \calO_\K$ such that for all (square-free if $n = 3$) ideals $\mathfrak{q} \subset \calO_\K$ coprime to $\mathfrak{q}_0$ and $\phi, \psi \in C^1(\Gamma_\mathfrak{q} \backslash G)$ and $t > 0$, we have
\begin{multline*}
\left|\int_{\Gamma_\mathfrak{q} \backslash G} \phi(xa_t) \psi(x) \, dm_{\mathfrak{q}}^{\mathrm{BMS}}(x) - \frac{1}{m^{\mathrm{BMS}}_\mathfrak{q}(\Gamma_\mathfrak{q} \backslash G)} m_{\mathfrak{q}}^{\mathrm{BMS}}(\phi) \cdot m_{\mathfrak{q}}^{\mathrm{BMS}}(\psi)\right| \\
\leq CN_\K(\mathfrak{q})^C e^{-\eta t} \|\phi\|_{C^1} \|\psi\|_{C^1}.
\end{multline*}
\end{theorem}

\begin{remark}
\Cref{thm:UniformExponentialMixing} contains \cref{thm:UniformExponentialMixingOverQ} because \cref{assumption} holds trivially for $\K = \Q$. This assumption can be removed if one generalizes/modifies the symbolic coding of \cite{LP23} to allow finitely many elements of $\Gamma$ of any desired choice in the set of generators $\gen$ (including parabolic elements). We do not treat general number fields for $n = 3$ because (similar to \cref{rem:HdS_AbsolutelySimpleHypothesis}) the techniques employed to deal with trace fields require that the algebraic group be absolutely simple but $\SO_{3, 1}$ is not.
\end{remark}

\subsection{Applications to affine sieve and resonance-free regions}
As in \cite{Sar22a}, there are several theorems we can derive as applications. Namely, \cite[Theorems 1.4--1.7]{Sar22a} have the obvious generalizations thanks to the works of Mohammadi--Oh \cite{MO15} and Margulis--Mohammadi--Oh \cite{MMO14}. For the convenience of the reader, we present here two of them where the congruence setting is vital. Of course the main point here is that $\Gamma$ is geometrically finite and not necessarily convex cocompact.

\Cref{thm:AffineSieveTheorem} is the application to affine sieve as in \cite{BGS11}. It follows from \cref{thm:UniformExponentialMixingOverQ} via the work of \cite{MO15}. We specialize to the case $\mathbf{G} = \SO_Q$ for some integral quadratic form $Q$ of signature $(n, 1)$ and $\Gamma < \mathbf{G}(\mathbb Z)$. Equip $\mathbb R^{n + 1}$ with \emph{any} norm $\|\bigcdot\|$ and let $\mathbf{G}(\mathbb R)$ act on it by the standard representation.

\begin{theorem}
	\label{thm:AffineSieveTheorem}
	Let $w_0 \in \mathbb Z^{n + 1} \smallsetminus \{0\}$. If $n = 2$ and $Q(w_0) > 0$, we further assume that $w_0$ is not externally $\Gamma$-parabolic (see \cite{MO15} for details). Let $F = F_1 F_2 \dotsb F_r \in \mathbb Q[x_1, x_2, \dotsc, x_{n + 1}]$ for some $r \in \mathbb N$ be a factorization over $\mathbb Q$ into irreducible polynomials over $\mathbb C$ which are $\Z$-valued on $\Gamma w_0$. Then, there exist constants $C_1 > 0$, $C_2 > 0$, and $R \in \mathbb N$ such that for all $T > 0$, we have
	\begin{align*}
		\#\{x \in \Gamma w_0: \|x\| \leq T,  F_j(x) \text{ is prime for all } j \in \{1, 2, \dotsc, r\}\} &\leq C_1\frac{T^{\delta_\Gamma}}{\log(T)^r}, \\
		\#\{x \in \Gamma w_0: \|x\| \leq T, F(x) \text{ has at most $R$ prime factors}\} &\geq C_2\frac{T^{\delta_\Gamma}}{\log(T)^r}.
	\end{align*}
\end{theorem}

\Cref{thm:ResonanceFreeRegions} on resonance-free regions for the resolvent of the Laplacian can be deduced from \cref{thm:UniformExponentialMixing} exactly as in Li--Pan \cite[\S 9]{LP23}. We introduce some notation and background. Let $\mathfrak{q} \subset \mathcal{O}_{\mathbb K}$ be an ideal. If $\delta_\Gamma \in (\frac{n - 1}{2}, n - 1]$, then \cref{thm:ResonanceFreeRegions} follows from the spectral gap result in Step II of the proof of \cite[Theorem 11.1]{Kim15} which generalizes \cite[Theorem 1.2]{BGS11}. Thus, we may restrict to $\delta_\Gamma \in (0, \frac{n - 1}{2}]$. In this case, the Laplacian
\begin{align*}
\Delta_\mathfrak{q}: L^2(\Gamma_\mathfrak{q} \backslash \mathbb H^n, \mathbb C) \to L^2(\Gamma_\mathfrak{q} \backslash \mathbb H^n, \mathbb C)
\end{align*}
has a purely continuous spectrum, $\bigl[\left(\frac{n - 1}{2}\right)^2, \infty\bigr)$ \cite{LP82,Pat76,Sul79,Sul87}. It follows that the resolvent
\begin{align*}
\xi \mapsto R_\mathfrak{q}(\xi) = \bigl(\Delta_\mathfrak{q} - \xi(n - 1 - \xi)\Id\bigr)^{-1}
\end{align*}
is holomorphic on the half plane $\bigl(\frac{n - 1}{2}, +\infty\bigr) + i\R$. Guillarmou--Mazzeo \cite{GM12} showed that viewing the resolvent $R_\mathfrak{q}$ as a map whose values are $C_{\mathrm{c}}^\infty(\Gamma_\mathfrak{q} \backslash \mathbb H^n, \mathbb C) \to C^\infty(\Gamma_\mathfrak{q} \backslash \mathbb H^n, \mathbb C)$ bounded operators, it has a meromorphic extension to the \emph{entire} complex plane---its poles are called \emph{resonances}. Moreover, Patterson \cite{Pat88} showed that the map $\xi \mapsto \Gamma\bigl(\xi - \frac{n - 1}{2} + 1\bigr) R_\mathfrak{q}(\xi)$ has exactly one simple pole at $\delta_\Gamma$ on the line $\delta_\Gamma + i\R$.

\begin{theorem}
\label{thm:ResonanceFreeRegions}
If $n = 3$, assume $\K = \Q$; otherwise, impose \cref{assumption}. There exist $\eta > 0$ and a nontrivial proper ideal $\mathfrak{q}_0 \subset \mathcal{O}_{\mathbb K}$ such that for all (square-free if $n = 3$) ideals $\mathfrak{q} \subset \mathcal{O}_{\mathbb K}$ coprime to $\mathfrak{q}_0$, we have that
\begin{align*}
(\delta_\Gamma - \eta, +\infty) + i\R
\end{align*}
is a resonance-free half plane for $R_\mathfrak{q}$ if $\delta_\Gamma \in \bigl\{\frac{n - 1}{2} - k: k \in \mathbb N\bigr\}$ but with an exception of a simple pole at $\delta_\Gamma$ if $\delta_\Gamma \notin \bigl\{\frac{n - 1}{2} - k: k \in \mathbb N\bigr\}$.
\end{theorem}

\subsection{Proof outline and organization}
In the past decade or two, there have been several works to develop tools related to mixing, exponential mixing, and superstrong approximation. In this paper, we need to combine a variety of such tools. In doing so, the main obstacle that arises is the seeming incompatibility between the fact that the \emph{return trajectory subgroups} are infinitely generated and the fact that the associated expander graphs would be well-defined only for subgroups that are finitely generated. The return trajectory subgroups were introduced in \cite{Sar22a} and in the geometrically finite case, they are defined from the free subsemigroup generated by the countably infinite alphabet and hence infinitely generated. We detail the plan of the paper below.

\begin{itemize}
\item In \cref{sec:Preliminaries}, we cover the basic setup and objects related to geometrically finite subgroups. In \cref{sec:SymbolicModel}, we use the countably infinite coding constructed by Li--Pan \cite{LP23} to obtain a symbolic model which handles the frame flow simultaneously for all the congruence covers. Then, in \cref{sec:TransferOperators}, we introduce the \emph{congruence} transfer operators with holonomy
\begin{align*}
\mathcal{M}_{\xi\Ret, \mathfrak{q}, \rho}: C_{\mathrm{b}}\bigl(\Lambda_+, L^2(F_{\mathfrak{q}}, \C) \otimes V_\rho^{\oplus \dim(\rho)}\bigr) \to C_{\mathrm{b}}\bigl(\Lambda_+, L^2(F_{\mathfrak{q}}, \C) \otimes V_\rho^{\oplus \dim(\rho)}\bigr)
\end{align*}
for $\xi = a + ib \in \C$, a nontrivial ideal $\mathfrak{q} \subset \mathcal{O}_\mathbb{K}$, and an irreducible representation $\rho: M \to \U(V_\rho)$, defined by
\begin{align*}
	\mathcal{M}_{\xi\Ret, \mathfrak{q}, \rho}(H)(x) &= \sum_{\gamma \in \gen} e^{-(a + \delta_\Gamma -ib)\Ret(\gamma x)} \bigl(\gamma^{-1} \otimes \rho(\Hol(\gamma x)^{-1})\bigr) H(\gamma x).
\end{align*}
Here, $\Lambda_+$ is some subset of the limit set of $\Gamma$ which admits a countably infinite coding by $\gen$, and $\Ret$ is the \emph{return time map}, and $\Hol$ is the $M$-valued \emph{holonomy} which ``keeps track of the coordinate in the fibers'' of the principal $M$-bundle $\Gamma_\mathfrak{q} \backslash G \to \Gamma_\mathfrak{q} \backslash G/M$. Also, we refer to $\gamma^{-1}$ appearing as a tensor factor as the \emph{congruence cocycle} which ``keeps track of the coordinate in the fibers'' of the congruence cover $\Gamma_\mathfrak{q} \backslash G/M \to \Gamma \backslash G/M$. Now, the Laplace transform of the correlation functions for frame flows can be written in terms of the Neumann series of the transfer operators with holonomy. We seek certain spectral bounds for these operators which are uniform in the ideals $\mathfrak{q} \subset \mathcal{O}_\mathbb{K}$. Then, such bounds can be converted to uniform exponential mixing by a standard Payley--Wiener type of analysis involving the inverse Laplace transform.
\item For the uniform spectral bounds for large frequencies, i.e., $|b| \gg 1$ or $\rho$ nontrivial, we use a frame flow version of Dolgopyat's method together with the feature that the congruence cocycles are locally constant on cylinders. This simple but crucial fact was first observed and used in \cite{OW16} and later in \cite{MOW19,Sar22a}. In \cref{sec:LNIC_NCP_LDP,sec:Dolgopyat'sMethod}, we provide the key ingredients and the key theorem from which the uniform spectral bounds are deduced. We refer the reader to \cite{LPS25} for the full proofs but include the parts where there are differences due to the congruence cocycle.
\item In \cref{sec:Reduction}, we outline a standard reduction and in \cref{sec:ProofOfUniformSpectralBoundIReducedViaExpansionMachinery}, we outline how to prove the reduced theorem. We refer the reader to \cite{Sar22a} for the full proofs but include the parts where there are differences due to the countably infinite coding. The idea (from prior works) is to mimic a random walk process and derive an $L^2$-flattening lemma using the expansion machinery. The main difference arises in the proof of \cref{lem:GV_Expander} which uses the expansion machinery. The arguments from the prequel \cite{Sar22a} does not work directly. The issue is that the semigroup and the associated return trajectory subgroup is infinitely generated---this gives rise to two interlinked problems:
\begin{enumerate*}
\item there is now no well-defined notion of associated expander graphs; and
\item the $\epsilon > 0$ of the lemma which depend on the letters of the alphabet cannot na\"ively be chosen uniformly in the letters as the alphabet is now countably infinite.
\end{enumerate*}
We address these problems simultaneously. The idea is that the existing argument from \cite{Sar22a} works if we restrict to a suitable finite subset of the alphabet while the corresponding return trajectory subgroups still have the desired properties to form expanders.
\item \Cref{sec:ZariskiDensityAndTraceFieldOfTheReturnTrajectorySubgroups} is the key section where we show the above desired properties. We wish to transfer the Zariski density and full trace field properties to the return trajectory subgroups. Here, we generalize the arguments from \cite{Sar22a} in a suitable fashion. The lemma that the limit set of the return trajectory subgroups coincides with that of $\Gamma$ no longer holds true. However, one can still prove sufficiently strong properties for the \emph{radial} limit set. This suffices when we are working over $\Q$. For other number fields, transferring the full trace field property is harder. First, due to certain limitations of the construction of the coding from \cite{LP23}, we impose an assumption---it is mild if the degree of the number field is small. Even then, the transferring argument requires new tools related to the concept of weak commensurability. We show that the return trajectory subgroups are weakly commensurable to $\Gamma$ which ensures that the trace fields coincide.
\end{itemize}

\subsection*{Acknowledgements}
Part of this research was completed during the author's time at ETH Z\"urich; he acknowledges support under SNSF grant 10003145. Part of this research was also completed while the author was visiting the Mathematical Sciences Research Institute (MSRI), now becoming the Simons Laufer Mathematical Sciences Institute (SLMath), which is supported by the National Science Foundation (Grant No. DMS-2424139).

\section{Preliminaries}
\label{sec:Preliminaries}
We first introduce the setting and the main objects for the rest of the paper.

Let $\HS$ be the $n$-dimensional hyperbolic space for $n \geq 2$. We often use the upper half space model:
\begin{align*}
\HS = \{(x_1, x_2, \dotsc, x_n) \in \R^n: x_n > 0\}.
\end{align*}
Let $\K$ be a totally real number field and $\calO_\K$ be its ring of integers. Let $\mathbf{G} < \GL_N$ for some $N \in \mathbb N$ be a linear algebraic group defined over $\K$ such that $\mathbf{G}(\R) \cong \SO(n, 1)$ and $\mathbf{G}^\sigma(\R)$ is compact for all nontrivial embeddings $\sigma: \K \hookrightarrow \R$. Let $G \coloneq \mathbf{G}(\R)^\circ$, the connected component containing the identity element $e \in \mathbf{G}(\R)$. We can identify it with the group of orientation preserving isometries of $\HS$. Let $o \in \HS$ be a reference point and $v_o \in \T^1(\HS)$ be a reference tangent vector based at $o$. Define the stabilizer subgroups $K = \Stab_G(o)$ which is isomorphic to $\SO(n)$ and $M = \Stab_G(v_o) < K$ which is isomorphic to $\SO(n - 1)$. Let $\Gamma < G$ be a Zariski dense torsion-free discrete subgroup. Our base hyperbolic manifold is $X = \Gamma \backslash \HS \cong \Gamma \backslash G/K$, its unit tangent bundle is $\T^1(X) \cong \Gamma \backslash G/M$, and its (oriented orthonormal) frame bundle is $\F(X) \cong \Gamma \backslash G$ which is a principal $\SO(n)$-bundle over $X$ and a principal $\SO(n - 1)$-bundle over $\T^1(X)$. Let $A = \{a_t\}_{t \in \R} < G$ be a one-parameter subgroup of semisimple elements parametrized such that its canonical right translation action on $G/M$ and $G$ corresponds to the geodesic flow and the frame flow respectively. Note that $C_G(A) = AM$. We fix a left $G$-invariant and right $K$-invariant Riemannian metric on $G$ which descends to the hyperbolic metric on $\HS \cong G/K$. We denote by $\langle \bigcdot, \bigcdot\rangle$ and $\|\bigcdot\|$ the inner product and norm respectively on any tangent space of $G$ or its quotient spaces. We also denote by $d$ the distance function on $G$ or its quotient spaces.

We will need to use the strong approximation theorem of Weisfeiler \cite{Wei84} later on. Consequently, we need to work on the simply connected cover $\tilde{\mathbf G}$ endowed with the covering map $\pi: \tilde{\mathbf G} \to \mathbf G$ defined over $\K$ (see \cite{PR94,Mil17}). Let $\tilde{G} = \tilde{\mathbf G}(\R)$ which is connected and projects down to $\pi(\tilde{G}) = G$ \cite[Proposition 1.5.5 and Theorem 2.3.1]{Mar91}. Let $\tilde{\Gamma} < \tilde{G}$ be a Zariski dense subgroup containing the finite central subgroup $\ker(\pi) = \{e, -e\}$ as the only torsion elements. In order to have the notion of congruence subgroups, let us suppose that $\tilde{\Gamma} < \tilde{\mathbf G}(\calO_\K)$. For notational convenience, we denote the trace field of any subgroup $\Gamma' < \mathbf{G}(\R)$ or $\Gamma' < \tilde{\mathbf{G}}(\R)$ by
\begin{align*}
\K_{\Gamma'} \coloneq \mathbb Q(\tr(\Ad(\Gamma'))).
\end{align*}
To use the strong approximation theorem, we also need to make the technical assumption that $\tilde{\Gamma}$ has trace field $\K_{\tilde{\Gamma}} = \K$. Then we take $\Gamma = \pi(\tilde{\Gamma})$ which also has trace field $\K_\Gamma = \K$ \cite[Corollary 1.4.8]{Mar91}.

\subsection{Geometrically finite subgroups}
We now introduce some basic concepts related to geometrically finite subgroups. Let $(e_1, e_2, \dotsc, e_n)$ be the standard basis of $\R^n$. We identify $\R^{n - 1} = \langle e_1, e_2, \dotsc, e_{n - 1}\rangle$ and view it as a subspace of $\R^n$. Denote by $\partial_\infty\HS = \R^{n - 1} \cup \{\infty\}$ the boundary at infinity and by $\overline{\HS} = \HS \cup \partial_\infty\HS$ the compactification of $\HS$.

\begin{definition}[Limit set]
The \emph{limit set} of $\Gamma$ is the set of limit points $\Lambda(\Gamma) = \lim(\Gamma o) \subset \partial_\infty\HS \subset \overline{\HS}$.
\end{definition}

\begin{definition}[Radial limit set]
\label{def:RadialLimitSet}
A point $\xi \in \Lambda(\Gamma)$ is called a \emph{radial limit point} of $\Gamma$ if for some geodesic $\xi: \R \to \HS$ with $\lim_{t \to \infty} \xi(t) = \xi$, there exists $r > 0$ and sequences $\{\gamma_k\}_{k \in \mathbb N} \subset \Gamma$ and $\{t_k\}_{k \in \mathbb N} \subset \R$ such that $d(\gamma_k o, \xi(t_k)) < r$ and $\lim_{k \to \infty} \gamma_k o = \xi$. The \emph{radial limit set} is the set of radial limit points $\Lambda_{\mathrm{r}}(\Gamma) \subset \Lambda(\Gamma)$.
\end{definition}

\begin{definition}[Parabolic limit set]
\label{def:ParabolicLimitSet}
A point $\xi \in \Lambda(\Gamma)$ is called a \emph{parabolic limit point} of $\Gamma$ if it is the (unique) fixed point of a parabolic element in $\Gamma$. Furthermore, it is said to be \emph{bounded} if $\Stab_{\Gamma}(\xi)\backslash (\Lambda(\Gamma) \smallsetminus \{\xi\})$ is compact. The \emph{parabolic (resp. bounded parabolic) limit set} is the set of parabolic (resp. bounded parabolic) limit points $\Lambda_{\mathrm{p}}(\Gamma) \subset \Lambda(\Gamma)$ (resp. $\Lambda_{\mathrm{bp}}(\Gamma) \subset \Lambda(\Gamma)$).
\end{definition}

\begin{definition}[Critical exponent]
The \emph{critical exponent} $\delta_\Gamma$ of $\Gamma$ is the abscissa of convergence of the Poincar\'{e} series $\mathscr{P}_\Gamma(s) = \sum_{\gamma \in \Gamma} e^{-s d(o, \gamma o)}$.
\end{definition}

\begin{remark}
The above definitions are independent of the choice of $o \in \HS$.
\end{remark}

\begin{definition}[Geometrically finite]
The subgroup $\Gamma < G$ is said to be \emph{geometrically finite} if the $1$-neighborhood of the \emph{convex core} $\Core(X) = \Gamma \backslash \Hull(\Lambda(\Gamma)) \subset X$, where $\Hull$ denotes the convex hull, is of finite volume.
\end{definition}

We have the following theorem due to Bowditch \cite{Bow93} (see also \cite{KL19}).

\begin{theorem}[Geometrically finite]
The following are equivalent:
\begin{enumerate}
\item $\Gamma < G$ is geometrically finite;
\item $\Lambda(\Gamma) = \Lambda_{\mathrm{r}}(\Gamma) \sqcup \Lambda_{\mathrm{p}}(\Gamma)$;
\item $\Lambda(\Gamma) = \Lambda_{\mathrm{r}}(\Gamma) \sqcup \Lambda_{\mathrm{bp}}(\Gamma)$.
\end{enumerate}
\end{theorem}

We assume that $\Gamma$ is \emph{geometrically finite with parabolic elements} throughout the paper.

\begin{remark}
In our case, $\delta_\Gamma \in (0, n - 1]$ and coincides with the Hausdorff dimension of $\Lambda(\Gamma)$.
\end{remark}

\subsection{Structure of cusps}
\label{subsec:StructureOfCusps}
Assume henceforth (by modifying the action on $\HS$ using a conjugation if necessary) that $\infty \in \partial_\infty\HS$ is a parabolic fixed point of the geometrically finite subgroup $\Gamma < G$. Let $\Gamma'_{\infty}=\Stab_{\Gamma}(\infty) <\Gamma$ be the parabolic subgroup fixing $\infty$. Then $\Gamma'_{\infty}$ acts on $\R^{n - 1} \subset \partial_\infty\HS$ by Euclidean isometries. In fact, a theorem of Bieberbach (see \cite[\S 7.5]{Rat19}) states the following: there exist a maximal normal abelian subgroup $\Gamma_{\infty} < \Gamma'_{\infty}$ of finite index and a $\Gamma'_{\infty}$-invariant affine subspace $Z\subset \R^{n - 1}$ of dimension $k \in \N$ such that $\Gamma_{\infty}$ acts on $Z$ as a group of translations of rank $k$. Consequently, decomposing $\R^{n - 1} = Y \oplus Z$ into orthogonal affine subspaces and viewing the latter as vector spaces in their own right, we can write the $\Gamma'_{\infty}$-action on $\R^{n - 1}$ in the following form: for all $\gamma \in \Gamma'_{\infty}$, there exist
\begin{align*}
A_{\gamma} &\in \operatorname{O}(Y), & R_{\gamma} &\in \operatorname{O}(Z), & b_{\gamma} &\in Z,
\end{align*}
with $R_{\gamma} = \Id$ if $\gamma \in \Gamma_{\infty}$, such that
\begin{align*}
\gamma(y,z)=(A_{\gamma} y, R_{\gamma}z+b_{\gamma}) \qquad \text{for all $(y, z) \in Y \oplus Z$}.
\end{align*}
Here, the dimension $k$ is called the \emph{rank} of the parabolic fixed point $\infty$. The $\Gamma_{\infty}$ action on $Z$ admits a fundamental domain $\Delta'_{\infty} \subset Z$ which is an open $k$-dimensional parallelotope. Since $\Gamma$ is geometrically finite, $\infty$ is a bounded parabolic fixed point, implying that $\Gamma_{\infty}\backslash (\Lambda(\Gamma) \smallsetminus \{\infty\})$ is compact. Therefore, there exists a constant $C_\infty>0$ such that the set $B_{Y}\coloneq\{y\in Y: \|y\|<C_\infty\}$ has the property that 
\begin{align}
\label{limit set inclusion}
\Lambda(\Gamma) \smallsetminus \{\infty\} \subset \bigcup_{\gamma\in \Gamma_{\infty}} \gamma \bigl(\overline{B_{Y}\times \Delta_{\infty}'}\bigr).
\end{align}
We call the open set $\Delta_{\infty} \coloneq B_Y \times \Delta'_{\infty}$ a {\it{fundamental domain}} for the parabolic fixed point $\infty$.

\subsection{Patterson--Sullivan density}
\label{subsec:Patterson--SullivanDensity}
In \cite{Pat76,Sul79}, a unique (when $\Gamma$ is of divergence type as in the geometrically finite case) $\Gamma$-invariant conformal density of dimension $\delta_\Gamma$ was constructed by Patterson and Sullivan, called the \emph{Patterson--Sullivan (PS) density}, which we characterize below. Let $\beta_{\xi}$ denote the \emph{Busemann function} at $\xi \in \partial_\infty\HS$ defined by $\beta_{\xi}(y, x) = \lim_{t \to \infty} (d(\xi(t), y) - d(\xi(t), x))$, where $\xi: \R \to \HS$ is any geodesic such that $\lim_{t \to \infty} \xi(t) = \xi$. We also allow tangent vector arguments for the Busemann function in which case we will use their basepoints in the definition. The PS density is the set $\bigl\{\mu^{\mathrm{PS}}_x\bigr\}_{x \in \HS}$ of finite Borel measures on $\partial_\infty\HS$ supported on $\Lambda(\Gamma)$ such that
\begin{enumerate}
\item	$\gamma_*\mu^{\mathrm{PS}}_x = \mu^{\mathrm{PS}}_{\gamma x}$ for all $\gamma \in \Gamma$ and $x \in \HS$,
\item	$\frac{d\mu^{\mathrm{PS}}_x}{d\mu^{\mathrm{PS}}_y}(\xi) = e^{\delta_\Gamma \beta_{\xi}(y, x)}$ for all $\xi \in \partial_\infty\HS$ and $x, y \in \HS$.
\end{enumerate}
Since the PS measures are absolutely continuous with respect to each other, it is convenient to fix $\mu \coloneq \mu_o^{\mathrm{PS}}$ for the rest of the paper.

\subsection{Bowen--Margulis--Sullivan measure}
\label{subsec:BMS_Measure}
For all $u \in \T^1(\HS)$, let $u^+$ and $u^-$ denote its forward and backward limit points. Using the Hopf parametrization via the homeomorphism $G/M \cong \T^1(\HS) \to \{(u^+, u^-) \in \partial_\infty\HS \times \partial_\infty\HS: u^+ \neq u^-\} \times \R$ defined by $u \mapsto (u^+, u^-, t = \beta_{u^-}(o, u))$, we define the \emph{Bowen--Margulis--Sullivan (BMS) measure} $m^{\mathrm{BMS}}$ on $G/M$ \cite{Mar04,Bow71,Kai90} by
\begin{align*}
dm^{\mathrm{BMS}}(u) = e^{\delta_\Gamma \beta_{u^+}(o, u)} e^{\delta_\Gamma \beta_{u^-}(o, u)} \, d\mu^{\mathrm{PS}}_o(u^+) \, d\mu^{\mathrm{PS}}_o(u^-) \, dt.
\end{align*}
Note that this definition only depends on $\Gamma$, and that $m^{\mathrm{BMS}}$ is left $\Gamma$-invariant and right $A$-invariant. We now define induced measures on other spaces all of which we call the BMS measures and denote by $m^{\mathrm{BMS}}$ by abuse of notation. Since $M$ is compact, we can use the Haar probability measure on $M$ to lift $m^{\mathrm{BMS}}$ to a right $M$-invariant measure on $G$. By left $\Gamma$-invariance, $m^{\mathrm{BMS}}$ now descends to measures on $\Gamma \backslash G$ and $\Gamma \backslash G/M$. Since $\Gamma$ is geometrically finite, $m^{\mathrm{BMS}}$ is a finite measure on $\Gamma \backslash G$ which we may normalize so that $m^{\mathrm{BMS}}(\Gamma \backslash G) = 1$.

\section{Symbolic model for the frame flow for congruence covers}
\label{sec:SymbolicModel}
Recall that we assume $\Gamma < G$ has parabolic elements so that $\Gamma\backslash G$ has cusps. In this section, we first recall the basic symbolic model for the geodesic flow constructed by Li--Pan \cite[Proposition 4.1]{LP23}. Using it as a starting point, it was extended to the frame flow in \cite[\S 3.3]{LPS25}. Here, we extend it further to congruence covers.

\subsection{Expanding map on the boundary}
\label{subsec:expanding map}
Fix $\Delta_0 \coloneq \Delta_\infty$ to be the fundamental domain for the parabolic fixed point $\infty \in \partial_\infty\HS$. We write $\|\bigcdot\|_{\mathrm{op}}$ for the operator norm in the upper half space model.

\begin{proposition}[{\cite[Proposition 4.1]{LP23}}]
\label{prop:Coding}
There exist $C_1 > 0$, $\theta \in (0, 1)$, $\epsilon_0 \in (0, 1)$, and a set of generators $\gen=\{\gamma_j\}_{j\in \calA}$ associated to a countably infinite alphabet $\calA$ such that the following holds. Fix:
\begin{itemize}
\item $\Delta_j \coloneq \gamma_j \Delta_0$ for all $j \in \calA$, and $\Delta_{\sqcup} \coloneq \bigsqcup_{j \in \calA} \Delta_j$;
\item an expanding map $T: \Delta_{\sqcup} \to \Delta_0$ defined by
\begin{align*}
T|_{\Delta_j} = \gamma_j^{-1} \qquad \text{for all $j \in \calA$};
\end{align*}
\item a return time map $\Ret: \Delta_{\sqcup} \to \R$ defined by
\begin{align*}
\Ret(x) = \log\|(dT)_x\|_{\mathrm{op}} \qquad \text{for all $x \in \Delta_{\sqcup}$}.
\end{align*}
\end{itemize}
We have:
\begin{enumerate}
\item $\mu(\Delta_0) = \sum_{j \in \calA} \mu(\Delta_j)$;
\item for all $j \in \calA$, the element $\gamma_j$ acts by a uniform contraction: we have $\|(d\gamma_j)_x\|_{\mathrm{op}} \leq \theta$ for all $x \in \Delta_0$;
\item for all $j \in \calA$, we have $\|d(\log\|(d\gamma_j)_{\bigcdot}\|_{\mathrm{op}})_x\|_{\mathrm{op}}\leq C_1$ for all $x \in \Delta_0$;
\item $\Ret$ satisfies the exponential tail property: $e^{\epsilon_0 \Ret} \in L^1(\Delta_{\sqcup}, \mu)$.
\end{enumerate}
\end{proposition}

By a \emph{sequence} $\alpha = (\alpha_1, \dotsc, \alpha_k)$ of \emph{length} $k \in \Z_{\geq 0}$, we shall mean a sequence in the letters of $\calA$; we follow the convention that $k = 0$ gives the empty sequence. Correspondingly, we have an inverse branch
\begin{align}
	\label{eqn: cocycle for a sequence}
	\gamma(\alpha) := \gamma_{\alpha_1} \dotsb \gamma_{\alpha_k} \in \Gamma
\end{align}
of length $k$. Denote by
\begin{equation*}
\gen^{(k)}\coloneq\{\gamma(\alpha): \alpha \text{ is a sequence of length } k\} \subset \Gamma
\end{equation*}
the set of inverse branches of length $k \in \Z_{\geq 0}$; note that $\gen^{(0)} \coloneq \{e\}$ and $\gen^{(1)} \coloneq \gen$. We denote by
\begin{align*}
\langle \gen\rangle_\sg \coloneq \bigsqcup_{k \in \Z_{\geq 0}}\gen^{(k)} < \Gamma
\end{align*}
the sub\emph{semi}group generated by $\gen$. For the rest of the paper, $\Gamma$ will be under \cref{assumption}, i.e.,
\begin{align*}
\mathbb Q(\tr(\Ad(\langle\gen\rangle_\sg))) = \K
\end{align*}
which holds trivially for $\K = \Q$.

For any $\gamma\in \langle \gen\rangle_\sg$, set
\begin{align*}
\|d\gamma\|\coloneq\sup_{x\in \Delta_0} \|(d\gamma)_x\|_{\mathrm{op}}.
\end{align*}
Define the following, which is the space that admits a countably infinite coding:
\begin{align}
\nonumber
\Lambda_+ &:= \bigcap_{k \in \N} T^{-k}(\Delta_0) = \{x\in \Lambda(\Gamma): T^k \text{ is defined at } x \text{ for all } k \in \mathbb{N}\} \\
\label{eqn: + limit set}
&\subsetneq \Lambda(\langle \gen\rangle_\sg) \subsetneq \Lambda(\Gamma).
\end{align}

The expanding map $T$ gives a contraction action in a neighborhood of $\infty$. We denote by $d_{\mathbb{S}^{n - 1}}$ the spherical metric on $\partial_\infty\HS \cong \mathbb S^{n - 1}$ in the ball model centered at $o \in \HS$.
\begin{lemma}[{\cite[Lemma 4.8]{LP23}}]
\label{lambda-}
There exist $\theta \in (0, 1)$ and an open neighborhood $\Lambda_- \subset \Lambda(\Gamma)$ of $\infty$ such that $\Lambda_{-}$ is disjoint from $\overline{\Delta_0}$ and $\gamma^{-1}(\Lambda_-) \subset \Lambda_-$ with
\begin{align*}
d_{\mathbb{S}^{n - 1}}(\gamma^{-1}y, \gamma^{-1}y') \leq \theta d_{\mathbb{S}^{n - 1}}(y,y') \qquad \text{for all $y,y'\in \Lambda_-$ and $\gamma\in \gen$}.
\end{align*}
\end{lemma}

Using \cref{prop:Coding} and following \cite[Lemma 2]{You98}, there exists a $T$-invariant ergodic probability measure $\nu$ on $\Delta_0$ such that
\begin{equation}
\label{eqn:nu_and_mu}
d\nu = f_0 \, d\mu
\end{equation}
where $f_0: \Delta_0 \to \R$ is a positive Lipschitz density function bounded away from $0$ and $+\infty$. Then, $\Lambda_+ \subset \Delta_0$ is a full measure subset with respect to the equivalent measures $\mu$ and $\nu$.

\begin{definition}[Cylinder]
Subsets of the form $\gamma\Delta_0 \subset \Delta_0$ for some $\gamma \in \gen^{(k)}$ and $k \in \Z_{\geq 0}$ are \emph{cylinders} of \emph{length} $k$. We denote them by $\mathtt{C}$ or other typewriter style letters.
\end{definition}

We introduce another distance function $D$ on $\Delta_0$ as in \cite[\S 5]{Sto11} by
\begin{align*}
D(x, y) =
\begin{cases}
\min \{\diam(\mathtt{C}): \text{cylinders } \mathtt{C} \subset \Delta_0 \text{ such that } x,y\in\overline{\mathtt{C}}\}, & x, y \in \Delta_0, x \neq y \\
0, & \text{otherwise}.
\end{cases}
\end{align*}

\begin{lemma}[{\cite[Lemma 3.4]{LPS25}}]
The following holds:
\begin{enumerate}
\item $D$ is a distance function on $\Delta_0$;
\item for all $x,y\in \Delta_0$, we have $d_{\mathrm{E}}(x,y)\leq D(x,y)$;
\item for all $k \in \mathbb{N}$, $\gamma\in \gen^{(k)}$, and $x,y\in \Delta_0$, we have $D(\gamma x,\gamma y)\leq \theta^k D(x,y)$.
\end{enumerate}
\end{lemma}

\begin{lemma}[{\cite[Lemma 3.5]{LPS25}}]
\label{lem:CylinderEstimate}
For all cylinders $\mathtt{C} \subset \Delta_0$, Borel subsets $E \subset \Delta_0$, and $\gamma \in \langle \gen\rangle_\sg$, we have
\begin{align*}
\diam(\gamma\mathtt{C}) &\asymp \|d\gamma\|\diam(\mathtt{C}), & \mu(\gamma E) &\asymp \|d\gamma\|^{\delta_\Gamma} \mu(E), & \nu(\gamma E) &\asymp \|d\gamma\|^{\delta_\Gamma} \nu(E),
\end{align*}
with some implicit constant $C_{\mathrm{cyl}} > 1$.
\end{lemma}

\subsection{Symbolic model}
\label{subsec:SymbolicModelForFrameFlows}
Define the map $\hat{T}: \Lambda_+ \times \Lambda_- \to \Lambda_+ \times \Lambda_-$ by
\begin{equation*}
\hat{T}(x,y) = \bigl(\gamma_j^{-1}x,\gamma_j^{-1}y\bigr) \qquad \text{if $x\in \Delta_j$, for all $y \in \Lambda_-$ and $j \in \calA$}.
\end{equation*}

Recall the return time map $\Ret:\Delta_{\sqcup}\to \R$ from \cref{prop:Coding}. We extend it to the map $\Ret: \Delta_{\sqcup} \times \Lambda_- \to \R$ by setting $\Ret(x,y)=\Ret(x)$ for all $(x,y) \in \Delta_{\sqcup} \times \Lambda_-$ and use the notation $\Ret_0 = 0$ and
\begin{align}
\label{eqn:BirkhoffSum}
\Ret_k=\sum_{j=0}^{k-1}\Ret\circ \hat{T}^j \qquad \text{for all $k \in \N$}.
\end{align}
We define the space
\begin{align*}
\Lambda^\Ret \coloneq \{(x,y,s)\in \Lambda_+\times \Lambda_-\times \R:0\leq s<\Ret(x, y)\}.
\end{align*}

We introduce the following map which makes a time change and then an embedding using the Hopf parametrization $\T^1(\HS) \cong \{(x, y) \in \partial_\infty\HS \times \partial_\infty\HS: x \neq y\}\times \R$ as follows:
\begin{align*}
\tilde{\Phi}: \overline{\Delta_0} \times \overline{\Lambda_-} \times \R &\to \T^1(\HS)\\
\nonumber
(x,y,s)&\mapsto (x,y,s+\log(1+\|x\|^2)).
\end{align*}
It induces a map
\begin{align}
\label{eqn:embedding}
\Phi:\overline{\Delta_0} \times \overline{\Lambda_-} \times \R \to \T^1(\HS).
\end{align}
By abuse of notation, they restrict to maps
\begin{align*}
\tilde{\Phi}: \Lambda^\Ret &\to \T^1(\HS), & \Phi: \Lambda^\Ret &\to \T^1(X).
\end{align*}
Then $\tilde{\Phi}$ maps $\Lambda_+ \times \{\infty\}\times \{0\}$ into the unstable horosphere based at $\infty$ which contains $o \in \HS$, and $\Lambda^\Ret \cap (\{x\} \times \Lambda_- \times \{s\})$ into the stable horosphere based at $x \in \Lambda_+$ for all $s \in \R_{\geq 0}$.

Consider the set of rectangles $\mathcal{R} = \{R_j\}_{j \in \calA}$ where for all $j \in \calA$, we define $R_j \coloneq \Phi(\overline{\Delta_j} \times \overline{\Lambda_-} \times \{0\})$ with \emph{center} $w_j \coloneq \Phi(\gamma_j x_0, \infty, 0)$.
It satisfies the Markov property with respect to $\hat{T}$ and hence plays the role of a Markov section which was used in the convex cocompact case in \cite{Sar22a,Sar22b}. However, there are important differences:
\begin{itemize}
\item the alphabet $\gen$ is countably infinite;
\item the diameter of the rectangles are bounded above but not below by a positive number;
\item the return time map $\Ret$ is unbounded above.
\end{itemize}

Fix a reference point $x_0\in \Lambda_+ \subset \Delta_0$. We define a section
\begin{equation*}
\tilde{F}: \Delta_0 \times \Lambda_- \to \F(\HS)
\end{equation*}
which is \emph{smooth} in the first argument in the following fashion: 
\begin{itemize}
\item Fix a frame $\tilde{F}(x_0,\infty) \in \F(\HS)$ based at the tangent vector $\tilde{\Phi}(x_0, \infty, 0)$.
\item Extend the section $\tilde{F}$ such that for any $x,x' \in \Delta_0$, the frames $\tilde{F}(x,\infty)$ and $\tilde{F}(x',\infty)$ are backward asymptotic, i.e.,
\begin{equation*}
\lim_{t\to -\infty} d(\tilde{F}(x,\infty)a_t,\tilde{F}(x',\infty)a_t)=0.
\end{equation*}
Then, we must have $\tilde{F}(x',\infty)=\tilde{F}(x,\infty)n^+$ for some unique $n^+\in N^+$.
\item Extend the section $\tilde{F}$ such that for any $x \in \Delta_0$ and $y,y'\in \Lambda_-$, the frames $\tilde{F}(x,y)$ and $\tilde{F}(x,y')$ are forward asymptotic, i.e.,
\begin{equation*}
\lim_{t\to \infty}d(\tilde{F}(x,y)a_t,\tilde{F}(x,y')a_t)=0.
\end{equation*}
Then, we must have $\tilde{F}(x,y')=\tilde{F}(x,y)n^-$ for some unique $n^-\in N^-$.
\end{itemize}
Then, $\tilde{F}$ induces a map $F: \Delta_0 \times \Lambda_- \to \F(X)$.
\begin{definition}[Holonomy]
\label{holonomy}
The \emph{holonomy} is a map $\Hol: \Delta_{\sqcup} \times \Lambda_- \to M$ such that for all $(x,y) \in \Delta_{\sqcup} \times \Lambda_-$, we have
\begin{equation*}
F(x,y)a_{\Ret(x,y)}=F\bigl(\hat{T}(x,y)\bigr)\Hol(x,y)^{-1}.
\end{equation*}
\end{definition}

\begin{definition}[Generalized holonomy]
We call the combined map $\GHol: \Delta_{\sqcup} \times \Lambda_- \to AM$ defined by $\GHol(x, y) = a_{\Ret(x, y)} \Hol(x, y)$ for all $(x, y) \in \Delta_{\sqcup} \times \Lambda_-$  the \emph{generalized holonomy}.
\end{definition}

The following lemma can be deduced using the construction of $F$ (see \cite[Lemma 4.2]{SW21} for its proof).

\begin{lemma}
\label{lem:stable direction}
For all $x \in \Delta_{\sqcup}$, the maps $\Ret|_{\{x\} \times \Lambda_-}$, $\Hol|_{\{x\} \times \Lambda_-}$, and $\GHol|_{\{x\} \times \Lambda_-}$ are constant.
\end{lemma}

Now, we introduce the congruence aspect. Let $\mathfrak{q} \subset \calO_\K$ be an ideal. We have the reduction map
\begin{align*}
\pi_{\mathfrak{q}}: \tilde{\mathbf{G}}(\calO_\K) \to \tilde{\mathbf{G}}(\calO_\K/\mathfrak{q})
\end{align*}
and we define the \emph{principal congruence subgroup of level $\mathfrak{q}$} to be $\ker(\pi_{\mathfrak{q}}) \cdot \ker(\pi)$. We define the \emph{congruence subgroups} of $\tilde{\Gamma}$ and of $\Gamma$ of \emph{level $\mathfrak{q}$} to be the normal subgroups
\begin{align*}
\tilde{\Gamma}_{\mathfrak{q}} &= \tilde{\Gamma} \cap \ker(\pi_{\mathfrak{q}}) \cdot \ker(\pi) \lhd \tilde{\Gamma}, & \Gamma_{\mathfrak{q}} = \pi(\tilde{\Gamma}_{\mathfrak{q}}) \lhd \Gamma,
\end{align*}
respectively. Then, $\tilde{\Gamma}_{\mathfrak{q}}$ contains $\ker(\pi)$ as the only torsion elements. When $\mathfrak{q} \subset \calO_\K$ is a nontrivial ideal, define the finite group
\begin{align*}
F_{\mathfrak{q}} = \Gamma_{\mathfrak{q}} \backslash \Gamma \cong \tilde{\Gamma}_{\mathfrak{q}} \backslash \tilde{\Gamma}.
\end{align*}
Using the strong approximation theorem of Weisfeiler \cite[Theorem 10.1]{Wei84}, there exists a nontrivial proper ideal $\mathfrak{q}_0 \subset \calO_\K$ such that for all ideals $\mathfrak{q} \subset \calO_\K$ coprime to $\mathfrak{q}_0$, the map $\pi_{\mathfrak{q}}|_{\tilde{\Gamma}}$ is surjective and hence induces the isomorphism
\begin{align*}
\overline{\pi_{\mathfrak{q}}|_{\tilde{\Gamma}}}: F_{\mathfrak{q}} \to \tilde{G}_{\mathfrak{q}}
\end{align*}
where $\tilde{G}_{\mathfrak{q}} \coloneq \pi_{\mathfrak{q}}(\ker(\pi)) \backslash \tilde{\mathbf{G}}(\calO_\K/\mathfrak{q})$. For simplicity, we assume that the ideal $\mathfrak{q}_0 \subset \calO_\K$ suffices for the strong approximation theorem, \cite[Corollary 6]{GV12}, and \cite[Theorem 6.1]{HdS22} for the return trajectory subgroups introduced in \cref{sec:ZariskiDensityAndTraceFieldOfTheReturnTrajectorySubgroups}.

Let $\mathfrak{q} \subset \calO_\K$ be a nontrivial ideal for the rest of the section. The \emph{congruence cover} of $X$ of \emph{level $\mathfrak{q}$} is $X_\mathfrak{q} = \Gamma_\mathfrak{q} \backslash \mathbb H^n$. We have the isometries $\T^1(X_\mathfrak{q}) \cong \Gamma_\mathfrak{q} \backslash G/M$ and $\F(X_\mathfrak{q}) \cong \Gamma_\mathfrak{q} \backslash G$. Since $m^{\mathrm{BMS}}$ is left $\Gamma$-invariant, so in particular it is left $\Gamma_\mathfrak{q}$-invariant. Thus, it descends to the BMS measure $m^{\mathrm{BMS}}_\mathfrak{q}$ on $\Gamma_\mathfrak{q} \backslash G$ which, by right $M$-invariance, descends again to the BMS measure $m^{\mathrm{BMS}}_\mathfrak{q}$ on $\Gamma_\mathfrak{q} \backslash G/M$ by abuse of notation. Note that $m^{\mathrm{BMS}}_\mathfrak{q}(\Gamma_\mathfrak{q} \backslash G) = m^{\mathrm{BMS}}_\mathfrak{q}(\Gamma_\mathfrak{q} \backslash G/M) = \#F_{\mathfrak{q}}$.

We define the space
\begin{equation*}
\Lambda^{\Ret, \mathfrak{q}, M} := \{(x,y,\Gamma_{\mathfrak{q}}\gamma,m,s)\in \Lambda_+\times \Lambda_-\times F_{\mathfrak{q}}\times M\times \R:0\leq s<\Ret(x,y)\}.
\end{equation*}
Without the third component (or effectively for $\mathfrak{q} = \calO_\K$), we denote the above space by $\Lambda^{\Ret, M}$. We use the notations $\Hol^0 = \GHol^0 = e$ and
\begin{align*}
\Hol^k &= \prod_{j = 0}^{k - 1} \Hol \circ \hat{T}^j, & \GHol^k &= \prod_{j = 0}^{k - 1} \GHol \circ \hat{T}^j \qquad \text{for all $k \in \N$}
\end{align*}
where the products are taken in \emph{ascending} order from left to right.
The symbolic frame semiflow $\{\hat{T}_t\}_{t\geq 0}$ is defined by
\begin{equation*}
\hat{T}_t(x,y,\Gamma_{\mathfrak{q}}\gamma,m,s)=(\hat{T}^k(x,y),\Gamma_{\mathfrak{q}}\gamma \cdot \hat{\gamma},\Hol^k(x,y)^{-1}m, s+t-\Ret_k(x,y))
\end{equation*}
for all $(x,y,\Gamma_{\mathfrak{q}}\gamma,m,s) \in (\Lambda_+ \cap \hat{\gamma}\Delta_0)\times \Lambda_- \times F_{\mathfrak{q}}\times M \times \R$ and $\hat{\gamma} \in \gen^{(k)}$, where $k \in \Z_{\geq 0}$ such that $0\leq s+t-\Ret_k(x,y)<\Ret(\hat{T}^k(x,y))$. We call the acting element $\hat{\gamma}$ a \emph{cocycle}, and $\Gamma_{\mathfrak{q}} \hat{\gamma}$ a \emph{congruence cocycle}.

Abusing notation, we define the map
\begin{align*}
\tilde{\Phi}:\Lambda^{\Ret, M}&\to \F(\HS)\\
(x,y,m,s)&\mapsto \tilde{F}(x,y)a_sm.
\end{align*}
This map is well-defined, and it induces a map $\Phi_{\mathfrak{q}}:\Lambda^{\Ret, \mathfrak{q}, M} \to \F(X_\mathfrak{q})$. In fact,
\begin{equation*}
\Phi_\mathfrak{q}(x,y,\Gamma_{\mathfrak{q}}\gamma,m,a_s)=\Gamma_{\mathfrak{q}}\gamma F(x,y)a_sm.
\end{equation*}

Finally, we relate the measures $\hat{\nu}^{\Ret, \mathfrak{q}, M}$ and $m^{\operatorname{BMS}}_\mathfrak{q}$. By, \cite[Proposition 4.11]{LP23}, there exists a unique $\hat{T}$-invariant ergodic probability measure $\hat{\nu}$ on $\Lambda_+\times\Lambda_-$ which projects to the measure $\nu$ on $\Lambda_+$. We define a $\{\hat{T}_t\}_{t \geq 0}$-invariant measure on $\Lambda^{\Ret}$ by
\begin{align*}
d\hat{\nu}^{\Ret} \coloneq \hat{\nu}(\Ret)^{-1} \, d\hat{\nu}\, d\Leb,
\end{align*}
where $\Leb$ is the Lebesgue measure. We use the counting measure on $F_{\mathfrak{q}}$ and the Haar probability measure on $M$ to lift the measure $\hat{\nu}^{\Ret}$ to $\hat{\nu}^{\Ret, \mathfrak{q}, M}$ on $\Lambda^{\Ret, \mathfrak{q}, M}$.

The following is a corollary of \cite[Proposition 4.15]{LP23} and generalizes \cite[Corollary 3.10]{LPS25}. Here, $\{\mathcal{F}_t\}_{t \in \R}$ denotes the frame flow.

\begin{corollary}
\label{cor:factormap}
Let $\mathfrak{q} \subset \calO_\K$ be a nontrivial ideal. The map
\begin{equation*}
\Phi_{\mathfrak{q}}:\bigl(\Lambda^{\Ret, \mathfrak{q}, M},\{\hat{T}_t\}_{t \geq 0}, \hat{\nu}^{\Ret, \mathfrak{q}, M}\bigr)\to \bigl(\F(X_{\mathfrak{q}}),\{\mathcal{F}_t\}_{t \in \R},\m_{\mathfrak{q}}\bigr)
\end{equation*}
is a factor map, i.e.,
\begin{equation*}
(\Phi_{\mathfrak{q}})_{*}\bigl(\hat{\nu}^{\Ret, \mathfrak{q}, M}\bigr)=\m_{\mathfrak{q}} \qquad \text{and} \qquad \Phi_{\mathfrak{q}}\circ \hat{T}_t=\mathcal{F}_t\circ \Phi_{\mathfrak{q}} \qquad \text{for all $t\geq0$}.
\end{equation*}
\end{corollary}

\section{Congruence transfer operators with holonomy}
\label{sec:TransferOperators}
In this section, we introduce the congruence transfer operators with holonomy associated to the countably infinite coding. The main technical objective in this paper is to obtain uniform spectral bounds for these operators as stated in \cref{thm:UniformSpectralBoundI,thm:UniformSpectralBoundII}.

Recall that we have fixed the Haar probability measure on $M$. Then, $L^2(M)$ is a Hilbert space equipped with the standard inner product $\langle \bigcdot, \bigcdot \rangle$ defined by $\langle f, g\rangle = \int_M f\overline{g}$. As usual, we denote by $\|\bigcdot\|_2$ the corresponding $L^2$ norm on $L^2(M)$ and any of its subspaces. Similarly, for a nontrivial ideal $\mathfrak{q} \subset \calO_\K$, we have the counting measure on $F_\mathfrak{q}$ and the Hilbert space $L^2(F_\mathfrak{q})$. We will also use the Hilbert subspace $L_0^2(F_\mathfrak{q}) \subset L^2(F_\mathfrak{q})$ consisting of functions with $0$ mean.

We need to treat the function space
\begin{multline*}
C\bigl(\Lambda^\Ret \times F_\mathfrak{q} \times M, \C\bigr) \cong C\bigl(\Lambda^\Ret, C(F_\mathfrak{q} \times M, \C)\bigr) \\
\subset C\bigl(\Lambda^\Ret, L^2(F_\mathfrak{q} \times M)\bigr) \cong C\bigl(\Lambda^\Ret, L^2(F_\mathfrak{q}) \otimes L^2(M)\bigr).
\end{multline*}
Denote by $\widehat{M}$ the unitary dual of $M$ and by $1 \in \widehat{M}$ the trivial irreducible representation. Define $\widehat{M}_0 \coloneq \widehat{M} \smallsetminus \{1\}$. Applying the Peter--Weyl theorem for the left regular representation $(\varsigma, L^2(M))$ gives the Hilbert direct sum decomposition
\begin{align*}
(\varsigma, L^2(M)) = \widehat{\bigoplus}_{\rho \in \widehat{M}} (\rho, V_\rho)^{\oplus \dim(\rho)}.
\end{align*}

For all $b \in \R$ and $\rho \in \widehat{M}$, we define the tensored unitary representations $\rho_b: AM \to \U\bigl(V_\rho^{\oplus \dim(\rho)}\bigr)$ and $\rho_{b, \mathfrak{q}}: AM \to \U\bigl(L^2(F_\mathfrak{q}) \otimes V_\rho^{\oplus \dim(\rho)}\bigr)$ by
\begin{align*}
	\rho_b(a_tm)(z) &= e^{ibt}\rho(m)(z) \qquad \text{for all $z \in V_\rho^{\oplus \dim(\rho)}$}, \\
	\rho_{b, \mathfrak{q}}(a_tm)(z) &= (\Id \otimes \rho_b(a_tm))(z) \qquad \text{for all $z \in L^2(F_\mathfrak{q}) \otimes V_\rho^{\oplus \dim(\rho)}$},
\end{align*}
for all $t \in \mathbb R$, and $m \in M$.

Denote the Lie algebras $\mathfrak{a} := \operatorname{T}_e(A)$ and $\mathfrak{m} := \operatorname{T}_e(M)$. For any unitary representation $\rho: M \to \U(V)$ for some Hilbert space $V$, we denote the differential at $e \in M$ by $d\rho := (d\rho)_e: \mathfrak{m} \to \mathfrak{u}(V)$, and define the norm
\begin{align*}
\|\rho\| := \sup_{z \in \mathfrak{m}, \|z\| = 1} \|d\rho(z)\|_{\mathrm{op}}
\end{align*}
and similarly for any unitary representation $\rho: AM \to \U(V)$.

We recall the following useful fact regarding the Lie theoretic norms from \cite{Sar22a,Sar22b} which is a uniform version of \cite[Lemma 4.4]{SW21}.

\begin{lemma}[{\cite[Lemmas 4.4]{Sar22a}}]
\label{lem:maActionLowerBound}
There exists $\varepsilon_1 \in (0, 1)$ such that for all $b \in \R$, $\rho \in \widehat{M}$, nontrivial ideals $\mathfrak{q} \subset \calO_\K$, and $\omega \in L^2(F_\mathfrak{q}) \otimes V_\rho^{\oplus \dim(\rho)}$ with $\|\omega\|_2 = 1$, there exists $z \in \mathfrak{a} \oplus \mathfrak{m}$ with $\|z\| = 1$ such that $\|d\rho_{b, \mathfrak{q}}(z)(\omega)\|_2 \geq \varepsilon_1 \|\rho_b\|$.
\end{lemma}

In Dolgopyat's method, the source of the required oscillations come from the \emph{local non-integrability condition (LNIC)} (see \cref{subsec:LNIC}) and the oscillations are actually realized when $\|\rho_b\|$ is sufficiently large. This occurs precisely when $|b|$ is sufficiently large or $\rho \in \widehat{M}$ is nontrival. This motivates us to define the following for some $b_0 > 0$ which we take to be $b_0 = 1$ later:
\begin{align*}
\widehat{M}_0(b_0) := \{(b, \rho) \in \R \times \widehat{M}: |b| > b_0 \text{ or } \rho \neq 1\}.
\end{align*}
We fix the related constant $\delta_{\varsigma} := \inf_{b \in \R, \rho \in \widehat{M}_0} \|\rho_b\| = \inf_{\rho \in \widehat{M}_0} \|\rho\|$ which is positive because $M$ is a compact connected Lie group (recall from \cite[Example 3.1.4]{Lub10} that compact Lie groups have property (T)). Then, $\inf_{(b, \rho) \in \widehat{M}_0(b_0)} \|\rho_b\| \geq \min\{b_0, \delta_{\varsigma}\}$. Thus, we also fix the constant
\begin{align}
\label{eqn:Constantdelta1varrho}
\delta_{1, \varsigma} := \min\{1, \delta_{\varsigma}\}.
\end{align}

We make the following definition. Here and elsewhere, we use the subscript $\mathrm{b}$ to restrict to bounded functions. We also use the notation $\xi = a + ib \in \C$ throughout the paper and omit writing the decomposition explicitly.

\begin{definition}[Congruence transfer operator with holonomy]
For all $\xi \in \mathbb C$ with $a > -\epsilon_0$ and $\rho \in \widehat{M}$, the \emph{congruence transfer operator with holonomy} $\mathcal{M}_{\xi\Ret, \mathfrak{q}, \rho}: C_{\mathrm{b}}\bigl(\Lambda_+, L^2(F_{\mathfrak{q}}) \otimes V_\rho^{\oplus \dim(\rho)}\bigr) \to C_{\mathrm{b}}\bigl(\Lambda_+, L^2(F_{\mathfrak{q}}) \otimes V_\rho^{\oplus \dim(\rho)}\bigr)$ is defined by
\begin{align*}
\mathcal{M}_{\xi\Ret, \mathfrak{q}, \rho}(H)(x) &= \sum_{\gamma \in \gen} e^{-\xi\Ret(\gamma x)} \bigl(\gamma^{-1} \otimes \rho(\Hol(\gamma x)^{-1})\bigr) H(\gamma x) \\
&= \sum_{\gamma \in \gen} e^{-a\Ret(\gamma x)} \bigl(\gamma^{-1} \otimes \rho_b(\GHol(\gamma x)^{-1})\bigr) H(\gamma x)
\end{align*}
for all $x \in \Lambda_+$ and $H \in C_{\mathrm{b}}\bigl(\Lambda_+, L^2(F_{\mathfrak{q}}) \otimes V_\rho^{\oplus \dim(\rho)}\bigr)$.
\end{definition}

The above is well-defined due to the exponential tail property (see Property~(5) in \cref{prop:Coding}) and boundedness of the argument functions. Denote $\mathcal{M}_{\xi\Ret, \mathfrak{q}} \coloneq \mathcal{M}_{\xi\Ret, \mathfrak{q}, 1}$ and $\mathcal{L}_{\xi\Ret} \coloneq \mathcal{M}_{\xi\Ret, \mathcal{O}_\mathbb{K}}$.

Let $(V, \|\bigcdot\|)$ be any normed vector space over $\R$ or $\C$. Let $d$ be any distance function on $\Delta_0$; in particular, $d = d_{\mathrm{E}}$ or $d = D$. For any function $H: \Lambda_+\to V$, denote
\begin{align*}
\|H\|_{\infty} &= \sup\{\|H(x)\|: x\in \Lambda_+\},\\
\Lip_d(H)&=\sup\left\{\frac{\|H(x) - H(x')\|}{d(x, x')}: x, x'\in \Lambda_+, x\neq x'\right\},\\
\|H\|_{\Lip(d)}&=\|H\|_{\infty}+\Lip_d(H).
\end{align*}
Denote by $\Lip_d(\Lambda_+, V)$ the space of functions $H:\Lambda_+\to V$ with $\|H\|_{\Lip(d)} < +\infty$. We omit $d$ from the above notations if $d = d_{\mathrm{E}}$. In particular, we will work with the function spaces $\Lip\bigl(\Lambda_+, V_\rho^{\oplus \dim(\rho)}\bigr)$ and $\Lip_D(\Lambda_+, \R)$ corresponding to the normed vector spaces $\bigl(V_\rho^{\oplus \dim(\rho)}, \|\bigcdot\|_2\bigr)$ for some $\rho \in \widehat{M}$ and $(\R, |\bigcdot|)$. For any function $H:\Delta_0\to V$, denote
\begin{equation*}
\Lip_d^{\mathrm{e}}(H)=\sup\left\{\frac{\|H(x) - H(x')\|}{d(x, x')}: x, x'\in \Delta_j, x\neq x', j \in \calA\right\}.
\end{equation*}

Define the PS measure $\mu_{\mathrm{E}}$ on $\Delta_0$ with respect to the Euclidean metric by
\begin{equation}
\label{equ:mu E}
d\mu_{\mathrm{E}}(x)=(1+\|x\|^2)^{\delta_\Gamma} \, d\mu(x).
\end{equation}
Note that $\Lambda_+ \subset \Delta_0$ is a full measure subset with respect to $\mu_{\mathrm{E}}$. Using the quasi-invariance of the PS measure $\mu$, a straightforward computation gives $\mathcal{L}_{\delta_\Gamma \Ret}^*(\mu_{\mathrm{E}}) = \mu_{\mathrm{E}}$.

By \cite[Lemma A.1]{LPS25}, the family $\xi\mapsto \mathcal{L}_{(\delta_\Gamma+\xi)\Ret}$ of operators on $\Lip(\Lambda_+,\mathbb{C})$ is analytic on $\{\xi=a+ib\in \mathbb{C}: a>-\frac{\epsilon_0}{2}\}$. It can be shown as in \cite[Proposition A]{You98} that $\mathcal{L}_{\delta_\Gamma\Ret}|_{\Lip(\Lambda_+,\mathbb{C})}$ has a spectral gap with a maximal simple eigenvalue $\lambda_0 = 1$ with a corresponding eigenfunction $h_0:\Lambda_+\to \R$ given by $h_0(x)=(1+\|x\|^2)^{-\delta_\Gamma}f_0(x)$, where $f_0$ is the density function defined in \cref{eqn:nu_and_mu}. It satisfies $\int_{\Lambda_+} h_0 \, d\mu_{\mathrm{E}} = 1$. Moreover, by perturbation theory of operators (see \cite[Chapter 7]{Kat95}), there exists $a_0' \in (0, \frac{\epsilon_0}{2})$ and analytic maps
\begin{itemize}
\item $[-a_0', a_0'] \to \R$ denoted by $a \mapsto \lambda_a$,
\item $[-a_0', a_0'] \to \Lip(\Lambda_+, \R)$ denoted by $a \mapsto h_a$,
\end{itemize}
such that $\mathcal{L}_{(\delta_\Gamma + a)\Ret}h_a = \lambda_a h_a$, and $h_a$ is bounded away from $0$ and $+\infty$ and normalized such that $\int_{\Lambda_+} h_a \, d\mu_{\mathrm{E}}=1$.

Recall the measures $\mu$ and $\nu$ from \cref{subsec:Patterson--SullivanDensity,subsec:expanding map}, respectively, and also \cref{eqn:nu_and_mu}. Combining with \cref{equ:mu E} and the definition of $h_0$, we have
\begin{equation}
d\nu = h_0 \, d\mu_{\mathrm{E}}.
\end{equation}

Define the function
\begin{align*}
\FRet^{(a)} = -(\delta_\Gamma + a)\Ret - \log(\lambda_a) + \log \circ h_a - \log \circ h_a \circ T
\end{align*}
which is cohomologous to $-(\delta_\Gamma + a)\Ret - \log(\lambda_a)$. Due to Property~(4) in \cref{prop:Coding}, we can fix some $a_0' > 0$ and
\begin{align}
\label{eqn:ConstantC1'C2}
C_1' &> \max\left(\Lip^{\mathrm{e}}(\Ret), \Lip^{\mathrm{e}}(\GHol), \sup_{|a| \leq a_0'} \Lip^{\mathrm{e}}\bigl(\FRet^{(a)}\bigr)\right), & C_2 &= \frac{C_1'}{1 - \theta},
\end{align}
where $\Lip^{\mathrm{e}}(\GHol)$ is defined similarly using the Riemannian metric on $AM$.

For all $\xi \in \mathbb C$ with $a > -\epsilon_0$ and $\rho \in \widehat{M}$, we \emph{normalize} the transfer operator with holonomy as
\begin{align*}
\mathcal{M}_{\xi, \mathfrak{q}, \rho} = m_{\lambda_a h_a}^{-1} \circ \mathcal{M}_{(\delta_\Gamma + \xi)\Ret, \mathfrak{q}, \rho} \circ m_{h_a}
\end{align*}
where $m_h: C\bigl(\Lambda_+, L^2(F_{\mathfrak{q}}) \otimes V_\rho^{\oplus \dim(\rho)}\bigr) \to C\bigl(\Lambda_+, L^2(F_{\mathfrak{q}}) \otimes V_\rho^{\oplus \dim(\rho)}\bigr)$ denotes the multiplication operator by $h \in C(\Lambda_+, \R)$. For all $k \in \N$, its $k$-th iteration is simply given by
\begin{align}
\label{eqn:k^thIterationOfCongruenceTransferOperatorOfType_rho}
\mathcal{M}_{\xi, \mathfrak{q}, \rho}^k(H)(x) = \sum_{\gamma \in \gen^{(k)}} e^{\FRet_k^{(a)}(\gamma x)} \bigl(\gamma^{-1} \otimes \rho_b(\GHol^k(\gamma x)^{-1})\bigr) H(\gamma x)
\end{align}
for all $x \in \Lambda_+$ and $H \in C_{\mathrm{b}}\bigl(\Lambda_+, L^2(F_{\mathfrak{q}}) \otimes V_\rho^{\oplus \dim(\rho)}\bigr)$. Denote $\mathcal{M}_{\xi, \mathfrak{q}} \coloneq \mathcal{M}_{\xi, \mathfrak{q}, 1}$ and  $\calL_\xi \coloneq \calM_{\xi, \mathcal{O}_\mathbb{K}}$. Due to the above normalization, it satisfies $\cal{L}_0^*(\nu) = \nu$.

We now state the main technical theorem regarding uniform spectral bounds for the congruence transfer operators with holonomy.
We treat the different regions of the unitary dual $\R \times \widehat{M}$ separately.
The first should be thought of as estimates for low-frequence functions and the second as estimates for high-frequency functions. The proofs for each of these theorems use completely different techniques: expansion machinery for the first, and Dolgopyat's method for the second.

\begin{theorem}
\label{thm:UniformSpectralBoundI}
If $n = 3$, assume $\K = \Q$. Under \cref{assumption}, there exist $\eta > 0$, $C \geq 1$, $a_0 > 0$, $b_0 > 0$, and a nontrivial proper ideal $\mathfrak{q}_0' \subset \calO_\K$ such that for all $\xi \in \C$ with $|a| < a_0$ and $|b| \leq b_0$, (square-free if $n = 3$) ideals $\mathfrak{q} \subset \calO_\K$ coprime to $\mathfrak{q}_0\mathfrak{q}_0'$, $k \in \N$, and $H \in \Lip\bigl(\Lambda_+, L_0^2(F_\mathfrak{q})\bigr)$, we have
\begin{align*}
\bigl\|\mathcal{M}_{\xi, \mathfrak{q}}^k(H)\bigr\|_2 \leq C N_\K(\mathfrak{q})^C e^{-\eta k} \|H\|_{\Lip}.
\end{align*}
\end{theorem}

\begin{theorem}
\label{thm:UniformSpectralBoundII}
There exist $\eta > 0$, $C \geq 1$, $a_0 > 0$, and $b_0 > 0$ such that for all $\xi \in \C$ with $|a| < a_0$, if $(b, \rho) \in \widehat{M}_0(b_0)$, then for all nontrivial ideals $\mathfrak{q} \subset \calO_\K$, $k \in \N$, and $H \in \Lip\bigl(\Lambda_+, L^2(F_\mathfrak{q}) \otimes V_\rho^{\oplus \dim(\rho)}\bigr)$, we have
\begin{align*}
\bigl\|\mathcal{M}_{\xi, \mathfrak{q}, \rho}^k(H)\bigr\|_2 \leq C e^{-\eta k} \|H\|_{1, \|\rho_b\|}.
\end{align*}
\end{theorem}

\subsection{Outline of the proof of \cref{thm:UniformExponentialMixing}}
In the following, we outline the proof of \cref{thm:UniformExponentialMixing}. We avoid writing a full proof since it is nearly a verbatim repetition of the derivation from \cite[\S 9]{LPS25} (cf. \cite{AGY06,SW21}).
\begin{enumerate}[label=Step \arabic*:]
\item \textit{Reduction for $F_\mathfrak{q}$-invariant functions.} For any function $\phi \in C^1(\Gamma_\mathfrak{q} \backslash G)$, we may decompose it as $\phi = \overline{\phi} + \phi_0$ where $\overline{\phi}$ is $F_\mathfrak{q}$-invariant, i.e., constant along $F_\mathfrak{q}$-fibers and $\phi_0$ has mean $0$ along $F_\mathfrak{q}$-fibers. In doing so for both functions $\phi, \psi \in C^1(\Gamma_\mathfrak{q} \backslash G)$, the correlation function for the pair $(\phi, \psi)$ in \cref{thm:UniformExponentialMixing} is reduced to the sum of correlation functions for the pairs $(\overline{\phi}, \overline{\psi})$ and $(\phi, \psi_0)$. For the first term, since $\overline{\phi}$ and $\overline{\psi}$ can be thought of as functions in $C^1(\Gamma \backslash G)$, we obtain the desired exponential error term from \cite{LPS25} with a constant factor $\#\tilde{G}_\mathfrak{q} \leq N_{\mathbb K}(\mathfrak{q})^{C}$ where $C > 0$ depends on $n$. It remains to treat the second term.
\item \textit{Integrating out the strong stable direction.} We invoke \cref{cor:factormap} to work with the symbolic model. We work with test functions $\Lambda^{\Ret, \mathfrak{q}, M} \to \R$ which are Lipschitz, $C^1$, $C^r$ (for sufficiently large $r \geq 0$), and $C^1$ in the $\Lambda_+$\nobreakdash-, $\Lambda_-$\nobreakdash-, $M$\nobreakdash-, and $\R$\nobreakdash-coordinates, respectively---a standard convolution argument guarantees the main theorem for $C^1$ or even H\"older test functions. Thanks to Step 1, we may assume that the second test function has $0$ mean along $F_\mathfrak{q}$-fibers. We first take the correlation function and subdivide it into integrals over flow boxes. We then use half of the flow time to contract the $\Lambda_-$-coordinate and approximate, with an exponential error term with a constant factor $N_{\mathbb K}(\mathfrak{q})^{C}$, the first test function by one which is independent of $\Lambda_-$. We can then integrate the second test function in the $\Lambda_-$-coordinate to obtain one which is independent of $\Lambda_-$ as well. It suffices to show uniform exponential decay for the new correlation function.
\item \textit{Laplace transform of the new correlation function.} Taking the Laplace transform of the new correlation function, we can derive the following formula in terms of the congruence transfer operators with holonomy (see \cite[\S 9]{LPS25} for the precise notation): for all $\phi, \psi \in C\bigl(\Lambda_+^\Ret\times F_\mathfrak{q} \times M\bigr)$ and $\xi \in \mathbb C$ with $a > 0$, we have
	\begin{align}
		\label{eqn: correlation function formula}
		\hat{\Upsilon}_{\phi, \psi}^0(\xi) = \sum_{k = 1}^\infty \sum_{\rho \in \widehat{M}} \lambda_a^k \left\langle \hat{\phi}_{\xi, \rho}, \mathcal{M}_{\overline{\xi}, \mathfrak{q}, \rho}^k\big(\hat{\psi}_{-\overline{\xi}, \rho}\big) \right\rangle.
	\end{align}

\item \textit{Payley--Wiener theorem.} The spectral bounds from \cref{thm:UniformSpectralBoundII} imply that the sum $\sum_{k = 1}^\infty \sum_{\rho \in \widehat{M}_0(b_0)}$ from \cref{eqn: correlation function formula} can be extended holomorphically to a strip to the left of the critical line (which is $i\R$ in our parametrization) \emph{uniformly} in the ideals $\mathfrak{q} \subset \mathcal{O}_\mathbb{K}$. Now we restrict to $|b| \leq b_0$ and $\rho = 1$ in \cref{eqn: correlation function formula}. By our assumption stemming from Step 1, the argument of the congruence transfer operator with holonomy is in $\Lip\bigl(\Lambda_+, L_0^2(F_\mathfrak{q})\bigr)$. So again, the spectral bounds from \cref{thm:UniformSpectralBoundI} imply that the remaining sum $\sum_{k = 1}^\infty$ from \cref{eqn: correlation function formula}, can be extended holomorphically to a uniform strip to the left of the critical line. Now via the inverse Laplace transform, a Paley--Wiener type analysis gives the desired uniform exponential decay while carrying through the factor $N_\K(\mathfrak{q})^C$.
\end{enumerate}

Therefore, it suffices to focus on the proofs of \cref{thm:UniformSpectralBoundI,thm:UniformSpectralBoundII} in the rest of the paper.

\section{Reduction of \cref{thm:UniformSpectralBoundI}}
\label{sec:Reduction}
In this section, we state \cref{thm:UniformSpectralBoundIReduced}. Using the concept of new invariant functions, \cref{thm:UniformSpectralBoundI} can be reduced to \cref{thm:UniformSpectralBoundIReduced} (cf. \cite{OW16}). Since the proofs are nearly identical to those in \cite{Sar22a}, we omit them and refer the reader to loc. cit.

Let $\mathfrak{q} \subset \mathfrak{q}' \subset \mathcal{O}_{\mathbb K}$ be ideals. We have the reduction map and the induced pull back
\begin{align*}
\pi_{\mathfrak{q}, \mathfrak{q}'}&: \tilde{\mathbf{G}}(\mathcal{O}_{\mathbb K}/\mathfrak{q}) \to \tilde{\mathbf{G}}(\mathcal{O}_{\mathbb K}/\mathfrak{q}'), \\
\pi_{\mathfrak{q}, \mathfrak{q}'}^*&: L^2(\tilde{\mathbf{G}}(\mathcal{O}_{\mathbb K}/\mathfrak{q}')) \to L^2(\tilde{\mathbf{G}}(\mathcal{O}_{\mathbb K}/\mathfrak{q})).
\end{align*}
Define
\begin{align*}
\hat{E}_{\mathfrak{q}'}^\mathfrak{q} &= \pi_{\mathfrak{q}, \mathfrak{q}'}^*(L^2(\tilde{\mathbf{G}}(\mathcal{O}_{\mathbb K}/\mathfrak{q}'))), & \dot{E}_{\mathfrak{q}'}^\mathfrak{q} &= \hat{E}_{\mathfrak{q}'}^\mathfrak{q} \cap \left(\bigoplus_{\mathfrak{q}' \subsetneq \mathfrak{q}''} \hat{E}_{\mathfrak{q}''}^\mathfrak{q}\right)^\perp.
\end{align*}
Then, we have the orthogonal decomposition
\begin{align*}
L_0^2(\tilde{\mathbf{G}}(\mathcal{O}_{\mathbb K}/\mathfrak{q})) = \bigoplus_{\mathfrak{q} \subset \mathfrak{q}' \subsetneq \mathcal{O}_{\mathbb K}} \dot{E}_{\mathfrak{q}'}^\mathfrak{q} \qquad \text{for all ideals $\mathfrak{q} \subset \mathcal{O}_{\mathbb K}$}.
\end{align*}
Using the induced reduction map $\overline{\pi_{\mathfrak{q}, \mathfrak{q}'}}: \tilde{G}_\mathfrak{q} \to \tilde{G}_{\mathfrak{q}'}$, we also have the orthogonal decomposition
\begin{align*}
L_0^2(\tilde{G}_\mathfrak{q}) = \bigoplus_{\mathfrak{q} \subset \mathfrak{q}' \subsetneq \mathcal{O}_{\mathbb K}} E_{\mathfrak{q}'}^\mathfrak{q} \qquad \text{for all ideals $\mathfrak{q} \subset \mathcal{O}_{\mathbb K}$}.
\end{align*}
For all $\phi \in L^2(\tilde{G}_\mathfrak{q})$, define $\dot{\phi} \in L^2(\tilde{\mathbf{G}}(\mathcal{O}_{\mathbb K}/\mathfrak{q}))$ by
\begin{align*}
\dot{\phi}(g) = \phi(\pi_{\mathfrak{q}}(\ker(\pi))g) \qquad \text{for all $g \in \tilde{\mathbf{G}}(\mathcal{O}_{\mathbb K}/\mathfrak{q})$}.
\end{align*}
Then $E_{\mathfrak{q}'}^\mathfrak{q} = \{\phi \in L^2(\tilde{G}_\mathfrak{q}): \dot{\phi} \in \dot{E}_{\mathfrak{q}'}^\mathfrak{q}\}$. Suppose $\mathfrak{q}$ is coprime to $\mathfrak{q}_0$ so that we may use the isomorphism $\overline{\pi_\mathfrak{q}|_{\tilde{\Gamma}}}: F_\mathfrak{q} \to \tilde{G}_\mathfrak{q}$. We see that the subspace $E_{\mathfrak{q}'}^\mathfrak{q} \subset L^2(\tilde{G}_\mathfrak{q})$ consists of all \emph{new} functions which are invariant under $\Gamma_{\mathfrak{q}'}$ but not invariant under $\Gamma_{\mathfrak{q}''}$, for any $\mathfrak{q}' \subsetneq \mathfrak{q}''$. We also have the orthogonal decomposition
\begin{align*}
\Lip(\Lambda_+, L^2(F_\mathfrak{q})) = \bigoplus_{\mathfrak{q} \subset \mathfrak{q}' \subsetneq \mathcal{O}_{\mathbb K}} \Lip\bigl(\Lambda_+, E_{\mathfrak{q}'}^\mathfrak{q}\bigr) \qquad \text{for all ideals $\mathfrak{q} \subset \mathcal{O}_{\mathbb K}$}.
\end{align*}
It turns out that the congruence transfer operator $\mathcal{M}_{\xi, \mathfrak{q}}$ preserves $\Lip\bigl(\Lambda_+, E_{\mathfrak{q}'}^\mathfrak{q}\bigr)$ for all $\xi \in \mathbb C$ and the following theorem implies \cref{thm:UniformSpectralBoundI}.

\begin{theorem}
\label{thm:UniformSpectralBoundIReduced}
There exist $\kappa \in (0, 1)$, $C > 0$, $a_0 > 0$, $b_0 > 0$, and $q_1 \in \mathbb N$ such that for all $\xi \in \mathbb C$ with $|a| < a_0$ and $|b| \leq b_0$, (square-free if $n = 3$) ideals $\mathfrak{q} \subset \calO_\K$ coprime to $\mathfrak{q}_0$ with $N_\K(\mathfrak{q}) > q_1$, there exists an integer $s \in (0, C\log(N_\K(\mathfrak{q})))$ such that for all $j \in \mathbb Z_{\geq 0}$ and $H \in \Lip\bigl(\Lambda_+, E_\mathfrak{q}^\mathfrak{q}\bigr)$, we have
\begin{align*}
\big\|\mathcal{M}_{\xi, \mathfrak{q}}^{js}(H)\big\|_{\Lip} \leq N_\K(\mathfrak{q})^{-j\kappa} \|H\|_{\Lip}.
\end{align*}
\end{theorem}

\section{Zariski density and trace field of the return trajectory subgroups}
\label{sec:ZariskiDensityAndTraceFieldOfTheReturnTrajectorySubgroups}
In this section we recall the return trajectory subgroups as defined in \cite{Sar22a}, with slight modifications to adapt it to the current setting, and prove that they are Zariski dense and have full trace field $\K$. This will be required in \cref{sec:ProofOfUniformSpectralBoundIReducedViaExpansionMachinery} in order to use the strong approximation theorem together with the expansion machinery of Golsefidy--Varj\'{u} \cite{GV12} and He--de Saxc\'{e} \cite{HdS22}.

\begin{definition}[Return trajectory subgroup]
Recalling \cref{eqn: cocycle for a sequence}, we define the \emph{return trajectory subgroup} to be the subgroup $\Omega < \Gamma$ generated by the subset
\begin{align*}
S = \{\gamma(\alpha) \cdot \gamma(\tilde{\alpha})^{-1}: \alpha, \tilde{\alpha} \text{ are sequences of any length } k \in \N\} \subset \Gamma.
\end{align*}
\end{definition}

We denote $\tilde{S} = \tilde{\pi}^{-1}(S) \subset \tilde{\Gamma}$ and $\tilde{\Omega} = \langle \tilde{S} \rangle = \tilde{\pi}^{-1}(\Omega) < \tilde{\Gamma}$. We will prove the following theorem.

\begin{theorem}
\label{thm:Z-DenseInSimplyConnectedCoverGAndTraceFieldK}
For the subgroup $\tilde{\Omega} < \tilde{\Gamma}$, the following holds:
\begin{enumerate}
\item $\tilde{\Omega}$ is Zariski dense in $\tilde{\mathbf{G}}$;
\item $\K_{\tilde{\Omega}} = \K$.
\end{enumerate}
\end{theorem}

Before working towards the proof, we introduce a few useful tools and deduce a corollary of the above theorem. The following lemma is a well-known; for a proof of the reverse direction, see for instance \cite[Lemma 7.4]{Sar25}, and the forward direction follows from \cite[Proposition 3.12]{Win15}.

\begin{lemma}
\label{lem: Zariski dense condition}
A subgroup $\Gamma' \leq G$ or $\Gamma' \leq \tilde{G}$ is Zariski dense if and only if $\Lambda(\Gamma') \subset \partial_\infty\HS$ is not contained in any $(n - 2)$-sphere in $\partial_\infty\HS$.
\end{lemma}

Though it can be proved by purely algebro-geometric arguments, we have the following corollary.

\begin{corollary}
\label{cor: finitely generated Zariski dense subgroup}
For any Zariski dense subgroup $\Gamma' \leq \tilde{G}$, there exists a finitely generated Zariski dense subgroup $\Gamma_{\mathrm{fin}}' \leq \Gamma' \leq \tilde{G}$.
\end{corollary}

\begin{proof}
Let $\Gamma' \leq \tilde{G}$ be Zariski dense. Recall that $\Lambda(\Gamma')$ can be obtained as the closure of the set of hyperbolic fixed points of $\Gamma'$. By \cref{lem: Zariski dense condition}, we may take a generating subset $S_{\mathrm{fin}}' \subset \Gamma'$ consisting of $n + 1$ hyperbolic elements whose attracting fixed points are not contained in any $(n - 2)$-sphere in $\partial_\infty\HS$. Then again by \cref{lem: Zariski dense condition}, $\Gamma_{\mathrm{fin}}' \coloneq \langle S_{\mathrm{fin}}' \rangle \leq \Gamma' \leq \tilde{G}$ is Zariski dense.
\end{proof}

The following is a corollary of \cref{thm:Z-DenseInSimplyConnectedCoverGAndTraceFieldK}.

\begin{corollary}
\label{cor:Z-DenseInSimplyConnectedCoverGAndTraceFieldK}
There exists a subgroup $\tilde{\Omega}_{\mathrm{fin}} = \langle \tilde{S}_{\mathrm{fin}} \rangle < \tilde{\Omega} < \tilde{\Gamma}$ generated by the finite symmetric subset $\tilde{S}_{\mathrm{fin}} \subset \tilde{S} \subset \tilde{\Gamma}$ such that the following holds:
\begin{enumerate}
\item $\tilde{\Omega}_{\mathrm{fin}}$ is Zariski dense in $\tilde{\mathbf{G}}$;
\item $\K_{\tilde{\Omega}_{\mathrm{fin}}} = \K$.
\end{enumerate}
\end{corollary}

\begin{proof}
We start with the properties provided by \cref{thm:Z-DenseInSimplyConnectedCoverGAndTraceFieldK}. By \cref{cor: finitely generated Zariski dense subgroup}, there exists a finite generating subset $\tilde{S}_{\mathrm{fin}}' \subset \tilde{\Omega}$ and a subgroup $\tilde{\Omega}_{\mathrm{fin}}' \coloneq \langle \tilde{S}_{\mathrm{fin}}' \rangle < \Gamma$ which is Zariski dense in $\tilde{\mathbf{G}}$. Now, we can simply take a finite subset $\tilde{S}_{\mathrm{fin}} \subset \tilde{S} \subset \Gamma$ such that
\begin{itemize}
\item $\tilde{S}_{\mathrm{fin}}$ is symmetric, i.e., $\gamma \in \tilde{S}_{\mathrm{fin}}$ if and only if $\gamma^{-1} \in \tilde{S}_{\mathrm{fin}}$;
\item $\tilde{S}_{\mathrm{fin}}' \subset \langle \tilde{S}_{\mathrm{fin}}\rangle$;
\item $\Q(\tr(\Ad(\langle\tilde{S}_{\mathrm{fin}}\rangle))) = \K$.
\end{itemize}
Finally, taking $\tilde{\Omega}_{\mathrm{fin}} \coloneq \langle \tilde{S}_{\mathrm{fin}} \rangle < \Gamma$, it is clear that the required Zariski density and full trace field conditions are satisfied.
\end{proof}

Our goal is now to prove \cref{thm:Z-DenseInSimplyConnectedCoverGAndTraceFieldK}. As in \cite{Sar22a}, we need to study the limit set of $\tilde{\Omega}$. However, for $\Gamma$ geometrically finite with parabolic elements and $\K$ strictly larger than $\Q$, this is not sufficient. In that case, we will also need to introduce further tools from the works of Susskind--Swarup \cite{SS92} and another machinery of Prasad--Rapinchuk \cite{PR09}.

We prepare by first fixing some notations for the rest of the section in addition to those in \cref{sec:Preliminaries}. We will work in the upper half space model. We will also use the isometry $\T^1(\HS) \cong \HS \times \R^n$ and denote by $\pi_{\HS}: \HS \times \R^n \to \HS$ the tangent bundle projection map.
Let $\pi_{\R^{n - 1}}: \R^n \to \R^{n - 1}$ be the orthogonal projection onto $\R^{n - 1} = \langle e_1, e_2, \dotsc, e_{n - 1}\rangle$. Let $B^{\mathrm{E}}_\epsilon(u) \subset \R^{n - 1} \subset \partial_\infty\HS$ denote the open Euclidean ball of radius $\epsilon > 0$ centered at $u \in \R^{n - 1}$. We reserve the notation $B_\epsilon(Q) \subset \HS$, where $\epsilon > 0$, for the $\epsilon$-neighborhood of a subset $Q \subset \HS$ with respect to the hyperbolic metric. Denote by $d_{\mathrm{ss}}$ (resp. $d_{\mathrm{su}}$) the distance function induced by the Riemannian metric on any strong stable (resp. strong unstable) leaf in $\T^1(\HS)$. Accordingly, we decorate balls and neighborhoods of sets by superscript $\mathrm{ss}$ (resp. $\mathrm{su}$). Define the map $\vis^\pm: \T^1(\HS) \to \partial_\infty \HS$ by $\vis^\pm(u) = u^\pm$ for all $u \in \T^1(\HS)$. Consider the set of unit tangent vectors
\begin{align*}
V = \{(u, -e_n) \in \T^1(\HS): \langle u, e_n \rangle = 1\}
\end{align*}
perpendicular to a horosphere. Setting $C_{\mathrm{E}} > 0$ to be the image of the constant map $V \to \R$ defined by $u \mapsto \|(d\vis^+)_u\|_{\mathrm{op}}$, we have
\begin{align*}
\frac{1}{C_{\mathrm{E}}}d_{\mathrm{su}}(u, v)\leq \|u^+ - v^+\| \leq C_{\mathrm{E}}d_{\mathrm{su}}(u, v) \qquad \text{for all $u, v \in V$}.
\end{align*}

Denote $\mathsf{Z} := \supp\bigl(m^{\mathrm{BMS}}\bigr) \subset \T^1(\HS)$. Recall the map $\tilde{\Phi}$ from \cref{subsec:SymbolicModelForFrameFlows}. For all $j \in \calA$, denote
\begin{align*}
\mathsf{S}_j &\coloneq \tilde{\Phi}(\{\gamma_jx_0\} \times \overline{\Lambda_-} \times \{0\}), & \hat{\mathsf{S}}_j &\coloneq \tilde{\Phi}(\{\gamma_jx_0\} \times \Lambda_- \times \{0\}), \\
\mathsf{U}_j &\coloneq \tilde{\Phi}(\overline{\Delta_j} \times \{\infty\} \times \{0\}), & \hat{\mathsf{U}}_j &\coloneq \tilde{\Phi}((\Lambda_+ \cap \Delta_j) \times \{\infty\} \times \{0\}).
\end{align*}
Denote the unions $\mathsf{U} := \bigsqcup_{j \in \calA} \mathsf{U}_j$, $\mathsf{U}_\sqcup := \bigsqcup_{j \in \calA} \interior(\mathsf{U}_j)$, and $\mathsf{S} := \bigsqcup_{j \in \calA} \mathsf{S}_j$. For all $\gamma \in \Gamma$, we define the map $[\bigcdot, \bigcdot]: \gamma \mathsf{U} \times \gamma \mathsf{S} \to \T^1(\HS)$ by
\begin{align*}
[\gamma u, \gamma s] = \gamma \tilde{\Phi}(x, y, 0)
\end{align*}
for all $u = (x, \infty, 0) \in \mathsf{U}$ and $s = (\gamma_jx_0, y, 0) \in \mathsf{S}_j$ for some $j \in \calA$. For all $j \in \calA$, denote the rectangles and their subsets
\begin{align*}
\mathsf{R}_j &:= [\mathsf{U}_j, \mathsf{S}_j], & \hat{\mathsf{R}}_j &:= [\hat{\mathsf{U}}_j, \hat{\mathsf{S}}_j];
\end{align*}
the former are lifts of $R_j$ from \cref{subsec:SymbolicModelForFrameFlows}. Denote the unions $\mathsf{R} := \bigsqcup_{j \in \calA} \mathsf{R}_j$, $\mathsf{R}_\sqcup := \bigsqcup_{j \in \calA} [\interior(\mathsf{U}_j), \hat{\mathsf{S}}_j]$, and $\hat{\mathsf{R}} := \bigsqcup_{j \in \calA} \hat{\mathsf{R}}_j$. We generalize the previous notation: for all $\gamma \in \Gamma$, define $[v_1, v_2] = [u_1, s_2]$ for all $v_1 = [u_1, s_1] \in \gamma \mathsf{R}$ and $v_2 = [u_2, s_2] \in \gamma \mathsf{R}$. We define the trajectory isomorphism $\psi$ analogous to \cite[Definition 1.1]{Rat73} in the following fashion: for all $u, u' \in [\mathsf{U}, u]$, we have $u' = \psi_u^{-1}(u')a_t$ for some unique point $\psi_u^{-1}(u') \in \T^1(\HS)$ in the strong unstable leaf through $u$ and $t \in \R$. It will be convenient to abuse notation and define the return time map $\Ret: \Gamma\mathsf{R}_\sqcup \to \R$ by $\Ret(\gamma u) = \Ret(x)$, and the Poincar\'e map $\mathcal{P}: \Gamma\mathsf{R}_\sqcup \to \Gamma\interior(\mathsf{R})$ by
\begin{align*}
\mathcal{P}(\gamma u) = \mathcal{G}_{\Ret(\gamma u)}(\gamma u) \in \Gamma \tilde{\Phi}(\hat{T}(x, y), 0) \subset \Gamma \interior(\mathsf{R}),
\end{align*}
for all $\gamma \in \Gamma$ and $u = \tilde{\Phi}(x, y, 0) \in \mathsf{R}_\sqcup$. We use the same notation from \cref{eqn:BirkhoffSum}. Although $\Ret_k$ for $k \geq 0$ is not well-defined at all points $x \in \overline{\Delta_0}$ a priori, we make sense of it and extend the definition using the same formula from \cref{prop:Coding} if $x \in \gamma \overline{\Delta_0}$ for an inverse branch $\gamma \in \gen^{(k)}$ (not necessarily unique); in this case, we write $\Ret_\alpha$ where we specify the inverse branch with the sequence $\alpha$ of length $k$ corresponding to $\gamma$. Define the constant
\begin{align*}
\underline{\Ret} = \inf\{\Ret(x): x \in \Delta_0\} \in (0, +\infty).
\end{align*}

The following \cref{lem:ProducingElementsofH} gives a description of the elements in the generating set of the return trajectory subgroup using the Poincar\'{e} map.

\begin{lemma}
\label{lem:ProducingElementsofH}
Let $p \in \mathbb N$, and $h \in \Gamma$. If there exist $(y, z) \in \calA^2$, $g_z \in \langle \gen\rangle_\sg$, $u_0 \in \hat{\mathsf{R}}_y$, and $v_0 \in g_z\hat{\mathsf{R}}_z$ such that $\mathcal P^{p + 1}(u_0) \in g_z\hat{\mathsf{R}}_z$ and $\mathcal P^{-(p + 1)}(v_0) \in h\hat{\mathsf{R}}_y$, then $h \in S$.
\end{lemma}

\cref{lem:ProducingElementsofH} provides a procedure to produce elements in the return trajectory subgroup. Our strategy is to use this procedure to construct a sequence of orbit points whose limit is $0 \in \partial_\infty \HS$. We also ensure the crucial property that such a sequence can be constructed inside a cone, as depicted in \cref{fig:OrbitVectorCloseToOrigin}, to show that $0 \in \partial_\infty \HS$ is a radial limit point (see \cref{def:RadialLimitSet}).

We introduce a type of limit set
\begin{align*}
\Lambda_\infty(\langle \gen\rangle_\sg) &\coloneq \bigcap_{k \in \N} \bigcup_{\gamma \in \gen^{(k)}} \gamma \cdot \overline{\Delta_0} \\
&\subsetneq \Lambda(\langle \gen\rangle_\sg) \subsetneq \Lambda(\Gamma).
\end{align*}
Note the above containment as in \cref{eqn: + limit set}. Note also that $\Lambda_+ \subsetneq \Lambda_\infty(\langle \gen\rangle_\sg)$. Intuitively, we think of it as the set of limit points of $\langle \gen\rangle_\sg$ associated to infinite sequences/words. We will show a stronger containment and make the previous intuition precise in the proposition below. Here, we set the shift space
\begin{align*}
\Sigma^+ \coloneq \calA^\N
\end{align*}
endowed with:
\begin{itemize}
\item the metric $d_\theta$, where $\theta \in (0, 1)$ is the constant from \cref{prop:Coding}, defined by
\begin{align*}
	d_\theta(x, y) = \theta^{\inf\{j \in \N: x_j \neq y_j\}} \qquad \text{for all $x, y \in \Sigma^+$},
\end{align*}
\item and the shift operator $\sigma: \Sigma^+ \to \Sigma^+$.
\end{itemize}

\begin{proposition}
\label{prop: infty limit set has infty sequence}
The following holds.
\begin{enumerate}
\item There exists a Lipschitz continuous surjective map $\zeta: \Sigma^+ \to \Lambda_\infty(\langle \gen\rangle_\sg)$ which is injective on $\zeta^{-1}(\Lambda_+)$ and defines a semiconjugacy between the shift operator $\sigma: \Sigma^+ \to \Sigma^+$ and $T: \Lambda_+ \to \Lambda_+$, i.e.,
\begin{align*}
\zeta \circ \sigma = T \circ \zeta \qquad \text{on $\zeta^{-1}(\Lambda_+)$}.
\end{align*}
\item We have the containment $\Lambda_\infty(\langle \gen\rangle_\sg) \subset \Lambda_{\mathrm{r}}(\langle \gen\rangle_\sg)$.
\end{enumerate}
\end{proposition}

\begin{proof}
We simply define the map $\zeta: \Sigma^+ \to \Lambda_\infty(\langle \gen\rangle_\sg)$ in the following fashion: for all $\alpha = (\alpha_1, \alpha_2, \dotsc) \in \Sigma^+$, we assign the point $\zeta(\alpha) \in \Lambda_\infty(\langle \gen\rangle_\sg)$ uniquely determined by
\begin{align*}
\{\zeta(\alpha)\} = \bigcap_{k \in \N} \gamma_{\alpha_1} \dotsb \gamma_{\alpha_k} \cdot \overline{\Delta_0} \subset \Lambda_\infty(\langle \gen\rangle_\sg).
\end{align*}
Indeed, the above intersection produces a singleton by the following standard argument. Denoting the closed cylinders
\begin{align*}
\overline{\mathtt{C}_k} \coloneq \gamma_{\alpha_1} \dotsb \gamma_{\alpha_k} \cdot \overline{\Delta_0} \qquad \text{for all $k \in \N$},
\end{align*}
we have $\overline{\mathtt{C}_k} \supset \overline{\mathtt{C}_l}$ for all $k, l \in \N$ with $k \leq l$, and $\diam(\overline{\mathtt{C}_k}) \to 0$ as $k \to +\infty$ by \cref{lem:CylinderEstimate} and Property~(2) of \cref{prop:Coding}. By compactness, $\bigcap_{k \in \N} \overline{\mathtt{C}_k}$ is a singleton. The properties of $\zeta$ in the proposition easily follow.

Let us now show that any $x \in \Lambda_\infty(\langle \gen\rangle_\sg)$ is a radial limit point of $\langle \gen\rangle_\sg$. Write $x = \zeta(\alpha)$ for some $\alpha \in \Sigma^+$.
By the above characterization of $\zeta(\alpha)$, we deduce that for any $y \in \overline{\H^n}$, we have
\begin{align*}
\gamma_{\alpha_1} \dotsb \gamma_{\alpha_k} \cdot y \to x \qquad \text{as $k \to +\infty$}.
\end{align*}
For any choice $y \in \H^n$, we claim that $\{\gamma_{\alpha_1} \dotsb \gamma_{\alpha_k} \cdot y\}_{k \in \N}$ is a radial sequence. Let us choose $y \coloneq \pi_{\H^n}(\tilde{\Phi}(x, \infty, 0))$. Write $\alpha_{[k]} := (\alpha_1, \dotsc, \alpha_k)$ and apply the geodesic flow for time $\Ret_{\alpha_{[k]}}(x, \infty)$. After taking basepoints, we then have
\begin{align*}
\pi_{\H^n}\Bigl(\mathcal{G}_{\Ret_{\alpha_{[k]}}(x, \infty)}(\tilde{\Phi}(x, \infty, 0))\Bigr) &= \gamma_{\alpha_1} \dotsb \gamma_{\alpha_k} \cdot \pi_{\H^n}(\tilde{\Phi}(\hat{T}^k(x, \infty), 0)) \\
&= \gamma_{\alpha_1} \dotsb \gamma_{\alpha_k} \cdot y_k
\end{align*}
where we write $y_k \coloneq \pi_{\H^n}(\tilde{\Phi}(\hat{T}^k(x, \infty), 0))$, for all $k \in \N$. Of course the hyperbolic distances between the points $y$ and $y_k$ are bounded over all $k \in \N$ since the points are all on the compact set $\tilde{\Phi}(\overline{\Delta_0} \times \{\infty\} \times \{0\})$ on the same unstable horosphere. Observing from the above identity that $\{\gamma_{\alpha_1} \dotsb \gamma_{\alpha_k} \cdot y_k\}_{k \in \N}$ is a sequence on the same geodesic, our claim is proven.
\end{proof}

We need the following lemma. The proof is nearly identical to the one in \cite{Sar22a} using the Markov property of $\mathcal{R}$.

\begin{lemma}
\label{lem:ChooseAFarAway_v_SuchThat0IsInView}
There exist a finite subset $\calA_0 \subset \calA$, and constants $\delta_1 > 0$, $\delta_2 > 0$, and $r \in \mathbb N$, such that the following holds. Let $z \in \calA_0$. For all $u \in \mathsf{R}$, there exists $v \in [\hat{\mathsf{U}}_z, u]$ such that
\begin{enumerate}
\item $v = \tilde{\Phi}(x, \bullet, 0)$ where $x = \zeta(\alpha)$ and $\alpha \in \Sigma^+$ is a $\sigma$-periodic sequence with period $r$ composed only of letters in $\calA_0$ with the first letter being $z$;
\item\label{itm:ChooseAFarAway_v} $d_{\mathrm{su}}(u, \psi_u^{-1}(v)) > \delta_1$;
\item\label{itm:0IsInView} $\mathcal{P}^r(v) \in g_z[\hat{\mathsf{U}}_z, \mathsf{S}_z]$ such that $B_{\delta_2}^{\mathrm{ss}}(\mathcal{P}^r(v)) \cap \mathsf{Z} \subset [\mathcal{P}^r(v), g_z\mathsf{S}_z]$ for some $g_z \in \Gamma$ and if $g \in G$ such that $gv = (e_n, -e_n)$, then we have $\vis^-\bigl(B_{\delta_2}^{\mathrm{ss}}(g\mathcal{P}^r(v))\bigr) \supset \partial_\infty \HS \smallsetminus B_{C_{\mathrm{E}}^{-1}\delta_1e^{-\underline{\Ret}}}^{\mathrm{E}}(0)$.
\end{enumerate}
\end{lemma}

The following proposition formulates in a precise manor the relationship between the limit sets of $\langle \gen\rangle_\sg$ and $\Omega$. Unlike in \cite{Sar22a,Sar22b}, the situation in the geometrically finite case is more complicated---we do not simply prove that they are equal and rather prove containment between certain types of limits sets.

\begin{proposition}
\label{pro:LimitSetH(yy)RadialEqualsLimitSetH(yy)EqualsLimitSetGamma}
We have $\Lambda_\infty(\langle \gen\rangle_\sg) \subset \Lambda_{\mathrm{r}}(\Omega)$.
\end{proposition}

\begin{proof}
Let $u \in \mathsf{U}_y$ for some $y \in \calA$ be a tangent vector with an arbitrary forward limit point $u^+ \in \Lambda_\infty(\langle \gen\rangle_\sg)$. There exists $g \in G$ such that $gu = (e_n, -e_n)$ and $gu^+ = 0 \in \partial_\infty \HS$. Without loss of generality, we can assume $g = e \in G$. We will now construct a sequence of $\Omega$-orbit points of $\pi_{\HS}(u) = e_n \in \HS$ inside some cone whose limit is $u^+ = 0 \in \partial_\infty \HS$.
We will use the following key property: since $u^+ \in \Lambda_\infty(\langle \gen\rangle_\sg)$, by \cref{prop: infty limit set has infty sequence}, there exists an associated sequence $\alpha \in \Sigma^+$ such that $\zeta(\alpha) = u^+ = 0$.

Fix $\calA_0$, $\delta_1$, $\delta_2$, and $r \in \mathbb N$ from \cref{lem:ChooseAFarAway_v_SuchThat0IsInView}. Fix the constant
\begin{align*}
\overline{\Ret} = \sup\{\Ret(x): x \in \gamma_j\Delta_0, j \in \calA_0\} \in (0, +\infty);
\end{align*}
it is well-defined since $\calA_0$ is a finite set. Let $q \in \mathbb N$ and $z \in \calA_0$. Let $\mathtt{C}_q \subset \mathsf{U}_y$ be a cylinder and $\alpha_{[q]}$ be a corresponding sequence of length $q$ such that $u \in \overline{\mathtt{C}_q}$ and $\overline{\mathcal{P}^q(\mathtt{C}_q)} = \gamma[\mathsf{U}, s]$ for some $\gamma \in \Gamma$ and $s \in \mathsf{S}$. Note that $ua_{\Ret_{\alpha_{[q]}}(u)} \in \overline{\mathcal{P}^q(\mathtt{C}_q)}$, corresponding to which we can fix $u_q' \in \mathcal{P}^q(\mathtt{C}_q) \cap \gamma[\hat{\mathsf{U}}_z, \mathsf{S}_z]$ provided by \cref{lem:ChooseAFarAway_v_SuchThat0IsInView}. Let $u_q = \mathcal{P}^{-q}(u_q') \in \hat{\mathsf{U}}_y$. By \cref{lem:ChooseAFarAway_v_SuchThat0IsInView}, we have $\mathcal{P}^{q + r}(u_q) \in g_z[\hat{\mathsf{U}}_z, \mathsf{S}_z]$ for some $g_z \in \Gamma$. Fix the constant $\hat{\delta} := \sup\{\diam_{d_{\mathrm{su}}}([\mathsf{U}, s]): s \in \mathsf{S}\}$. Since $\|\pi_{\R^{n - 1}}(\pi_{\HS}(\mathcal{P}^{q + r}(u_q)))\| = \|\pi_{\R^{n - 1}}(\pi_{\HS}(u_q'))\|$, so $\pi_{\R^{n - 1}}(\pi_{\HS}(\mathcal{P}^{q + r}(u_q)))$ satisfies the bound
\begin{align}
\label{eqn:u_q^+_HorizontalCoordinateBounds}
\frac{1}{C_{\mathrm{E}}} \delta_1 e^{-\Ret_{\alpha_{[q]}}(u)} \leq \|\pi_{\R^{n - 1}}(\pi_{\HS}(\mathcal{P}^{q + r}(u_q)))\| \leq C_{\mathrm{E}}\hat{\delta} e^{-\Ret_{\alpha_{[q]}}(u)}
\end{align}
using $d_{\mathrm{su}}\bigl(ua_{\Ret_{\alpha_{[q]}}(u)}, \psi_{ua_{\Ret_{\alpha_{[q]}}(u)}}^{-1}(u_q')\bigr) > \delta_1$ from property~(2) of \cref{lem:ChooseAFarAway_v_SuchThat0IsInView}. Moreover, by property~(1) of \cref{lem:ChooseAFarAway_v_SuchThat0IsInView}, it lies at height
\begin{align*}
\langle \pi_{\HS}(\mathcal{P}^{q + r}(u_q)), e_n \rangle = e^{-\Ret_{q + r}(u_q)} = e^{-(\Ret_{\alpha_{[q]}}(u_q) + \Ret_r(u_q'))} \geq e^{-(\Ret_{\alpha_{[q]}}(u) + \underline{\Ret} + r\overline{\Ret})}.
\end{align*}
We then have the calculation
\begin{align*}
\frac{\langle \pi_{\HS}(\mathcal{P}^{q + r}(u_q)), e_n \rangle}{\|\pi_{\R^{n - 1}}(\pi_{\HS}(\mathcal{P}^{q + r}(u_q)))\|} \geq \frac{e^{-(\Ret_\alpha(u) + \underline{\Ret} + r\overline{\Ret})}}{C_{\mathrm{E}}\hat{\delta} e^{-\Ret_\alpha(u)}} = \frac{e^{-(\underline{\Ret} + r\overline{\Ret})}}{C_{\mathrm{E}}\hat{\delta}}
\end{align*}
which is a constant. Thus, $\mathcal{P}^{q + r}(u_q) \in \mathcal{C}_0$ where we define the cone
\begin{align*}
\mathcal{C}_0 = \{(\tilde{w}, w_n) \in \R^{n - 1} \times \R: w_n > 2C_{\mathrm{E}}^{-1}\hat{\delta}^{-1}e^{-(\underline{\Ret} + r\overline{\Ret})} \|\tilde{w}\|\}.
\end{align*}
Let $\iota \in G$ be a translation by $-\pi_{\R^{n - 1}}(\pi_{\HS}(u_q')) = -u_q^+$ followed by a dilation by a factor of $\langle \pi_{\HS}(u_q'), e_n \rangle^{-1} = e^{\Ret_{\alpha_{[q]}}(u_q)}$. Then, $\iota u_q' = (e_n, -e_n)$ and hence by property~(3) of \cref{lem:ChooseAFarAway_v_SuchThat0IsInView}, we have $B_{\delta_2}^{\mathrm{ss}}(\mathcal{P}^{q + r}(u_q)) \cap \mathsf{Z} \subset [\mathcal{P}^{q + r}(u_q), g_z\mathsf{S}_z]$ and $\vis^-\bigl(B_{\delta_2}^{\mathrm{ss}}(\iota \mathcal{P}^{q + r}(u_q))\bigr) \supset \partial_\infty \HS \smallsetminus B_{C_{\mathrm{E}}^{-1}\delta_1e^{-\underline{\Ret}}}^{\mathrm{E}}(0)$. Applying $\iota^{-1}$, we get
\begin{multline*}
\vis^-\bigl(B_{\delta_2}^{\mathrm{ss}}(\mathcal{P}^{q + r}(u_q))\bigr) \\
\supset \partial_\infty \HS \smallsetminus B_{C_{\mathrm{E}}^{-1}\delta_1 e^{-\Ret_{\alpha_{[q]}}(u_q) - \underline{\Ret}}}^{\mathrm{E}}(u_q^+) \\
\supset \partial_\infty \HS \smallsetminus B_{C_{\mathrm{E}}^{-1}\delta_1 e^{-\Ret_{\alpha_{[q]}}(u)}}^{\mathrm{E}}(u_q^+).
\end{multline*}
Recalling \cref{eqn:u_q^+_HorizontalCoordinateBounds}, we have $0 \in \vis^-\bigl(B_{\delta_2}^{\mathrm{ss}}(\mathcal{P}^{q + r}(u_q))\bigr)$. Also recalling $0, u_q^+ \in \Lambda(\Gamma)$, there exists $v_q' \in B_{\delta_2}^{\mathrm{ss}}(\mathcal{P}^{q + r}(u_q)) \cap \mathsf{Z} \subset [\mathcal{P}^{q + r}(u_q), g_z\mathsf{S}_z]$ with $(v_q')^+ = u_q^+$ and $(v_q')^- = 0$. Then, $\pi_{\HS}(v_q') \in B_{\hat{\delta}}(\mathcal{C}_0)$. It follows that $\gamma \subset B_{\hat{\delta}}(\mathcal{C}_0)$ for the geodesic ray $\gamma = \{\pi_{\HS}(v_q' a_{-t}) \in \HS: t \geq \underline{\Ret}\}$ because $(v_q')^- = 0$. Thus, there exists $v_q'' \in [\mathcal{P}^{q + r}(u_q), g_z\hat{\mathsf{S}}_z]$ such that $\pi_{\HS}(w_q) \in B_{\hat{\delta}}(\gamma) \subset B_{2\hat{\delta}}(\mathcal{C}_0)$ where $w_q = \mathcal{P}^{-(q + r - 1)}(v_q'') \in \gamma'\mathsf{R}$ for some $\gamma' \in \Gamma$. Then, there exists $w_q' \in \gamma'\hat{\mathsf{R}}_{z'}$ for some $z' \in \calA_0$ such that $\mathcal{P}^{- 1}(w_q') \in h_q\hat{\mathsf{R}}_y$ for some $h_q \in \Gamma$. Let $v_q = \mathcal{P}^{q + r - 1}([w_q, w_q']) \in [\mathcal{P}^{q + r}(u_q), g_z\hat{\mathsf{S}}_z]$. Then, $\mathcal{P}^{-(q + r - 1)}(v_q) \in \gamma'\hat{\mathsf{R}}_{z'}$ and $\mathcal{P}^{-(q + r)}(v_q) \in h_q\hat{\mathsf{R}}_y$. The first fact we can conclude is $h_q \in S \subset \Omega$ by \cref{lem:ProducingElementsofH}. The second fact we can conclude is $h_q\pi_{\HS}(u) \in B_{2\hat{\delta}}(\mathcal{P}^{-(q + r)}(v_q))$ and hence we have $h_q\pi_{\HS}(u) \in B_{\overline{\Ret} + 4\hat{\delta}}(B_{\hat{\delta}}(\gamma)) \subset B_{\overline{\Ret} + 4\hat{\delta}}(B_{2\hat{\delta}}(\mathcal{C}_0))$.

Define the cone $\mathcal{C} = B_{\overline{\Ret} + 4\hat{\delta}}(B_{2\hat{\delta}}(\mathcal{C}_0))$. Then, $\{h_q\}_{q \in \mathbb N} \subset \Omega$ is a sequence such that $\{h_q\pi_{\HS}(u)\}_{q \in \mathbb N} \subset \mathcal{C}$. In fact, since $h_q\pi_{\HS}(u) \in B_{\overline{\Ret} + 4\hat{\delta}}(B_{\hat{\delta}}(\gamma))$ for all $q \in \mathbb N$ and $\lim_{q \to \infty} u_q^+ = 0$, we also have $\lim_{q \to \infty} h_q\pi_{\HS}(u) = 0$. Hence, $u^+ = 0 \in \Lambda_{\mathrm{r}}(\Omega)$.
\end{proof}

\begin{figure}
\centering
\includegraphics{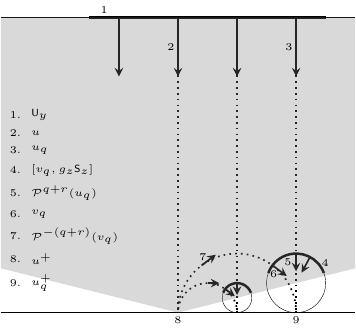}
\caption{This illustrates the idea of the proof of \cref{pro:LimitSetH(yy)RadialEqualsLimitSetH(yy)EqualsLimitSetGamma}. Note that the actual positions of $v_q$ and $\mathcal{P}^{-(q + r)}(v_q)$ are perturbations of what is shown.}
\label{fig:OrbitVectorCloseToOrigin}
\end{figure}

For the proof of \cref{thm:Z-DenseInSimplyConnectedCoverGAndTraceFieldK}, we require some more tools. First we recall a fact due to Susskind--Swarup \cite{SS92}.

\begin{lemma}[{\cite[Corollary 2]{SS92}}]
\label{lem:SS92}
Let $\Gamma' < G$ be a discrete subgroup and $\gamma \in G$ be a hyperbolic element, one of whose fixed points is a radial limit point of $\Gamma'$. If $\langle \Gamma', \gamma \rangle < G$ is a discrete subgroup, then $\gamma^n \in \Gamma'$ for some $n \in \N$.
\end{lemma}

We also require a machinery of Prasad--Rapinchuk \cite{PR09}. However, we are forced to work with Zariski dense subsemigroups of $G$ rather than Zariski dense subgroups of $G$. Fortunately, with appropriate generalizations of the required definition and theorem, the proof goes through nearly verbatim. First we generalize \cite[Definitions 1.1]{PR09} to subsemigroups. We emphasize that in our generalization of the second part of the definition, it is no longer symmetric and has directionality.

\begin{definition}[Weakly commensurable]
Let $\mathbf G_1$ and $\mathbf G_2$ be semisimple linear algebraic groups defined over a field $\mathbb F$.
\begin{enumerate}
\item We say that $g_1 \in \mathbf G_1(\mathbb F)$ and $g_2 \in \mathbf G_2(\mathbb F)$ are \emph{weakly commensurable} if there exist:
\begin{itemize}
\item maximal $\mathbb F$-tori $\mathbf T_1 < \mathbf G_1$ and $\mathbf T_2 < \mathbf G_2$;
\item $\overline{\mathbb F}$-characters $\chi_1: \mathbf T_1 \to \mathbf{G}_{\mathrm{m}}$ and $\chi_2: \mathbf T_2 \to \mathbf{G}_{\mathrm{m}}$;
\end{itemize}
such that
\begin{align*}
g_1 &\in \mathbf T_1(\mathbb F), & g_2 &\in \mathbf T_2(\mathbb F), & \chi_1(g_1) &= \chi_2(g_2) \neq 1.
\end{align*}
\item We say that a Zariski dense subsemigroup $\Gamma_1 < \mathbf G_1(\mathbb F)$ is \emph{weakly commensurable} to another Zariski dense subsemigroup $\Gamma_2 < \mathbf G_2(\mathbb F)$ if for all semisimple elements $\gamma_1 \in \Gamma_1$ of infinite order, there exists a semisimple element $\gamma_2 \in \Gamma_2$ of infinite order such that $\gamma_1$ and $\gamma_2$ are weakly commensurable.
\end{enumerate}
\end{definition}

The following theorem is the generalization of \cite[Theorem 2]{PR09} to subsemigroups with a slight modification for the finite generation hypothesis.

\begin{theorem}
\label{thm:PR}
Let $\mathbf{G}_1$ and $\mathbf{G}_2$ be connected absolutely simple linear algebraic groups defined over a field $\mathbb F$ of characteristic zero. For a finitely generated subring $R \subset \mathbb F$, let $\Gamma_1 < \mathbf{G}_1(R)$ and $\Gamma_2 < \mathbf{G}_2(R)$ be Zariski dense subsemigroups. If $\Gamma_1$ is weakly commensurable to $\Gamma_2$, then we have the containment of their trace fields: $\K_{\Gamma_1} \subset \K_{\Gamma_2}$.
\end{theorem}

Let us now explain the key points due to which the proof of \cite[Theorem 2]{PR09} in \cite[\S 5]{PR09} also works for \cref{thm:PR} with minor changes; we keep the same notation and make references to loc. cit.

\begin{enumerate}
\item \cite[Proposition 5.1]{PR09} is a general proposition about fields and does not concern the subsemigroups $\Gamma_1$ and $\Gamma_2$.
\item \cite[Lemma 5.3]{PR09} and its proof hold for a \emph{semigroup} $\Gamma$ due to the general fact that the Zariski closure of its image under a homomorphism into an algebraic group over $\mathbb F$ must be an algebraic \emph{subgroup}.
\item \cite[Proposition 5.2]{PR09} and its proof hold for a \emph{semigroup} $\Gamma$ due to the general fact that the closure of its image under a homomorphism into the $\mathbb Q_p$-points of an algebraic group over $\mathbb Q_p$ must be a \emph{subgroup}.
\item In the beginning of the proof of \cite[Theorem 2]{PR09}, finite generation of $\Gamma_1$ and $\Gamma_2$ is used only to ensure that they are contained in $\mathbf{G}_1(R)$ and $\mathbf{G}_2(R)$, respectively, for some finitely generated subring $R \subset \mathbb F$. We simply impose the latter as a hypothesis in \cref{thm:PR}. We make this change in the theorem to allow infinitely generated subsemigroups which is required for our purposes.
\item \cite[Lemma 2]{PR03} invoked later in the proof of \cite[Theorem 2]{PR09} is actually stated for \emph{subsemigroups}. Using this lemma and again the fact from (3) above, there exists $\gamma_1$ such that $\rho(\gamma_1) \in \mathrm U \times \mathrm U^{(1)} \times \mathrm U^{(2)}$. The rest of the proof of \cite[Theorem 2]{PR09} goes through verbatim.
\end{enumerate}

We have the following immediate corollary in our setting.

\begin{corollary}
\label{cor:PR09}
Let $\tilde{\Gamma}_1, \tilde{\Gamma}_2 < \tilde{\mathbf G}(\calO_\K)$ be Zariski dense subsemigroups. If $\tilde{\Gamma}_1$ is weakly commensurable to $\tilde{\Gamma}_2$, then $\K_{\Gamma_1} \subset \K_{\Gamma_2}$.
\end{corollary}

Finally, we combine the above tools to prove \cref{thm:Z-DenseInSimplyConnectedCoverGAndTraceFieldK}.

\begin{proof}[Proof of \cref{thm:Z-DenseInSimplyConnectedCoverGAndTraceFieldK}]
By \cref{pro:LimitSetH(yy)RadialEqualsLimitSetH(yy)EqualsLimitSetGamma}, we have $\Lambda_\infty(\langle \gen\rangle_\sg) \subset \Lambda_{\mathrm{r}}(\Omega)$. Recalling that $\Lambda_\infty(\langle \gen\rangle_\sg)$ contains $\Lambda_+$, we know that the former is dense in $\Lambda(\Gamma) \cap \Delta_0$. Therefore, Zariski density of $\tilde{\Omega}$ follows using \cref{lem: Zariski dense condition} by the same argument as in the proof of \cref{thm:Z-DenseInSimplyConnectedCoverGAndTraceFieldK}.

Now we turn to the trace field. Now, any element $\gamma \in \langle \gen\rangle_\sg$ is hyperbolic and writing it as a word in the letters of $\gen$ gives an associated finite sequence $\beta$. We then obtain an infinite sequence $\alpha \in \Sigma^+$ by extending $\beta$ in a periodic fashion, i.e., by appending together repetitions of $\beta$. Then, it is clear that the attracting fixed point $\gamma^+$ of $\gamma$ satisfies $\gamma^+ = \zeta(\alpha) \in \Lambda_\infty(\langle \gen\rangle_\sg)$. In fact, by \cref{pro:LimitSetH(yy)RadialEqualsLimitSetH(yy)EqualsLimitSetGamma}, we have $\gamma^+ \in \Lambda_\infty(\langle \gen\rangle_\sg) \subset \Lambda_{\mathrm{r}}(\Omega)$. Thus, by \cref{lem:SS92} of Susskind--Swarup, we can conclude that $\gamma^n \in \Omega$ for some $n \in \N$. We conclude from definitions that $\langle \gen\rangle_\sg$ is weakly commensurable to $\Omega$. Therefore, $\widetilde{\langle \gen\rangle_\sg}$ is weakly commensurable to $\tilde{\Omega}$. \Cref{cor:PR09} then gives $\K_{\langle \gen\rangle_\sg} \subset \K_{\Omega}$. Recalling our standing \cref{assumption}, which says $\K_{\langle \gen\rangle_\sg} = \K$, the proof is complete.
\end{proof}

\section{Proof of \cref{thm:UniformSpectralBoundIReduced} via expansion machinery}
\label{sec:ProofOfUniformSpectralBoundIReducedViaExpansionMachinery}
We follow the rest of the arguments in \cite{Sar22a} (see also \cite{Sar22b}). First, the congruence transfer operators will be approximated by convolutions with measures to mimic a random walk. Then, we derive an $L^2$-flattening lemma using the expansion machinery. These two ingredients are required for the proof of \cref{thm:UniformSpectralBoundIReduced}. Since much of the proofs are verbatim repetitions of those in \cite{Sar22a}, we mainly focus on the differences that arise due to the countably infinite coding $\calA$.

Let $\mathfrak{q} \subset \mathcal{O}_{\mathbb K}$ be an ideal coprime to $\mathfrak{q}_0$. We define $\Sigma^+$-analogues of $T$, $\Ret$, $\FRet^{(a)}$, and $\mathcal{M}_{\xi, \mathfrak{q}}$ as follows. We simply take the $\Sigma^+$-analogue of $T$ to be the shift operator $\sigma: \Sigma^+ \to \Sigma^+$. For any function $f$ defined on $\Lambda_+$, in particular $\Ret$ and $\FRet^{(a)}$, we take its $\Sigma^+$-analogue to be $f \circ \zeta$. We now define the $\Sigma^+$-analogue
\begin{align*}
\mathcal{M}_{\xi, \mathfrak{q}}: C_{\mathrm{b}}\bigl(\Sigma^+, L^2(\tilde{G}_\mathfrak{q})\bigr) \to C_{\mathrm{b}}\bigl(\Sigma^+, L^2(\tilde{G}_\mathfrak{q})\bigr)
\end{align*}
as in \cref{eqn:k^thIterationOfCongruenceTransferOperatorOfType_rho} by replacing all objects with their $\Sigma^+$-analogues, in particular replacing the $T$-branches with $\sigma$-branches, and then taking $k = 1$ and $\rho = 1$. For the rest of the section, we will only work with $\Sigma^+$-analogous objects without further comments.

By a simple inductive proof (see \cite[\S 10]{Sar22a}), \cref{thm:UniformSpectralBoundIReduced} follows from the following proposition.

\begin{proposition}
\label{prop:ReducedTheoremEstimate}
There exist $\kappa \in (0, 1)$, $C > 0$, $a_0 > 0$, $b_0 > 0$, and $q_1 \in \mathbb N$ such that for all (square-free if $n = 3$) ideals $\mathfrak{q} \subset \mathcal{O}_{\mathbb K}$ coprime to $\mathfrak{q}_0$ with $N_{\mathbb K}(\mathfrak{q}) > q_1$, there exists an integer $s \in (0, C\log(N_{\mathbb K}(\mathfrak{q})))$ such that for all $H \in \Lip\bigl(\Sigma^+, E_\mathfrak{q}^\mathfrak{q}\bigr)$, we have
\begin{align*}
\big\|\mathcal{M}_{\xi, \mathfrak{q}}^s(H)\big\|_\infty \leq \frac{1}{2}N_{\mathbb K}(\mathfrak{q})^{-\kappa}\|H\|_{\Lip(d_\theta)} \qquad \text{for all $\xi \in \mathbb C$ with $|a| < a_0$},
\end{align*}
and
\begin{multline*}
\Lip_{d_\theta}\bigl(\mathcal{M}_{\xi, \mathfrak{q}}^s(H)\bigr) \leq \frac{1}{2}N_{\mathbb K}(\mathfrak{q})^{-\kappa}\|H\|_{\Lip(d_\theta)}\\
\text{for all $\xi \in \mathbb C$ with $|a| < a_0$ and $|b| \leq b_0$}.
\end{multline*}
\end{proposition}

\subsection{Approximating the congruence transfer operators}
We introduce some useful notations. Let $j \in \mathbb N$ and $(\alpha_j, \alpha_{j - 1}, \dotsc, \alpha_1)$ be a sequence. We denote $\alpha^j = (\alpha_j, \alpha_{j - 1}, \dotsc, \alpha_1)$. Also, when sequences are themselves written in a sequence, they are to be concatenated. For all $y \in \mathcal A$, denote $\omega(y) \in \Sigma^+$ to be any sequence such that $(y, \omega(y))$ is admissible. We naturally extend the notation for admissible sequences as well so that $\omega(\alpha^j) = \omega(\alpha_1)$.

Let $\mathfrak{q} \subset \mathcal{O}_{\mathbb K}$ be an ideal coprime to $\mathfrak{q}_0$. For any complex measure $\mu$ on $\tilde{G}_\mathfrak{q}$ and $\phi \in L^2(\tilde{G}_\mathfrak{q})$ the convolution $\mu * \phi \in L^2(\tilde{G}_\mathfrak{q})$ is defined by
\begin{align*}
	(\mu * \phi)(g) = \sum_{h \in \tilde{G}_\mathfrak{q}} \mu(h) \phi(gh^{-1}) \qquad \text{for all $g \in \tilde{G}_\mathfrak{q}$}.
\end{align*}

For all $\xi \in \mathbb C$ with $|a| < a_0'$, for all ideals $\mathfrak{q} \subset \mathcal{O}_{\mathbb K}$ coprime to $\mathfrak{q}_0$, $x \in \Sigma^+$, integers $0 < r < s$, and sequences $(\alpha_s, \alpha_{s - 1}, \dotsc, \alpha_{r + 1})$, we define the complex measures
\begin{align*}
	\mu_{(\alpha_s, \alpha_{s - 1}, \dotsc, \alpha_{r + 1})}^{\xi, \mathfrak{q}, x} &= \sum_{\alpha^r} e^{(\FRet_s^{(a)} + ib\Ret_s)(\alpha^s, x)} \delta_{\pi_\mathfrak{q}(\gamma(\alpha^r))}, \\
	\nu_0^{a, \mathfrak{q}, x, r} &= \sum_{\alpha^r} e^{\FRet_r^{(a)}(\alpha^r, x)} \delta_{\pi_\mathfrak{q}(\gamma(\alpha^r))}, \\
	\hat{\mu}_{(\alpha_s, \alpha_{s - 1}, \dotsc, \alpha_{r + 1})}^{a, \mathfrak{q}, x} &= \sum_{\alpha^r} e^{\FRet_s^{(a)}(\alpha^s, x)} \delta_{\pi_\mathfrak{q}(\gamma(\alpha^r))} = e^{\FRet_{s - r}^{(a)}(\alpha^s, x)} \nu_0^{a, \mathfrak{q}, x, r}, \\
	\nu_{(\alpha_s, \alpha_{s - 1}, \dotsc, \alpha_{r + 1})}^{a, \mathfrak{q}, x} &= e^{\FRet_{s - r}^{(a)}(\alpha_s, \alpha_{s - 1}, \dotsc, \alpha_{r + 1}, \omega(\alpha_{r + 1}))} \nu_0^{a, \mathfrak{q}, x, r},
\end{align*}
on $\tilde{G}_\mathfrak{q}$, and also for all $H \in C_{\mathrm{b}}\bigl(\Sigma^+, L^2(\tilde{G}_\mathfrak{q})\bigr)$, define the function
\begin{align*}
	\phi_{(\alpha_s, \alpha_{s - 1}, \dotsc, \alpha_{r + 1})}^{\mathfrak{q}, H} = \delta_{\pi_\mathfrak{q}(\gamma(\alpha_s, \alpha_{s - 1}, \dotsc, \alpha_{r + 1}))} * H(\alpha_s, \alpha_{s - 1}, \dotsc, \alpha_{r + 1}, \omega(\alpha_{r + 1}))
\end{align*}
in $L^2(\tilde{G}_\mathfrak{q})$. Note that $\big\|\phi_{(\alpha_s, \alpha_{s - 1}, \dotsc, \alpha_{r + 1})}^{\mathfrak{q}, H}\big\|_2 \leq \|H\|_\infty$.

In the course of the proof of \cref{prop:ReducedTheoremEstimate}, the basic \cite[Lemma 7.1]{Sar22a} is also required. In our setting with countably infinite coding $\calA$, the lemma is also required to ensure that the above measures, and all sums which appear in the rest of this section, are well-defined---we justify it below.

\begin{lemma}
There exists $C > 0$ such that for all $|a| < a_0'$, $x \in \Sigma^+$, and $k \in \N$, we have $\sum_{\alpha^k} e^{\FRet^{(a)}_k(\alpha^k, x)} \leq C$.
\end{lemma}

\begin{proof}
Let $|a| < a_0'$, $x \in \Sigma^+$, and $k \in \mathbb N$. We use the Cauchy--Schwarz inequality to calculate that
\begin{align*}
\sum_{\alpha^k} e^{\FRet_k^{(a)}(\alpha^k, x)} &= \sum_{\alpha^k} e^{\FRet_k^{(a)}(\alpha^k, x) - (1/2)\FRet_k^{(0)}(\alpha^k, x)} \cdot e^{(1/2)\FRet_k^{(0)}(\alpha^k, x)} \\
&\leq \left(\sum_{\alpha^k} e^{2\FRet_k^{(a)}(\alpha^k, x) - \FRet_k^{(0)}(\alpha^k, x)}\right)^{1/2} \left(\sum_{\alpha^k} e^{\FRet_k^{(0)}(\alpha^k, x)}\right)^{1/2} \\
&\leq C\mathcal{L}_0^k(\chi_{\Sigma^+})(x)^{1/2} \leq C
\end{align*}
for some absolute constant $C > 0$. To see this, we may expand $2\FRet_k^{(a)}(\alpha^k, x) - \FRet_k^{(0)}(\alpha^k, x)$ and observe that the first term is $-(\delta_\Gamma + 2a)\Ret_k(\alpha^k, x)$ and the rest of the terms are bounded uniformly in $|a| < a_0'$. We then recall $|2a| < 2a_0' < \epsilon_0$ and apply the exponential tail property from \cref{prop:Coding}. We have also used the fact that $\mathcal{L}_0(\chi_{\Sigma^+}) = \chi_{\Sigma^+}$.
\end{proof}

The following lemma (cf. \cite[Lemma 7.3]{Sar22a}) shows that congruence transfer operators can be approximated by certain convolutions; it is proved as in \cite[Lemma 4.21]{OW16}.

\begin{lemma}
	\label{lem:TransferOperatorConvolutionApproximation}
	For all $\xi \in \mathbb C$ with $|a| < a_0'$, for all ideals $\mathfrak{q} \subset \mathcal{O}_{\mathbb K}$ coprime to $\mathfrak{q}_0$, $x \in \Sigma^+$, integers $0 < r < s$, and $H \in \Lip\bigl(\Sigma^+, L^2(\tilde{G}_\mathfrak{q})\bigr)$, we have
	\begin{multline*}
		\left\|\mathcal{M}_{\xi, \mathfrak{q}}^s(H)(x) - \sum_{\alpha_{r + 1}, \alpha_{r + 2}, \dotsc, \alpha_s} \mu_{(\alpha_s, \alpha_{s - 1}, \dotsc, \alpha_{r + 1})}^{\xi, \mathfrak{q}, x} * \phi_{(\alpha_s, \alpha_{s - 1}, \dotsc, \alpha_{r + 1})}^{\mathfrak{q}, H}\right\|_2 \\
		\leq C_f \Lip_{d_\theta}(H)\theta^{s - r}.
	\end{multline*}
\end{lemma}

\subsection{$L^2$-flattening lemma}
\Cref{prop:ReducedTheoremEstimate} is proven using the following lemma exactly as in \cite[\S 10]{Sar22a} and so we omit the proof.

\begin{lemma}
\label{lem:L2FlatteningLemma}
There exist $C > 0$, $C_0 > 0$, and $l \in \mathbb N$ such that for all $\xi \in \mathbb C$ with $|a| < a_0'$, (square-free if $n = 3$) ideals $\mathfrak{q} \subset \mathcal{O}_{\mathbb K}$ coprime to $\mathfrak{q}_0$, $x \in \Sigma^+$, integers $C_0 \log(N_{\mathbb K}(\mathfrak{q})) \leq r < s$ with $r \in l\mathbb Z$, sequences $(\alpha_s, \alpha_{s - 1}, \dotsc, \alpha_{r + 1})$, and $\phi \in E_\mathfrak{q}^\mathfrak{q}$ with $\|\phi\|_2 = 1$, we have
\begin{align*}
\left\|\mu_{(\alpha_s, \alpha_{s - 1}, \dotsc, \alpha_{r + 1})}^{\xi, \mathfrak{q}, x} * \phi\right\|_2 &\leq C N_{\mathbb K}(\mathfrak{q})^{-\frac{1}{3}} \left\|\nu_{(\alpha_s, \alpha_{s - 1}, \dotsc, \alpha_{r + 1})}^{a, \mathfrak{q}, x}\right\|_1, \\
\left\|\hat{\mu}_{(\alpha_s, \alpha_{s - 1}, \dotsc, \alpha_{r + 1})}^{a, \mathfrak{q}, x} * \phi\right\|_2 &\leq C N_{\mathbb K}(\mathfrak{q})^{-\frac{1}{3}} \left\|\nu_{(\alpha_s, \alpha_{s - 1}, \dotsc, \alpha_{r + 1})}^{a, \mathfrak{q}, x}\right\|_1.
\end{align*}
\end{lemma}

In the rest of this section, we focus on the proof of \cref{lem:L2FlatteningLemma}. The arguments are based on \cite[\S 9]{Sar22a} which was largely inspired by and generalized from \cite[Appendix]{MOW19} by Bourgain--Kontorovich--Magee. We omit most derivation of estimates as they are very similar to those in \cite{Sar22a}. However, we focus on the key difference which is in the proof of \cref{lem:GV_Expander}, using the results of the previous section.

In the rest of this section, we consider $\tilde{S}_{\mathrm{fin}}$ from \cref{cor:Z-DenseInSimplyConnectedCoverGAndTraceFieldK} and fix $p \in \N$ to be half the largest word length of its elements viewed as words generated by $\gen$. Then, the corollary applies when it is needed in \cref{lem:GV_Expander}.

For the purposes of proving \cref{lem:L2FlatteningLemma}, we will fix $\xi \in \mathbb C$ with $|a| < a_0'$, $x \in \Sigma^+$, $r \in \mathbb N$ with factorization $r = r'l$ for some fixed $r' \in \mathbb N$ and some fixed integer $l > p$ henceforth. For all ideals $\mathfrak{q} \subset \mathcal{O}_{\mathbb K}$ coprime to $\mathfrak{q}_0$, for all sequences $(\alpha_s, \alpha_{s - 1}, \dotsc, \alpha_{r + 1})$, denote $\mu_{(\alpha_s, \alpha_{s - 1}, \dotsc, \alpha_{r + 1})}^{\xi, \mathfrak{q}, x}$, $\nu_0^{a, \mathfrak{q}, x, r}$, $\hat{\mu}_{(\alpha_s, \alpha_{s - 1}, \dotsc, \alpha_{r + 1})}^{a, \mathfrak{q}, x}$, and $\nu_{(\alpha_s, \alpha_{s - 1}, \dotsc, \alpha_{r + 1})}^{a, \mathfrak{q}, x}$ by $\mu_{(\alpha_s, \alpha_{s - 1}, \dotsc, \alpha_{r + 1})}^\mathfrak{q}$, $\nu_0^\mathfrak{q}$, $\hat{\mu}_{(\alpha_s, \alpha_{s - 1}, \dotsc, \alpha_{r + 1})}^\mathfrak{q}$, and $\nu_{(\alpha_s, \alpha_{s - 1}, \dotsc, \alpha_{r + 1})}^\mathfrak{q}$, respectively.

Let $\alpha^r$ be a sequence. The following additional notations will be useful. Define
\begin{align*}
	\alpha_j^l &= (\alpha_{jl}, \alpha_{jl - 1}, \dotsc, \alpha_{(j - 1)l + 1}), \\
	\alpha_j^{(p)_1} &= (\alpha_{jl}, \alpha_{jl - 1}, \dotsc, \alpha_{jl - p + 1}), \\
	\alpha_j^{(l - p)_2} &= (\alpha_{jl - p}, \alpha_{jl - p - 1}, \dotsc, \alpha_{(j - 1)l + 1}),
\end{align*}
for all $1 \leq j \leq r'$. For example, with these notations and conventions we have $\alpha^r = \big(\alpha_{r'}^l, \alpha_{r' - 1}^l, \dotsc,  \alpha_1^l\big) = (\alpha_r, \alpha_{r - 1}, \dotsc, \alpha_1)$ and $\alpha_j^l = \big(\alpha_j^{(p)_1}, \alpha_j^{(l - p)_2}\big)$ for all $1 \leq j \leq r'$. We also have $\sigma^k(\alpha^j) = \alpha^{j - k}$ for all $1 \leq j \leq r$ and $0 \leq k \leq j - 1$. We also add the convention that $\alpha_j^{(l - p)_2}$ is the empty sequence for $j \in \{0, r' + 1\}$.

For all ideals $\mathfrak{q} \subset \mathcal{O}_{\mathbb K}$ coprime to $\mathfrak{q}_0$ and sequences $\alpha^r$, define the coefficients and measures
\begin{align*}
E\big(\alpha_{j + 1}^{(l - p)_2}, \alpha_j^l\big) &=
\begin{cases}
e^{\FRet_{2l - p}^{(a)}(\alpha^{2l - p}, x)}, & j = 1 \\
e^{\FRet_l^{(a)}(\alpha_{j + 1}^{(l - p)_2}, \alpha_j^l, \omega(\alpha_j^l))}, & 2 \leq j \leq r' - 1 \\
e^{\FRet_p^{(a)}(\alpha_{r'}^l, \omega(\alpha_{r'}^l))}, & j = r'
\end{cases}
\\
\eta^\mathfrak{q}\big(\alpha_{j + 1}^{(l - p)_2}, \alpha_j^{(l - p)_2}\big) &= \sum_{\alpha_j^{(p)_1}} E\big(\alpha_{j + 1}^{(l - p)_2}, \alpha_j^l\big) \delta_{\pi_\mathfrak{q}(\gamma(\alpha_j^l))} \qquad \text{for all $1 \leq j \leq r'$}
\end{align*}
which depend on $\alpha_j^{(l - p)_2}$ only for $j = 1$. Due to the following lemma, which is proved as in \cite[Lemma 9.3]{Sar22a}, these measures are said to be \emph{nearly flat}.

\begin{lemma}
\label{lem:NearlyFlat}
There exists $C > 1$ such that for all $1 \leq j \leq r'$, and for all pairs of sequences $\big(\alpha_{j + 1}^{(l - p)_2}, \alpha_j^l\big)$ and $\big(\tilde{\alpha}_{j + 1}^{(l - p)_2}, \tilde{\alpha}_j^l\big)$ with $\big(\alpha_{j + 1}^{(l - p)_2}, \alpha_j^{(l - p)_2}\big) = \big(\tilde{\alpha}_{j + 1}^{(l - p)_2}, \tilde{\alpha}_j^{(l - p)_2}\big)$, we have
\begin{align*}
\frac{E\big(\tilde{\alpha}_{j + 1}^{(l - p)_2}, \tilde{\alpha}_j^l\big)}{E\big(\alpha_{j + 1}^{(l - p)_2}, \alpha_j^l\big)} \leq C
\end{align*}
where $C$ is independent of $|a| < a_0'$, $x \in \Sigma^+$, $r \in \mathbb N$ and its factorization $r = r'l$ with $l > p$.
\end{lemma}

For all ideals $\mathfrak{q} \subset \mathcal{O}_{\mathbb K}$ coprime to $\mathfrak{q}_0$, we also define the measure
\begin{align*}
	\nu_1^\mathfrak{q} = \sum_{\alpha_1^{(l - p)_2}, \alpha_2^{(l - p)_2}, \dotsc, \alpha_{r'}^{(l - p)_2}} \mathop{\bigast}\limits_{j = 1}^{r'} \eta^\mathfrak{q}\big(\alpha_{j + 1}^{(l - p)_2}, \alpha_j^{(l - p)_2}\big).
\end{align*}
The following lemma, which is proved as in \cite[Lemma 9.4]{Sar22a}, shows that $\nu_0^\mathfrak{q}$ and $\nu_1^\mathfrak{q}$ are equivalent up to a certain multiplicative constant.

\begin{lemma}
	\label{lem:EstimateNu}
	There exists $C > 0$ such that for all ideals $\mathfrak{q} \subset \mathcal{O}_{\mathbb K}$ coprime to $\mathfrak{q}_0$, we have $\nu_0^\mathfrak{q} \leq e^{r' C\theta^l}\nu_1^\mathfrak{q}$ and $\nu_1^\mathfrak{q} \leq e^{r' C\theta^l} \nu_0^\mathfrak{q}$ where $C$ is independent of $|a| < a_0'$, $x \in \Sigma^+$, $r \in \mathbb N$ and its factorization $r = r'l$ with $l > p$.
\end{lemma}

The following is the key lemma which uses the expansion machinery of Golsefidy--Varj\'{u} \cite{GV12} and He--de Saxc\'{e} \cite{HdS22}, for which \cref{cor:Z-DenseInSimplyConnectedCoverGAndTraceFieldK} is required.

\begin{lemma}
\label{lem:GV_Expander}
There exists $\epsilon \in (0, 1)$ such that for all integers $0 \leq j \leq r'$, (square-free if $n = 3$) ideals $\mathfrak{q} \subset \calO_\K$ coprime to $q_0$, and $\phi \in L_0^2(\tilde{\mathbf{G}}_q)$ with $\|\phi\|_2 = 1$, there exist sequences $\beta_j^{(p)_1}$ and $\tilde{\beta}_j^{(p)_1}$ such that for all sequences $\beta_j^{(l - p)_2} = \tilde{\beta}_j^{(l - p)_2}$, we have
\begin{align*}
\|\delta_\gamma * \phi - \phi\|_2 \geq \epsilon
\end{align*}
where
\begin{align*}
\gamma = \pi_\mathfrak{q}\bigl(\gamma\bigl(\beta_j^l\bigr)\gamma\bigl(\tilde{\beta}_j^l\bigr)^{-1}\bigr) = \pi_\mathfrak{q}\left(\prod_{k = 1}^p \gamma_{\beta_{jl + 1 - k}} \prod_{k = 1}^p \gamma_{\tilde{\beta}_{jl - p + k}}^{-1}\right)
\end{align*}
and $\epsilon$ is independent of $r \in \mathbb N$ and its factorization $r = r'l$ with $l > p$.
\end{lemma}

\begin{proof}
Clearly, the lemma is independent of $0 \leq j \leq r'$. Let $\tilde{\Omega}_{\mathrm{fin}} \coloneq \langle \tilde{S}_{\mathrm{fin}}\rangle < \tilde{\Omega} < \tilde{\Gamma}$ be the subgroup generated by the \emph{finite} symmetric subset $\tilde{S}_{\mathrm{fin}} \subset \tilde{S} \subset \tilde{\Gamma}$ provided by \cref{cor:Z-DenseInSimplyConnectedCoverGAndTraceFieldK}. Then, \cref{cor:Z-DenseInSimplyConnectedCoverGAndTraceFieldK} says that $\tilde{\Omega}_{\mathrm{fin}}$ is Zariski dense in $\tilde{\mathbf{G}}$ and $\K_{\tilde{\Omega}_{\mathrm{fin}}} = \K$. Denote
\begin{align*}
\tilde{S}_{\mathrm{fin}, \mathfrak{q}} &\coloneq \pi_\mathfrak{q}(\tilde{S}_{\mathrm{fin}}), & \tilde{\Omega}_{\mathrm{fin}, \mathfrak{q}} \coloneq \pi_\mathfrak{q}(\tilde{\Omega}_{\mathrm{fin}}).
\end{align*}
By the strong approximation theorem of Weisfeiler \cite[Theorem 10.1]{Wei84}, we can then conclude $\tilde{\Omega}_{\mathrm{fin}, \mathfrak{q}} = \tilde{G}_\mathfrak{q}$ for all ideals $\mathfrak{q} \subset \calO_\K$ coprime to $\mathfrak{q}_0$. We can further use \cite[Corollary 6]{GV12} and the analogous corollary of \cite[Theorem 6.1]{HdS22} for $n \neq 3$ to conclude that the Cayley graphs $\Cay(\tilde{G}_\mathfrak{q}, \tilde{S}_{\mathrm{fin}, \mathfrak{q}})$ form expanders with respect to (square-free if $n = 3$) ideals $\mathfrak{q} \subset \calO_\K$ coprime to $\mathfrak{q}_0$. This means that there exists $\epsilon \in (0, 1)$ such that the smallest eigenvalue of $\Delta_\mathfrak{q}$ is $\lambda_1(\Delta_\mathfrak{q}) = 0$ and the next smallest eigenvalue satisfies $\lambda_2(\Delta_\mathfrak{q}) \geq \epsilon$ for all (square-free if $n = 3$) ideals $\mathfrak{q} \subset \calO_\K$ coprime to $\mathfrak{q}_0$. Here, $\Delta_\mathfrak{q}: L^2(\tilde{G}_\mathfrak{q}) \to L^2(\tilde{G}_\mathfrak{q})$ is a self-adjoint operator called the graph Laplacian and defined by
\begin{align*}
\Delta_\mathfrak{q}(\phi)  = \phi - \frac{1}{\#\tilde{S}_{\mathrm{fin}, \mathfrak{q}}}\sum_{\gamma \in \tilde{S}_{\mathrm{fin}, \mathfrak{q}}} \delta_\gamma * \phi \qquad \text{for all $\phi \in L^2(\tilde{G}_\mathfrak{q})$}.
\end{align*}
This is well-defined thanks to the finiteness of $\tilde{S}_{\mathrm{fin}}$. Note that the eigenspace corresponding to $\lambda_1(\Delta_\mathfrak{q}) = 0$ consists of constant functions. So, for all $\phi \in L_0^2(\tilde{G}_\mathfrak{q})$ with $\|\phi\|_2 = 1$, we have $\|\Delta_\mathfrak{q}(\phi)\|_2 \geq \epsilon$. Therefore,
\begin{align*}
\sum_{\gamma \in \tilde{S}_{\mathrm{fin}, \mathfrak{q}}} \|\delta_\gamma * \phi - \phi\|_2 \geq \epsilon \cdot \#\tilde{S}_{\mathrm{fin}, \mathfrak{q}}.
\end{align*}
We deduce that there exists $\gamma \in \tilde{S}_{\mathrm{fin}, \mathfrak{q}}$ such that $\|\delta_\gamma * \phi - \phi\|_2 \geq \epsilon$, for all (square-free if $n = 3$) ideals $\mathfrak{q} \subset \calO_\K$ coprime to $\mathfrak{q}_0$. But $\tilde{S}_{\mathrm{fin}} \subset \tilde{\Gamma}$ and we have the induced isomorphisms $\overline{\pi_\mathfrak{q}|_{\tilde{\Gamma}}}: \tilde{\Gamma}_\mathfrak{q} \backslash \tilde{\Gamma} \to \tilde{G}_\mathfrak{q}$ and $\overline{\tilde{\pi}}: \tilde{\Gamma}_\mathfrak{q} \backslash \tilde{\Gamma} \to \Gamma_\mathfrak{q} \backslash \Gamma$. Following these isomorphisms shows that $\gamma$ has the form as in the lemma, completing the proof.
\end{proof}

The following lemma is proved as in \cite[Lemma 9.6]{Sar22a} using \cref{lem:NearlyFlat,lem:GV_Expander}.

\begin{lemma}
	\label{lem:EtaOperatorBound}
	There exists $C \in (0, 1)$ such that for all (square-free if $n = 3$) ideals $\mathfrak{q} \subset \mathcal{O}_{\mathbb K}$ coprime to $\mathfrak{q}_0$, integers $1 \leq j \leq r'$, pairs of sequences $\big(\alpha_{j + 1}^{(l - p)_2}, \alpha_j^{(l - p)_2}\big)$, and $\phi \in L_0^2(\tilde{G}_\mathfrak{q})$ with $\|\phi\|_2 = 1$, we have
	\begin{align*}
		\left\|\eta^\mathfrak{q}\big(\alpha_{j + 1}^{(l - p)_2}, \alpha_j^{(l - p)_2}\big) * \phi\right\|_2 \leq C \left\|\eta^\mathfrak{q}\big(\alpha_{j + 1}^{(l - p)_2}, \alpha_j^{(l - p)_2}\big)\right\|_1
	\end{align*}
	where $C$ is independent of $|a| < a_0'$, $x \in \Sigma^+$, $r \in \mathbb N$ and its factorization $r = r'l$ with $l > p$.
\end{lemma}

Then, the following lemma is proved as in \cite[Lemma 9.7]{Sar22a} using \cref{lem:EstimateNu,lem:EtaOperatorBound}.

\begin{lemma}
	\label{lem:ExpanderMachineryBound}
	There exists $l_0 \in \mathbb N$ such that if $l > l_0$, then there exists $C \in (0, 1)$ such that for all (square-free if $n = 3$) ideals $\mathfrak{q} \subset \mathcal{O}_{\mathbb K}$ coprime to $\mathfrak{q}_0$, integers $s > r$, sequences $(\alpha_s, \alpha_{s - 1}, \dotsc, \alpha_{r + 1})$, and $\phi \in L_0^2(\tilde{G}_\mathfrak{q})$ with $\|\phi\|_2 = 1$, we have
	\begin{align*}
		\left\|\nu_{(\alpha_s, \alpha_{s - 1}, \dotsc, \alpha_{r + 1})}^\mathfrak{q} * \phi\right\|_2 \leq C^r \left\|\nu_{(\alpha_s, \alpha_{s - 1}, \dotsc, \alpha_{r + 1})}^\mathfrak{q}\right\|_1
	\end{align*}
	where $C$ is independent of $|a| < a_0'$, $x \in \Sigma^+$, and $r \in \mathbb N$, but dependent on the factorization $r = r'l$.
\end{lemma}

In the following lemma, $\tilde{\mu}: L^2(\tilde{G}_\mathfrak{q}) \to L^2(\tilde{G}_\mathfrak{q})$ associated to a complex measure $\mu$ on $\tilde{G}_\mathfrak{q}$ denotes an operator acting by convolution: $\tilde{\mu}(\phi) = \mu * \phi$ for all $\phi \in L^2(\tilde{G}_\mathfrak{q})$. The lemma is proved as in \cite[Lemma 9.8]{Sar22a} and the subsequent remark in loc. cit. says that the square-free hypothesis is not required.

\begin{lemma}
	\label{lem:ConvolutionBoundOnE_q^q}
	There exists $C > 0$ such that for all ideals $\mathfrak{q} \subset \mathcal{O}_{\mathbb K}$ coprime to $\mathfrak{q}_0$ and complex measures $\mu$ on $\tilde{G}_\mathfrak{q}$, we have
	\begin{align*}
		\|\tilde{\mu}|_{E_\mathfrak{q}^\mathfrak{q}}\|_{\mathrm{op}} \leq C N_{\mathbb K}(\mathfrak{q})^{-\frac{1}{4}} (\#\tilde{G}_\mathfrak{q})^{\frac{1}{2}} \|\mu\|_2.
	\end{align*}
\end{lemma}

Finally, \cref{lem:L2FlatteningLemma} is proved as in \cite[Lemma 9.1]{Sar22a} using \cref{lem:ExpanderMachineryBound,lem:ConvolutionBoundOnE_q^q}.

\section{Local non-integrability condition, non-concentration property, and large deviation property}
\label{sec:LNIC_NCP_LDP}
In this section, we collect three essential properties for Dolgopyat's method from \cite{LPS25}. The first two have already appeared previously in \cite{SW21}. The third property is required due to the presence of cusps; it was introduced and proved in \cite{LPS25}.

\subsection{LNIC}
\label{subsec:LNIC}
Let us recall the first key property called the \emph{local non-integrability condition (LNIC)} which we state in our setting. Given any pair of inverse branches $(\alpha, \beta) \in \gen^{(m)} \times \gen^{(m)}$ for some $m\in \mathbb{N}$, define the map $\BP_{(\alpha, \beta)}:\Delta_0\times \Delta_0\to AM$ by
\begin{align*}
\BP_{(\alpha, \beta)}(x,y) \coloneq \GHol^m(\alpha x, \infty)^{-1}\GHol^m(\alpha y, \infty)\GHol^m(\beta y, \infty)^{-1}\GHol^m(\beta x, \infty)
\end{align*}
for all $x, y \in \Delta_0$. This map is related to Brin--Pesin moves \cite{BP74,Bri82} (cf. \cite{SW21}).

\begin{proposition}[LNIC; {\cite[Proposition 5.8]{LPS25}}]
\label{prop:LNIC}
There exist $\epsilon\in (0,1)$, $m_0\in \N$, $\jj\in \N$, such that for all $m\geq m_0$, there exist $\{\alpha_j\}_{j = 0}^{\jj} \subset \gen^{(m)}$ such that for all $x \in \Delta_0$ and $\omega \in \LieA \oplus \LieM$ with $\|\omega\| = 1$, there exist $1\leq j\leq \jj$ and $Z\in \operatorname{T}_x(\Delta_0)$ with $\|Z\| = 1$ such that
\begin{equation*}
|\langle  (d\BP_{j, x})_x(Z),\omega\rangle|\geq \epsilon
\end{equation*}
where we denote $\BP_{j}\coloneq\BP_{(\alpha_0,\alpha_j)}$ and $\BP_{j, x} \coloneq \BP_j(x, \bigcdot)$ for all $x \in \Delta_0$ and $1\leq j\leq \jj$. 
\end{proposition}

We also recall an approximation lemma from \cite{LPS25} which will be used in \cref{sec:Dolgopyat'sMethod} (see \cite[Lemma 7.1]{SW21} for a proof). Fix $\delta_{AM} > 0$ such that any pair of points in $B_{AM}(e,\delta_{AM}) \subset AM$ has a unique geodesic through them. We can also fix a constant $C_{\BP} > 0$ such that
\begin{equation*}
C_{\BP} \geq \sup\{\|(d\BP_{j, x})_y\|_{\mathrm{op}}:x, y \in \Delta_0, j \in \{1, 2, \dotsc, \jj\}\}
\end{equation*}
for any $m\geq m_0$ and corresponding maps $\{\BP_j\}_{j=1}^{\jj}$ provided by \cref{prop:LNIC}.

\begin{lemma}[{\cite[Lemma 5.11]{LPS25}}]
\label{lem:ComparingExpWithBP}
Let $m\geq m_0$ and $\{\BP_j\}_{j=1}^{\jj}$ be the corresponding maps provided by \cref{prop:LNIC}. Then there exists $C_{\exp, \BP} > 0$ such that for all $1 \leq j \leq \jj$ and $x, y \in \Delta_0$ with $\|x - y\| < \frac{\delta_{AM}}{C_{\BP}}$, we have
\begin{align*}
d_{AM}(\exp(Z), \BP_j(x, y)) \leq C_{\exp, \BP}\|x - y\|^2
\end{align*}
where $Z = (d\BP_{j, x})_x(y-x)$.
\end{lemma}

\subsection{NCP}
\label{subsec:NCP}
Let us recall the second key property called the \emph{non-concentration property (NCP)} which we state in our setting. It is the appropriate generalization of NCP in \cite[Proposition 6.6]{SW21} for geometrically finite hyperbolic manifolds with cusps. Define the map $u: \Delta_0 \to \T^1(X)$ by
\begin{align*}
u(x) = \Phi(x, \infty, 0) \qquad \text{for all $x \in \Delta_0$};
\end{align*}
recall the map $\Phi$ from \cref{eqn:embedding}. Note that $u({\Delta_0})$ is then the immersion of a subset of the horosphere corresponding to $\Delta_0 \times \{\infty\} \times \{0\}$. We define the compact region
\begin{align}
\label{compact set}
\Omega_R \subset \T^1(X)
\end{align}
to be the closed $(R + \log(C_{\Delta_0}))$-neighborhood of $u({\Delta_0})$ for some suitable constant $C_{\Delta_0} > 1$; we may take $C_{\Delta_0} = C_{\mathrm{cyl}} \cdot \max\{1, \diam(\Delta_0), \mu(\Delta_0)\}$ where $C_{\mathrm{cyl}} > 1$ is from \cref{lem:CylinderEstimate} (see \cite{LPS25}).

\begin{proposition}[NCP; {\cite[Proposition 6.1]{LPS25}}]
\label{prop:NCP}
Let $R > 0$. There exists $\eta \in (0, 1)$ such that for all $\epsilon \in (0, 1)$, $x \in \Lambda_\Gamma \cap \Delta_0 - \overline{B_\epsilon^{\mathrm{E}}(\partial \Delta_0)}$ with $u(x) a_{-\log(\epsilon)} \in \Omega_R$, and $w \in \R^{n - 1}$ with $\|w\| = 1$, there exists $y \in \Lambda_\Gamma \cap B(x,\epsilon) \subset \Delta_0$ such that $|\langle y - x, w \rangle| \geq \epsilon \eta$.
\end{proposition}

\subsection{LDP}
\label{subsec:LDP}
Let us recall the third key property called the \emph{large deviation property (LDP)}.
Recall the measures $\mu$ and $\nu$ from \cref{subsec:Patterson--SullivanDensity,subsec:expanding map}, respectively, and also \cref{eqn:nu_and_mu}.

For all $t > 0$ and $R > 0$, we define $\Omega^\dagger(t, R)$ to be the set consisting of maximal cylinders $\mathtt{C} \subset \Delta_0$ satisfying
\begin{align*}
e^{-R - t} \leq \diam(\mathtt{C}) \leq e^{R - t}.
\end{align*}
Define 
\begin{align}
\label{eqn: union of sandwiched diameter cylinders}
\Omega(t, R) := \bigcup_{\mathtt{C}\in \Omega^\dagger(t, R)}\mathtt{C} \subset \Delta_0.
\end{align}
It is related to the compact subset $\Omega_R \subset \T^1(X)$ defined in \cref{compact set} according to the following lemma. We include the simple proof which makes clear the source of the constant $C_{\Delta_0}$ in \cref{compact set}. The proof uses the estimate from \cref{lem:CylinderEstimate} (see \cite[Lemma 7.1]{LPS25}).

\begin{lemma}
\label{lem:goodpartitioncusp}
Let $t > 0$ and $R > 0$. For all $x \in \Omega(t, R)$, we have $u(x)a_t\in \Omega_R$.
\end{lemma}

\begin{proposition}[LDP; {\cite[Proposition 7.3]{LPS25}}]
\label{prop:LDP}
There exist $R_0>0$ and $\kappa \in (0, 1)$ such that the following holds.
For all $m,n \in \N$ and $t > 0$, we have
\begin{align*}
\nu\{x\in\Lambda_+: \#\{j \in \N: j \leq n, T^{jm}(x) \in \Omega(t,R_0) \} < \kappa n \} \leq  e^{-\kappa n}.
\end{align*}
\end{proposition}

\section{Proof of \cref{thm:UniformSpectralBoundII} via Dolgopyat's method}
\label{sec:Dolgopyat'sMethod}
In this section we first reduce \cref{thm:UniformSpectralBoundII} to \cref{thm:FrameFlowUniformDolgopyat}. The latter is a \emph{congruence} version of \cite[Theorem 8.2]{LPS25}. The full proof of \cref{thm:FrameFlowUniformDolgopyat} is nearly a verbatim repetition of the proof in loc. cit. using an adaptation of Dolgopyat's method via a combination of the works of Stoyanov \cite{Sto11}, Sarkar--Winter \cite{SW21}, and Tsujii--Zhang \cite{TZ23}. Thus, we opt to remark on the proof of \cref{thm:FrameFlowUniformDolgopyat} and focus on the key differences.

We fix some notation. Let $(V, \|\bigcdot\|)$ be any normed vector space over $\R$ or $\C$. Let $d$ be any distance function on $\Delta_0$; in particular, $d = d_{\mathrm{E}}$ or $d = D$. Let $H: \Lambda_+\to V$ be any function. Following Dolgopyat \cite{Dol98}, define a family of equivalent norms
\begin{equation*}
\|H\|_{1, b}=\|H\|_{\infty}+\frac{1}{\max\{1,|b|\}} \Lip_d(H), \qquad b\in \R.
\end{equation*}
The Lipschitz norm is then simply $\|H\|_{\Lip(d)} = \|H\|_{1, 1}$. Recall the measure $\nu$ from \cref{subsec:expanding map}. If $H$ is measurable, denote
\begin{align*}
\|H\|_2 = \biggl(\int_{\Lambda_+} \|H(x)\|^2 \, d\nu(x)\biggr)^{\frac{1}{2}}.
\end{align*}
We also define the function $\|H\|: \Lambda_+ \to \R$ by
\begin{align*}
\|H\|(x) = \|H(x)\| \qquad \text{for all $x \in \Lambda_+$}.
\end{align*}

We have the following theorem which uses the sets $\Omega(\log\|\rho_b\|,R_0) \subset \Delta_0$ for all $(b, \rho) \in \widehat{M}_0(b_0)$ as defined in \cref{eqn: union of sandwiched diameter cylinders}.

\begin{theorem}
\label{thm:FrameFlowUniformDolgopyat}
There exist $\eta \in (0, 1)$, $\kappa \in (0, 1)$, $a_0 > 0$, $b_0>0$, $m \in \N$, and a continuous function $\zeta: [-a_0, a_0] \to \R$ with $\zeta(0) = 1$ such that the following holds. For all $\xi=a+ib \in \C$ with $|a| < a_0$, if $(b, \rho) \in \widehat{M}_0(b_0)$, then for all nontrivial ideals $\mathfrak{q} \subset \calO_\K$ and $H \in \Lip\bigl(\Lambda_+, L^2(F_{\mathfrak{q}}) \otimes V_\rho^{\oplus \dim(\rho)}\bigr)$, there exist a sequence of positive functions $\{h_n\}_{n = 0}^\infty \subset \Lip_D(\Lambda_+, \R)$ with $h_0 = \|H\|_{1, \|\rho_b\|}$ and a sequence of closed subsets $\{\Omega_n\}_{n = 1}^\infty$ of $\Omega(\log\|\rho_b\|,R_0)$ such that:
\begin{enumerate}
\item for all $n \in \N$,  we have
\begin{align*}
\bigl\|\mathcal{M}_{\xi, \mathfrak{q}, \rho}^{nm}(H)(x)\bigr\|_2 \leq h_n(x) \qquad \text{for all $x \in \Lambda_+$};
\end{align*}
\item for all $n \in \N$, we have
\begin{align*}
h_n^2(x) \leq
\begin{cases}
\eta \mathcal{L}_0^m(h_{n - 1}^2)(x), & x \in \Omega_n, \\
\zeta(a)\mathcal{L}_0^m(h_{n - 1}^2)(x), & x \in \Lambda_+ \smallsetminus \Omega_n;
\end{cases}
\end{align*}
\item for all $n \in \N$, we have
\begin{align*}
\nu\{x \in \Lambda_+: \#\{j \in \N: j \leq n, T^{jm}x \in \Omega_j\} < \kappa n\} < 2e^{-\kappa n}.
\end{align*}
\end{enumerate}
\end{theorem}

We include the following nonstandard proof as in \cite{LPS25}. It is inspired from \cite[Proposition 3.15]{TZ23} but uses the language of transfer operators.

\begin{proof}[Proof that \cref{thm:FrameFlowUniformDolgopyat} implies \cref{thm:UniformSpectralBoundII}]
Denote by $\tilde{\eta}$, $\kappa$, $\tilde{a}_0$, $b_0$, $m$, and $\zeta$ the constants and function provided by \cref{thm:FrameFlowUniformDolgopyat}. Fix $\eta = \frac{1}{4}\min\{\kappa, -\kappa\log(\tilde{\eta})\}$ and $a_0 \in (0, \tilde{a}_0)$ such that $\sup_{a \in [-a_0, a_0]} |\log(\zeta(a))| \leq \eta$. Fix
\begin{align*}
C_0 &= \sup_{|a| \leq a_0, \{0\} \subsetneq \mathfrak{q} \subset \calO_\K, \rho \in \widehat{M}} \bigl\|\mathcal{M}_{\xi, \mathfrak{q}, \rho}\bigr\|_{\mathrm{op}}^{2m} \leq \sup_{|a| \leq a_0} \bigl\|{\mathcal{L}}_\xi\bigr\|_{\mathrm{op}}^{2m},\\
C &= \sqrt{3C_0},
\end{align*}
viewing the transfer operators as operators on $L^2\bigl(\Lambda_+, L^2(F_\mathfrak{q}) \otimes V_\rho^{\oplus \dim(\rho)}\bigr)$ and $L^2(\Lambda_+, \R)$ respectively. Let $\xi=a+ib \in \C$ with $|a| < a_0$. Suppose $(b, \rho) \in \widehat{M}_0(b_0)$. Let $k \in \N$ and write $k = nm + l$ for some integers $n \in \Z_{\geq 0}$ and $0 \leq l < m$.
Let $H \in \Lip\bigl(\Lambda_+, L^2(F_{\mathfrak{q}}) \otimes V_\rho^{\oplus \dim(\rho)}\bigr)$. We then obtain corresponding sequences $\{h_n\}_{n = 0}^\infty$ and $\{\Omega_n\}_{n = 1}^\infty$ provided by \cref{thm:FrameFlowUniformDolgopyat}.

We define the sequence of
functions $\{G_j: \Lambda_+ \to \R\}_{j = 0}^\infty$ recursively by
\begin{align*}
G_0 &= h_0^2, \\
G_j(x) &=
\begin{cases}
\tilde{\eta} G_{j - 1}(x), & x \in T^{-jm}(\Omega_j), \\
\zeta(a)G_{j - 1}(x), & x \in \Lambda_+ \smallsetminus T^{-jm}(\Omega_j),
\end{cases}
\qquad
\text{for all $j \in \N$}.
\end{align*}
Let us show by induction that the above sequence of functions satisfy
\begin{align}
\label{eqn:TransferOperatorOnG_j}
\calL_0^{jm}(G_j) \geq h_j^2 \qquad \text{for all $j \in \Z_{\geq 0}$}.
\end{align}
The base case $j = 0$ is trivial. Now let $j \in \N$ and assume \cref{eqn:TransferOperatorOnG_j} holds for $j - 1$. From definitions, for all $\gamma \in \gen^{(m)}$, we have
\begin{align}
\label{eqn:TransferOperatorOnG_j*}
\mathcal{L}_0^{(j - 1)m}(G_j)(\gamma x) =
\begin{cases}
\tilde{\eta} \mathcal{L}_0^{(j - 1)m}(G_{j - 1})(\gamma x), & \text{for all $x \in \Omega_j$} \\
\zeta(a)\mathcal{L}_0^{(j - 1)m}(G_{j - 1})(\gamma x), & \text{for all $x \in \Lambda_+ \smallsetminus \Omega_j$}.
\end{cases}
\end{align}
For all $x \in \Omega_j$, we use \cref{eqn:TransferOperatorOnG_j*}, the induction hypothesis, and Property~(2) in \cref{thm:FrameFlowUniformDolgopyat}, to get
\begin{align*}
\mathcal{L}_0^{jm}(G_j)(x) &= \sum_{\gamma \in \gen^{(m)}} e^{\FRet_m^{(0)}(\gamma x)} \mathcal{L}_0^{(j - 1)m}(G_j)(\gamma x) \\
&= \tilde{\eta} \sum_{\gamma \in \gen^{(m)}}  e^{\FRet_m^{(0)}(\gamma x)} \mathcal{L}_0^{(j - 1)m}(G_{j - 1})(\gamma x) \\
&\geq \tilde{\eta} \sum_{\gamma \in \gen^{(m)}} e^{\FRet_m^{(0)}(\gamma x)} h_{j - 1}^2(\gamma x) = \tilde{\eta} \mathcal{L}_0^m(h_{j - 1}^2)(x)\geq h_j^2(x).
\end{align*}
For all $x \in \Lambda_+ - \Omega_j$, we do a similar calculation to get
\begin{align*}
\mathcal{L}_0^{jm}(G_j)(x) &= \sum_{\gamma \in \gen^{(m)}} e^{\FRet_m^{(0)}(\gamma x)} \mathcal{L}_0^{(j - 1)m}(G_j)(\gamma x) \\
&=\zeta(a) \sum_{\gamma \in \gen^{(m)}} e^{\FRet_m^{(0)}(\gamma x)} \mathcal{L}_0^{(j - 1)m}(G_{j - 1})(\gamma x) \\
&\geq \zeta(a)\sum_{\gamma \in \gen^{(m)}} e^{\FRet_m^{(0)}(\gamma x)} h_{j - 1}^2(\gamma x) =\zeta(a) \mathcal{L}_0^m(h_{j - 1}^2)(x)\geq h_j^2(x).
\end{align*}
This establishes \cref{eqn:TransferOperatorOnG_j}.

Now, we use Property~(1) in \cref{thm:FrameFlowUniformDolgopyat}, \cref{eqn:TransferOperatorOnG_j}, and $\mathcal{L}_0^*(\nu) = \nu$, to get
\begin{align*}
\bigl\|\mathcal{M}_{\xi, \mathfrak{q}, \rho}^k(H)\bigr\|_2^2 &\leq C_0\bigl\|\mathcal{M}_{\xi, \mathfrak{q}, \rho}^{nm}(H)\bigr\|_2^2 \leq C_0\|h_n\|_2^2 \\
&\leq C_0\int_{\Lambda_+}\mathcal{L}_0^{nm}(G_n) \, d\nu \\
&= C_0\int_{\Lambda_+} G_n \, d\nu.
\end{align*}
To finish the proof, we decompose the integral over $\Lambda_+$ into integrals over
\begin{align*}
(\Lambda_+)_{\mathrm{main}} &\coloneq \{y \in \Lambda_+: \#\{j \in \N: j \leq n, T^{jm}y \in \Omega_j\} \geq \kappa n\}, \\
(\Lambda_+)_{\mathrm{err}} &\coloneq \{y \in \Lambda_+: \#\{j \in \N: j \leq n, T^{jm}y \in \Omega_j\} < \kappa n\}.
\end{align*}
Using $\zeta(a) > \tilde{\eta}$, for all $x \in (\Lambda_+)_{\mathrm{main}}$, we have the bound
\begin{align*}
G_n(x)\leq \zeta(a)^{n - \lceil\kappa n\rceil}\tilde{\eta}^{\lceil\kappa n\rceil}G_0(x) \leq e^{\eta n}\tilde{\eta}^{\kappa n}h_0^2(x)
\end{align*}
while for all $x \in (\Lambda_+)_{\mathrm{err}}$, we have the trivial bound
\begin{align*}
G_n(x) \leq \zeta(a)^nG_0(x) \leq e^{\eta n} h_0^2(x).
\end{align*} 
Thus, using Property~(3) in \cref{thm:FrameFlowUniformDolgopyat}, we have
\begin{align*}
\bigl\|\mathcal{M}_{\xi, \mathfrak{q}, \rho}^k(H)\bigr\|_2^2 &\leq C_0\int_{\Lambda_+} G_n \, d\nu = C_0\int_{(\Lambda_+)_{\mathrm{main}}} G_n \, d\nu + C_0\int_{(\Lambda_+)_{\mathrm{err}}} G_n \, d\nu \\
&\leq C_0 e^{\eta n} ({\tilde{\eta}}^{\kappa n} \cdot (1 - 2e^{-\kappa n}) + 2e^{-\kappa n}) \|h_0\|_\infty^2 \\
&\leq C^2 e^{-2\eta n} \|H\|_{1, \|\rho_b\|}^2.
\end{align*}
\end{proof}

Let us remark on the proof of \cref{thm:FrameFlowUniformDolgopyat}. LDP from \cref{prop:LDP} is used together with a stochastic dominance argument to directly deduce Property~(3) in \cref{thm:FrameFlowUniformDolgopyat} (see \cite[\S 8.3]{LPS25}). Here, the statement and the argument does not involve the cocycle at all. LNIC from \cref{prop:LNIC} and NCP \cref{prop:NCP} are required to prove a generalization of \cite[Lemma 8.5]{LPS25}, provided below, which is eventually used to prove Property~(1) in \cref{thm:FrameFlowUniformDolgopyat}. Here, the cocycle is involved. Let us now focus on the key differences that arise when incorporating the congruence aspect. We first list the required changes and then provide more details in the subsequent subsections for the parts where the cocycle plays a critical role.

\begin{enumerate}
\item Almost all statements need to be made uniform with respect to nontrivial ideals $\mathfrak{q} \subset \calO_\K$.
\item The transfer operators with holonomy need to be replaced by \emph{congruence} transfer operators with holonomy. Consequently, terms in some some proofs which come from the definition of transfer operators need to include an action by the cocycle.
\item Many proofs are unaffected by including the cocycle due to unitarity of both the cocycle and the irreducible representations of $M$.
\end{enumerate}

\subsection{Changes required for \cite[Lemma 8.3]{LPS25}}
This is a Lasota--Yorke type lemma. It is required for the proof of Properties~(1) and (2) of \cref{thm:FrameFlowUniformDolgopyat}. There are differences only in the proof of Property~(2) of \cite[Lemma 8.3]{LPS25}. We use the same notation and begin with the calculation
\begin{align*}
&\bigl\|\mathcal{M}_{\xi, \mathfrak{q}, \rho}^k(H)(x) - \mathcal{M}_{\xi, \mathfrak{q}, \rho}^k(H)(x')\bigr\|_{\mathfrak{q}, \rho} \\
\leq{}&\sum_{\gamma \in \gen^{(k)}} \bigl\|e^{\FRet_k^{(a)}(\gamma x)} \bigl(\gamma^{-1} \otimes \rho_b(\GHol^k(\gamma x)^{-1})\bigr) H(\gamma x) \\
{}&- e^{\FRet_k^{(a)}(\gamma x')} \bigl(\gamma^{-1} \otimes \rho_b(\GHol^k(\gamma x')^{-1})\bigr) H(\gamma x')\bigr\|_{\mathfrak{q}, \rho}  \\
\leq{}&\sum_{\gamma \in \gen^{(k)}} \Bigl(\left|1 - e^{\FRet_k^{(a)}(\gamma x') - \FRet_k^{(a)}(\gamma x)}\right|e^{\FRet_k^{(a)}(\gamma x)} \bigl\|\bigl(\gamma^{-1} \otimes \rho_b(\GHol^k(\gamma x)^{-1})\bigr) H(\gamma x)\bigr\|_{\mathfrak{q}, \rho} \\
&{}+ e^{\FRet_k^{(a)}(\gamma x')} \bigl\|\bigl(\gamma^{-1} \otimes \rho_b(\GHol^k(\gamma x)^{-1}) - \gamma^{-1} \otimes \rho_b(\GHol^k(\gamma x')^{-1})\bigr) H(\gamma x)\bigr\|_{\mathfrak{q}, \rho} \\
&{}+ e^{\FRet_k^{(a)}(\gamma x')} \bigl\|\bigl(\gamma^{-1} \otimes \rho_b(\GHol^k(\gamma x')^{-1})\bigr) (H(\gamma x) - H(\gamma x'))\bigr\|_{\mathfrak{q}, \rho}\Bigr).
\end{align*}
Then, we use unitarity to obtain:
\begin{itemize}
\item $\bigl\|\bigl(\gamma^{-1} \otimes \rho_b(\GHol^k(\gamma x)^{-1})\bigr) H(\gamma x)\bigr\|_{\mathfrak{q}, \rho} = \|H(\gamma x)\|_{\mathfrak{q}, \rho}$,
\item $\bigl\|\bigl(\gamma^{-1} \otimes \rho_b(\GHol^k(\gamma x)^{-1}) - \gamma^{-1} \otimes \rho_b(\GHol^k(\gamma x')^{-1})\bigr) H(\gamma x)\bigr\|_{\mathfrak{q}, \rho} = \|(\rho_b(\GHol^k(\gamma x)^{-1}) - \rho_b(\GHol^k(\gamma x')^{-1})) H(\gamma x)\|_{\mathfrak{q}, \rho}$,
\item $\bigl\|\bigl(\gamma^{-1} \otimes \rho_b(\GHol^k(\gamma x')^{-1})\bigr) (H(\gamma x) - H(\gamma x'))\bigr\|_{\mathfrak{q}, \rho} = \|\rho_b(\GHol^k(\gamma x')^{-1}) (H(\gamma x) - H(\gamma x'))\|_{\mathfrak{q}, \rho}$.
\end{itemize}
Thus, the cocycle is removed, and now the rest of the proof is as in the proof of \cite[Lemma 8.3]{LPS25} (though with a different norm due to the tensor factor $L^2(F_\mathfrak{q})$).

\subsection{Changes required for \cite[Lemma 8.5]{LPS25}}
\label{subsec:ChangesRequiredForLPSLemma8.5}
This lemma originally combines LNIC, NCP, and \cite[Lemma 4.2]{LPS25} (which is \cite[Lemma 4.4]{SW21}) to find appropriate points where large derivatives are realized. Below is the adaptation to the congruence setting with proof. The only change required in the proof is to use \cref{lem:maActionLowerBound} (which is \cite[Lemma 4.4]{Sar22a}) instead of \cite[Lemma 4.2]{LPS25}.

\begin{lemma}
\label{lem:PartnerPointInZariskiDenseLimitSetForBPBound}
For all $(b, \rho) \in \widehat{M}_0(b_0)$, nontrivial ideals $\mathfrak{q} \subset \calO_\K$, $x \in \Lambda_\Gamma \cap \bigl(\Delta_0 \smallsetminus \overline{B_{s_1}^{\mathrm{E}}(\partial \Delta_0)}\bigr)$ with $u(x) a_{\log\|\rho_b\|} \in \Omega_{R_0}$, and $\omega \in L^2(F_{\mathfrak{q}}) \otimes V_\rho^{\oplus \dim(\rho)}$ with $\|\omega\|_{\mathfrak{q}, \rho} = 1$, there exist $1 \leq j \leq \jj$ and $y \in \Lambda_\Gamma \cap \bigl(B_{s_1}^{\mathrm{E}}(x) \smallsetminus B_{s_2}^{\mathrm{E}}(x)\bigr)$ such that
\begin{equation*}
\left\|d\rho_{b, \mathfrak{q}}\left((d\BP_{j, x})_x(z)\right)(\omega)\right\|_{\mathfrak{q}, \rho} \geq 14\delta_1\epsilon_1
\end{equation*}
where $s_1 = \frac{\epsilon_1}{\|\rho_b\|}$, $s_2 = \frac{\delta_1\epsilon_1}{\|\rho_b\|}$, and $z = (x, y - x) \in \T_x(\Delta_0)$.
\end{lemma}
\begin{proof}
Let $(b, \rho)$, $\mathfrak{q}$, $x$, $\omega$, $s_1$, and $s_2$ be as in the lemma. Define the linear maps
\begin{align*}
T_j &= (d\BP_{j, x})_x: \T_x(\Delta_0) \to \LieA \oplus \LieM \qquad \text{for all $1 \leq j \leq j_0$}, \\
S &= d\rho_{b, \mathfrak{q}}(\bigcdot)(\omega): \LieA \oplus \LieM \to L^2(F_{\mathfrak{q}}) \otimes V_\rho^{\oplus \dim(\rho)}.
\end{align*}
\Cref{lem:maActionLowerBound} gives $\|S^*\|_{\mathrm{op}} = \|S\|_{\mathrm{op}} \geq \varepsilon_1\|\rho_b\|$. Thus, there exists $\omega_0 \in L^2(F_{\mathfrak{q}}) \otimes V_\rho^{\oplus \dim(\rho)}$ with $\|\omega_0\|_{\mathfrak{q}, \rho} = 1$ such that $\|S^*(\omega_0)\| \geq \varepsilon_1\|\rho_b\|$. By LNIC in \cref{prop:LNIC}, there exists $1 \leq j \leq j_0$ such that $\bigl\|T_j^*(S^*(\omega_0))\bigr\| \geq \varepsilon_2\varepsilon_1\|\rho_b\|$. Using NCP in \cref{prop:NCP} and $\delta_1 < \varepsilon_3$, there exists $y \in \Lambda_\Gamma \cap \bigl(B_{s_1}^{\mathrm{E}}(x) \smallsetminus B_{s_2}^{\mathrm{E}}(x)\bigr)$ such that writing $z = (x, y - x) \in \T_x(\Delta_0)$, we have
\begin{align*}
\bigl\langle z, T_j^*(S^*(\omega_0))\bigr\rangle \geq \frac{\epsilon_1}{\|\rho_b\|} \cdot \varepsilon_3 \cdot \bigl\|T_j^*(S^*(\omega_0))\bigr\| \geq \epsilon_1\varepsilon_1\varepsilon_2\varepsilon_3 \geq 14\delta_1\epsilon_1.
\end{align*}
Applying the Cauchy--Schwarz inequality finishes the proof.
\end{proof}

\Cref{lem:PartnerPointInZariskiDenseLimitSetForBPBound} can then be used to obtain the congruence version of \cite[Proposition 8.8]{LPS25}.

\subsection{Changes required for \cite[Lemma 8.19]{LPS25}}
This is a key lemma where sufficient cancellations are obtained from the oscillating summands of the congruence transfer operators with holonomy. The lemma is used to eventually prove Property~(1) of \cref{thm:FrameFlowUniformDolgopyat} which says that under iterations of the congruence transfer operators with holonomy, the resulting sequence of functions are dominated by a corresponding sequence of functions which is contracting on large subsets of its domain $\Lambda_+$ as the index goes to $+\infty$.

First, we make the following definitions in the congruence setting. For all $\xi=a+ib \in \mathbb C$ with $|a| < a_0'$, if $(b, \rho) \in \widehat{M}_0(b_0)$, then for all nontrivial ideals $\mathfrak{q} \subset \calO_\K$, $(H, h) \in \Lip\bigl(\Lambda_+, L^2(F_\mathfrak{q}) \otimes V_\rho^{\oplus \dim(\rho)}\bigr) \times \Lip_D(\Lambda_+, \R)$, and $1 \leq j \leq\jj$, we define the functions $\chi_{j, 1}^{[\xi, \rho, H, h]}, \chi_{j, 2}^{[\xi, \rho, H, h]}: \Lambda_+ \to \R$ by
\begin{align*}
\chi_{j, 1}^{[\xi, \rho, H, h]}(x) ={}& \bigl\|e^{\FRet_m^{(a)}(\alpha_0x)} \bigl(\alpha_0 \otimes \rho_b(\GHol^m(\alpha_0x)^{-1})\bigr) H(\alpha_0x) \\
{}&+ e^{\FRet_m^{(a)}(\alpha_jx)} \bigl(\alpha_j \otimes \rho_b(\GHol^m(\alpha_jx)^{-1})\bigr)H(\alpha_jx)\bigr\|_{\mathfrak{q}, \rho} \\
{}&\cdot \bigl((1 - \tau)e^{\FRet_m^{(a)}(\alpha_0x)}h(\alpha_0x) + e^{\FRet_m^{(a)}(\alpha_jx)}h(\alpha_jx)\bigr)^{-1}, \\
\chi_{j, 2}^{[\xi, \rho, H, h]}(x) ={}& \bigl\|e^{\FRet_m^{(a)}(\alpha_0x)} \bigl(\alpha_0 \otimes \rho_b(\GHol^m(\alpha_0x)^{-1})\bigr)H(\alpha_0x) \\
{}&+ e^{\FRet_m^{(a)}(\alpha_jx)} \bigl(\alpha_j \otimes \rho_b(\GHol^m(\alpha_jx)^{-1})\bigr)H(\alpha_jx)\bigr\|_{\mathfrak{q}, \rho} \\
{}&\cdot \bigl(e^{\FRet_m^{(a)}(\alpha_0x)}h(\alpha_0x) + (1 - \tau)e^{\FRet_m^{(a)}(\alpha_jx)}h(\alpha_jx)\bigr)^{-1}
\end{align*}
for all $x \in \Lambda_+$. These functions encode that when sufficient cancellations occur among the vectors in the numerator, then the resulting values are at most $1$. The second subscript indicates for which of the two inverse branches $\alpha_0$ or $\alpha_j$ the cancellations occur.

Now, keeping the same notation, we have the following changes to the proof of \cite[Lemma 8.19]{LPS25}. First, for all $x\in \Lambda_+$ and $\ell\in \{0,j\}$, we define
\begin{align*}
V_\ell(x) &\coloneq e^{\FRet_m^{(a)}(\alpha_\ell x)} \bigl(\alpha_\ell \otimes \rho_b(\phi_\ell(x)^{-1})\bigr)H(\alpha_\ell x), \\
\hat{V}_\ell(x) &\coloneq \frac{V_\ell(x)}{\|V_\ell(x)\|_{\mathfrak{q}, \rho}} = \bigl(\alpha_\ell \otimes \rho_b(\phi_\ell(x)^{-1})\bigr)\omega_\ell(x),
\end{align*}
where $\omega_\ell = \frac{H(\alpha_\ell \bigcdot)}{\|H(\alpha_\ell \bigcdot)\|_{\mathfrak{q}, \rho}}$. Using the fact that $\omega_\ell$ is Lipschitz on $(\mathtt{J}_1^H \sqcup \mathtt{J}_2^H)\cap\Lambda_+$ with Lipschitz constant $\frac{\delta_1\epsilon_1}{4}$, we have
\begin{align*}
&\big\|\hat{V}_0(x_2) - \hat{V}_j(x_2)\big\|_{\mathfrak{q}, \rho} \\
={}&\bigl\|\bigl(\alpha_0 \otimes \rho_b(\phi_0(x_2)^{-1})\bigr)\omega_0(x_2) - \bigl(\alpha_j \otimes \rho_b(\phi_j(x_2)^{-1})\bigr)\omega_j(x_2)\bigr\|_{\mathfrak{q}, \rho} \\
={}&\bigl\|\bigl(\alpha_0 \otimes \rho_b(\phi_j(x_2)\phi_0(x_2)^{-1})\bigr)\omega_0(x_2) - \bigl(\alpha_j \otimes \Id\bigr)\omega_j(x_2)\bigr\|_{\mathfrak{q}, \rho} \\
\geq{}&\bigl\|\bigl(\alpha_0 \otimes \rho_b(\phi_j(x_2)\phi_0(x_2)^{-1})\bigr)\omega_0(x_1) - \bigl(\alpha_j \otimes \Id\bigr)\omega_j(x_1)\bigr\|_{\mathfrak{q}, \rho} \\
{}&- \bigl\|\bigl(\alpha_0 \otimes \rho_b(\phi_j(x_2)\phi_0(x_2)^{-1})\bigr)\omega_0(x_2) - \bigl(\alpha_0 \otimes \rho_b(\phi_j(x_2)\phi_0(x_2)^{-1})\bigr)\omega_0(x_1)\bigr\|_{\mathfrak{q}, \rho} \\
{}&- \bigl\|\bigl(\alpha_j \otimes \Id\bigr)\omega_j(x_2) - \bigl(\alpha_j \otimes \Id\bigr)\omega_j(x_1)\bigr\|_{\mathfrak{q}, \rho} \\
={}&\bigl\|\bigl(\alpha_0 \otimes \rho_b(\phi_j(x_2)\phi_0(x_2)^{-1})\bigr)\omega_0(x_1) - \bigl(\alpha_j \otimes \Id\bigr)\omega_j(x_1)\bigr\|_{\mathfrak{q}, \rho} \\
{}&- \|\omega_0(x_2) - \omega_0(x_1)\|_{\mathfrak{q}, \rho} - \|\omega_j(x_2) - \omega_j(x_1)\|_{\mathfrak{q}, \rho} \\
\geq{}&\bigl\|\bigl(\alpha_0 \otimes \rho_b(\phi_j(x_2)\phi_0(x_2)^{-1})\bigr)\omega_0(x_1) - \bigl(\alpha_0 \otimes \rho_b(\phi_j(x_1)\phi_0(x_1)^{-1})\bigr)\omega_0(x_1)\bigr\|_{\mathfrak{q}, \rho} \\
{}&- \bigl\|\bigl(\alpha_0 \otimes \rho_b(\phi_j(x_1)\phi_0(x_1)^{-1})\bigr)\omega_0(x_1) - \bigl(\alpha_j \otimes \Id\bigr)\omega_j(x_1)\bigr\|_{\mathfrak{q}, \rho} - \delta_1\epsilon_1 \\
={}&\bigl\|\bigl(\Id \otimes \rho_b(\phi_0(x_1)^{-1})\bigr)\omega_0(x_1) \\
&{}- \bigl(\Id \otimes \rho_b(\phi_0(x_1)^{-1}\phi_0(x_2)\phi_j(x_2)^{-1}\phi_j(x_1)\phi_0(x_1)^{-1})\bigr)\omega_0(x_1)\bigr\|_{\mathfrak{q}, \rho} \\
{}&- \bigl\|\bigl(\alpha_0 \otimes \rho_b(\phi_0(x_1)^{-1})\bigr)\omega_0(x_1) - \bigl(\alpha_j \otimes \rho_b(\phi_j(x_1)^{-1})\bigr)\omega_j(x_1)\bigr\|_{\mathfrak{q}, \rho} - \delta_1\epsilon_1 \\
\geq{}&\bigl\|\rho_{b, \mathfrak{q}}(\phi_0(x_1)^{-1})\omega_0(x_1) - \rho_{b, \mathfrak{q}}(\BP_j(x_1, x_2))\rho_{b, \mathfrak{q}}(\phi_0(x_1)^{-1})\omega_0(x_1)\bigr\|_{\mathfrak{q}, \rho} \\
&{}- \big\|\hat{V}_0(x_1) - \hat{V}_j(x_1)\big\|_{\mathfrak{q}, \rho} - \delta_1\epsilon_1.
\end{align*}
By the congruence version of \cite[Proposition 8.8]{LPS25} (see \cref{subsec:ChangesRequiredForLPSLemma8.5}) for the unit vector $\omega = \rho_b(\phi_0(x_1)^{-1})\omega_0(x_1)$, we have
\begin{align*}
\|d\rho_{b, \mathfrak{q}}((d\BP_{j, x_1})_{x_1}(z))(\omega)\|_{\mathfrak{q}, \rho} \geq 7\delta_1\epsilon_1
\end{align*}
where $z = (x_1, x_2 - x_1) \in \T_{x_1}(\Delta_0)$. Let $Z = d(\BP_{j, x_1})_{x_1}(z)$. Then, we resume our calculation above and bound the first term of the last line to get
\begin{align*}
&\|\omega - \rho_{b, \mathfrak{q}}(\BP_j(x_1, x_2))(\omega)\|_{\mathfrak{q}, \rho} \\
\geq{}&\|\omega - \rho_{b, \mathfrak{q}}(\exp(Z))(\omega)\|_{\mathfrak{q}, \rho} - \|\rho_{b, \mathfrak{q}}(\exp(Z))(\omega) - \rho_{b, \mathfrak{q}}(\BP_j(x_1, x_2))(\omega)\|_{\mathfrak{q}, \rho} \\
\geq{}&\|\omega - \exp(d\rho_{b, \mathfrak{q}}(Z))(\omega)\|_{\mathfrak{q}, \rho} - \|\rho_{b, \mathfrak{q}}\| \cdot d_{AM}(\exp(Z), \BP_j(x_1, x_2)) \\
\geq{}&\|d\rho_{b, \mathfrak{q}}(Z)(\omega)\|_{\mathfrak{q}, \rho} - \|\rho_{b, \mathfrak{q}}\|^2 \|Z\|^2 - \|\rho_{b, \mathfrak{q}}\| \cdot d_{AM}(\exp(Z), \BP_j(x_1, x_2)) \\
\geq{}&\|d\rho_{b, \mathfrak{q}}(Z)(\omega)\|_{\mathfrak{q}, \rho} - \|\rho_b\|^2 C_{\BP}^2 d(x_1, x_2)^2 - C_{\exp, \BP} \cdot \|\rho_b\| \cdot d(x_1, x_2)^2 \quad\text{(\cref{lem:ComparingExpWithBP})} \\
\geq{}&7\delta_1\epsilon_1 - \delta_1\epsilon_1 - \delta_1\epsilon_1 \geq 5\delta_1\epsilon_1.
\end{align*}
Hence, we have
\begin{align*}
\big\|\hat{V}_0(x_1) - \hat{V}_j(x_1)\big\|_{\mathfrak{q}, \rho} + \big\|\hat{V}_0(x_2) - \hat{V}_j(x_2)\big\|_{\mathfrak{q}, \rho} \geq 4\delta_1\epsilon_1.
\end{align*}
The rest of the proof is as in \cite[Lemma 8.19]{LPS25} using unitarity of the cocycle.

%\nocite{*}
\bibliographystyle{amsalpha}
\bibliography{References}

\end{document}